\documentclass[10pt,reqno]{amsart}
\usepackage{amsthm,amsmath,amssymb,pdfsync,enumerate,color,verbatim,color,graphics}
\usepackage[usenames,dvipsnames]{xcolor}
\usepackage[active]{srcltx}
\RequirePackage[colorlinks,citecolor=blue,urlcolor=blue]{hyperref}
\allowdisplaybreaks
\theoremstyle{plain}
\newtheorem{ttheorem}{Theorem}
\newtheorem{theorem}{Theorem}[section]
\newtheorem{corollary}[theorem]{Corollary}
\newtheorem{lemma}[theorem]{Lemma}
\newtheorem{proposition}[theorem]{Proposition}
\newtheorem{definition}[theorem]{Definition}
\newtheorem{setting}[theorem]{Setting}
 
\theoremstyle{remark}
\newtheorem{remark}[theorem]{Remark}
\newtheorem{example}[theorem]{Example}

\def\R{{\mathbb R}}
\def\N{{\mathbb N}}
\def\Q{{\mathbb Q}}
\def\E{{\mathbb E}}
\def\P{{\mathbb P}}
\def\D{{\mathbb D}}
\def\F{{\mathbb F}}
\newcommand{\n}{\Vert}
\newcommand{\calB}{\mathcal B}
\newcommand{\calC}{\mathcal C}
\newcommand{\calF}{\mathcal F}
\newcommand{\calG}{\mathcal G}
\newcommand{\calH}{\mathcal H}
\newcommand{\calL}{\mathcal L}
\newcommand{\calN}{\mathcal N}
\newcommand{\calP}{\mathcal P}
\newcommand{\calQ}{\mathcal Q}
\newcommand{\calA}{\mathcal A}
\newcommand{\calS}{\mathcal S}
\newcommand{\eps}{\varepsilon}
\newcommand{\umd}{{\rm UMD}} 
\newcommand{\one}{{\rm 1}}
\DeclareMathOperator*{\esssup}{ess\,sup}
\newcommand{\inv}[1]{\frac{1}{#1}}
\newcommand{\tinv}[1]{\tfrac{1}{#1}}
\newcommand{\maxsym}{\vee}
\newcommand{\minsym}{\wedge}
\newcommand{\sptext}[3]{\hspace{#1 em}\mbox{#2}\hspace{#3 em}}
\newcommand{\equa}{\begin{eqnarray*}}
\newcommand{\tion}{\end{eqnarray*}}
\newcommand{\law}{{\rm law}}
\newcommand{\adapt}{{\mathcal{A}}}
\newcommand{\bmo}{{\rm BMO}}

\begin{document}

\title[Decoupling]{On decoupling in Banach spaces}

\author{Sonja Cox}
\address{Korteweg-de Vries Institute for Mathematics,
         University of Amsterdam,
         Postbus 94248,
         NL-1090 GE Amsterdam, Netherlands}
\email{s.g.cox@uva.nl}

\author{Stefan Geiss}
\address{Department of Mathematics and Statistics,
        University of Jyv\"askyl\"a,
        P.O. Box 35,
        FIN-40014 University of Jyv\"askyl\"a,
        Finland}
\email{stefan.geiss@jyu.fi}

\begin{abstract}
We consider decoupling inequalities for random variables taking values 
in a Banach space $X$. We restrict the class of distributions that appear
as conditional distributions while decoupling and show that each adapted process
can be approximated by a Haar type expansion in which only the same conditional 
distributions appear. Moreover, we show that in our framework a progressive 
enlargement of the underlying filtration does not effect the decoupling 
properties (e.g., the constants involved). 
As special case we  deal with one-sided moment inequalities when decoupling dyadic
(i.e., Paley-Walsh) martingales. We establish the decoupling constant in $\ell^\infty_K$.
As an example of an application, we demonstrate that 
Burkholder-Davis-Gundy type inequalities for stochastic integrals 
of $X$-valued processes can be obtained from decoupling inequalities 
for $X$-valued dyadic martingales.
\end{abstract}

\maketitle

\setcounter{tocdepth}{1}
\tableofcontents


\section{Introduction}
\label{sec:intro}

The UMD-property is crucial in harmonic and stochastic analysis 
in Banach spaces, see e.g.\ \cite{HytonenEtAl:2016, HytonenEtAl:2017}. 
A Banach space $X$ is said to satisfy the UMD-property if there exists a 
constant $ c_\eqref{eqn:umd}\ge 1$ such that for every martingale difference sequence 
$(d_n)_{n=1}^N$ with values in $X$ one has that 
\begin{equation}\label{eqn:umd}
       \frac{1}{c_\eqref{eqn:umd}} \left \| \sum_{n=1} d_n \right \|_{\calL^2(\P;X)} 
   \le \left \| \sum_{n=1}^N \theta_n d_n \right \|_{\calL^2(\P;X)}
   \le c_\eqref{eqn:umd} \left \| \sum_{n=1}^N d_n \right \|_{\calL^2(\P;X)}
\end{equation}
for all signs $\theta_n \in \{ -1,1 \}$, i.e., one has \textbf{U}nconditional \textbf{M}artingale \textbf{D}ifferences. 
Here we vary over all stochastic bases $(\Omega,\calF,\P,(\calF_n)_{n=0}^N)$ 
and martingale difference sequences $d_n:\Omega\to X \in \calL^2(X)$ with respect to 
$(\calF_n)_{n=0}^N$.
The transformations 
\[ \sum_{n=1}^N d_n \mapsto \sum_{n=1}^N \theta_n d_n \]
are called UMD-transforms. There is an inherent property of these UMD-trans\-forms: Assume that 
$\{1,\ldots,N\} = \bigcup_{\ell=1}^L I_\ell$, where $I_1,\ldots,I_L$ are non-empty consecutive intervals, so that 
$D_\ell:= \sum_{n\in I_\ell} d_n$ can be interpreted as a martingale difference sequence as well, then any 
UMD-transform 
\[ \sum_{\ell=1}^L D_\ell \mapsto \sum_{\ell=1}^L \eta_\ell  D_\ell \]
for $\eta_1,\ldots,\eta_L\in \{ -1,1 \}$ can be written as a UMD-transform of the finer sequence 
$(d_n)_{n=1}^N$.  Combining this observation  with an approximation argument of Maurey \cite{Maurey:1975}, where
a martingale difference sequence is approximated by a blocked sequence of Haar functions, one obtains: In 
order to verify the UMD-property of the Banach space $X$, it is sufficient to consider $X$-valued 
Haar- or dyadic martingales
(a martingale is \emph{dyadic} if it is adapted to a dyadic 
filtration, dyadic martingales are also known as \emph{Paley-Walsh martingales}).
\smallskip

On the other hand, McConnell \cite{McConnell:1989} proved that the UMD property is equivalent to
the existence of a constant $c_\eqref{eqn:McConnell}$ such that
\begin{equation}\label{eqn:McConnell}
       \frac{1}{c_\eqref{eqn:McConnell}} \left \| \sum_{n=1}^{N} d_n \right \|_{\calL^2(\P;X)} 
   \le \left \| \sum_{n=1}^N e_n \right \|_{\calL^2(\P;X)}
   \le c_\eqref{eqn:McConnell} \left \| \sum_{n=1}^N d_n \right \|_{\calL^2(\P;X)}
\end{equation}
for all $N\in \N$ and $(\calF_n)_{n=1}^{N}$-martingale difference sequences $(d_n)_{n=1}^{N}$ and $(e_n)_{n=1}^{N}$ 
such that $\calL(d_n|\calF_{n-1}) = \calL(e_n|\calF_{n-1})$, i.e.\ $(e_n)_{n=1}^{N}$ and $(d_n)_{n=1}^{N}$ are \emph{tangent}.
\smallskip 

Note that the first and the second inequality in~\eqref{eqn:McConnell} are of identical 
type. This is no
longer true if one imposes the additional condition that $(e_n)_{n\in \N}$ is the
\emph{decoupled} tangent sequence of $(d_n)_{n\in \N}$. 
In this case we refer to first inequality in \eqref{eqn:McConnell} as an {\it upper decoupling inequality},
and we refer to the second inequality as a
{\it lower decoupling inequality}. Analogously, the 
first and second inequality in~\eqref{eqn:umd} are of the same type, and this is no 
longer the case if one replaces 
the (deterministic) $(\theta_n)_{n\in \N}$ in~\eqref{eqn:umd} by a (random) Rademacher sequence $(r_n)_{n\in \N}$:
then one obtains the upper- and lower randomized UMD inequalities studied in~\cite{Garling:1990}.
Examples of spaces that satisfy upper decoupling inequalities but do not have the UMD property are $L^1$ and 
the space of bounded $\sigma$-additive measures (see e.g.~\cite[Examples 4.7 and 4.8]{CoxVeraar:2011}). Moreover, quasi-Banach 
spaces fail to satisfy the UMD property, but may satisfy decoupling 
inequalities, see~\cite[Section 5.1]{CioicaEtAl:2018} and e.g.~\cite[Example 4.7]{CoxVeraar:2011}.
\smallskip 

The notions of tangent and decoupled sequences
(see Definition~\ref{def:decoupled_sequence} below)
were introduced by Kwapie{\'n} and Woyczy{\'n}ski in 
\cite{KwapienWoyczynski:1989,KwapienWoyczynski:1991}, where also applications can be found. 
The decoupled tangent
sequence $(e_n)_{n=1}^N$ of a sequence $(d_n)_{n\in\N}$ (adapted to a filtration $(\calF_n)_{n\in \N}$)
is unique in distribution and replaces parts of the 
dependence structure of $(d_n)_{n=1}^N$ by a sequence of conditionally independent random variables.
Although the definition of decoupling might not be explicit, there are canonical representations of 
a decoupled tangent sequence, see Kwapie{\'n} and Woyczy{\'n}ski \cite{KwapienWoyczynski:1991}
and Montgomery-Smith \cite{MontgomerySmith:1998}.
\smallskip

There are various applications for decoupling in the literature.
The proofs of Burk\-holder \cite{Burkholder:1981a}  and Bourgain \cite{Bourgain:1983} of the equivalence of 
the UMD-property of a Banach space $X$ and the continuity of the $X$-valued Hilbert transform use decoupling 
arguments. For certain applications only one-sided inequalities are needed.
For example, one-sided decoupling inequalities for martingales
and the type- or cotype property imply martingale type or martingale cotype, respectively, and therefore
by Pisier \cite{Pisier:1975} an equivalent re-norming of the Banach space with a norm having a certain modulus of 
continuity or convexity, respectively. A classical case of decoupling, studied on its own, concerns randomly 
stopped sums of independent random variables, see for example the results of Klass \cite{Klass:1988,Klass:1990}. 
Another application for decoupling is stochastic integration.
For example, to get sufficient conditions for the existence of stochastic integrals, 
only the upper decoupling is needed. A starting point for this was \cite[Section 6]{KwapienWoyczynski:1991} where the existence of  decoupled tangent processes for left quasi-continuous processes  in the Skorohod 
space is studied. Finally, Kallenberg~\cite{Kallenberg:2017} proved the 
existence of decoupled tangent semi-martingales and two-sided decoupling inequalities, and considered applications to multiple stochastic integrals.
\smallskip

Let us come back to the relation between \eqref{eqn:umd} and \eqref{eqn:McConnell}.
By the above reduction of the UMD-property to Haar- or dyadic martingales 
the equivalence of~\eqref{eqn:McConnell} to the UMD-property remains true 
if we require that $(e_n)_{n=1}^N$ is a 
decoupled tangent sequence of $(d_n)_{n=1}^N$. However, when attempting to pass from 
UMD-trans\-forms to upper (or lower) decoupling inequalities 
in \eqref{eqn:McConnell} we encounter the following problems:
\smallskip

{\bf (P1)} {\it Blocking arguments} do not work as expected: If in the above notation $(e_n)_{n=1}^N$ is a decoupled tangent sequence of $(d_n)_{n=1}^N$ and 
$(E_\ell)_{\ell=1}^L$ is a decoupled tangent sequence of $(D_\ell)_{\ell=1}^L$, then, in general, the distributions of 
$\sum_{n=1}^N e_n$ and $\sum_{\ell=1}^L E_L$ do not coincide. In other words, the decoupling of 
$(D_\ell)_{\ell=1}^L$ cannot be obtained by the decoupling of  $(d_n)_{n=1}^N$
by taking blocks.
\smallskip

{\bf (P2)} As mentioned above, the consideration of {\it Haar martingales} 
in its {\it natural filtration} are sufficient to decide whether a Banach space is a UMD-space.
On the other hand, it is unknown whether, for example, upper decoupling inequalities for 
dyadic martingales  in \eqref{eqn:McConnell} 
imply upper decoupling for all martingales, see Section \ref{sec:relations_problems}. Regarding the filtration
it was, for example, not clear whether upper decoupling inequalities for $X$-valued stochastic integrals 
with respect to a scalar Brownian motion $(W(t))_{t \ge 0}$ in its {\it natural} filtration imply upper
decoupling inequalities under a {\it progressive enlargement} of the filtration such that $(W(t))_{t \ge 0}$
remains a Brownian motion in this filtration.
\smallskip 

The aim of this article is to contribute to these problems as follows:
\smallskip

\underline{Section \ref{sec:factorization}}: Theorem \ref{thm:factor} provides 
      a factorization of a random variable along regular conditional probabilities. 
      With this result we contribute to the results of
      Montgomery-Smith \cite{MontgomerySmith:1998} (see also Kallenberg \cite[Lemma 3.22]{Kallenberg:2002}).
      This result is the key to approximate our adapted processes in terms of Haar-like series.
\smallskip

\underline{Section \ref{sec:reduction_decoupling}}: The main result of Section \ref{sec:reduction_decoupling} is 
Theorem \ref{thm:simple_decoupling_of_vs_full_decoupling} with the following corollary 
(the definition of $\calP_{p\text{-ext}}$ and $\adapt_p(\Omega,\F;X,\calP_{p\textnormal{-ext}})$ is discussed 
below):
\newpage
\begin{ttheorem}
\label{thm:simple_decoupling_of_vs_full_decoupling_intro}
Let $X,Y,Z$ be Banach spaces, where $X$ is separable, let $S\in L(X,Y)$ and $T\in L(X,Z)$,
and let $p\in (0,\infty)$.
Assume $\underline{\Psi}_\lambda,\overline{\Psi}:[0,\infty) \to [0,\infty)$,
where $\Delta$ is a non-empty index-set, such that
\[
  \sup_{\xi\in (0,\infty)}  ( 1 + | \xi | )^{-p} \underline{\Psi}_\lambda (\xi) < \infty
    \sptext{1}{and}{1}
 \sup_{\xi\in (0,\infty)}  ( 1 + | \xi | )^{-p}  \overline{\Psi}(\xi) <\infty,
\]
and such that the $\underline{\Psi}_\lambda$ are lower semi-continuous and $\overline{\Psi}$ is upper semi-continuous.
Let $\calP$ be a set of Borel probability measures on $X$ such that 
$\int_X \|x\|^p \mu(dx)<\infty$ for all $\mu \in \calP$ and such that $\delta_0\in \calP$.
Then the following assertions are equivalent:
\begin{enumerate}
 \item\label{item:general_decoupling_intro} 
      For every stochastic basis $(\Omega,\calF,\P,\F)$
      and finitely supported\footnote{Only for only finitely many $n$ one has  $d_n\not \equiv 0$.}
      $(d_n)_{n=1}^\infty \in \adapt_p(\Omega,\F;X,\calP_{p\textnormal{-ext}})$ 
      it holds that
      \begin{equation}\label{eq:intro_upperdecouphard}
          \sup_{\lambda \in \Delta} 
	    \E \underline{\Psi}_{\lambda} \left ( 
	      \left \| 
		\sum_{n=1}^N S d_n 
	      \right \|_Y
	    \right )
	  \le 
	    \E \overline{\Psi} \left ( 
	      \left \| 
		\sum_{n=1}^N T e_n 
	      \right \|_Z 
	    \right )
	    ,
      \end{equation}
      whenever $(e_n)_{n\in \N}$ is an $\F$-decoupled tangent sequence 
      of $(d_n)_{n\in \N}$.
\item\label{item:simple_decoupling_intro}
      For every sequence
      of independent random variables 
      $(\varphi_n)_{n=1}^N \subset\calL^p(\P;X)$ 
      satisfying 
      $\calL(\varphi_n)\in \calP,$ and every
      $A_0\in \{\emptyset,\Omega\}$, $A_n\in \sigma(\varphi_1,\ldots,\varphi_n)$, $n\in \{1,\ldots,N\}$, 
      it holds that
      \begin{equation}\label{eq:intro_upperdecoupsimple}
	\sup_{\lambda \in \Delta} 
	  \E \underline{\Psi}_{\lambda} \left ( 
	    \left \|  
	      \sum_{n=1}^N  1_{A_{n-1}} S \varphi _n 
	    \right \|_Y 
	  \right )
         \le 
         \E \overline{\Psi} \left ( 
	  \left \| 
	    \sum_{n=1}^N 1_{A_{n-1}} T \varphi'_n 
	  \right \|_Z \right 
	  ), 
      \end{equation} 
      where $(\varphi'_n)_{n=1}^N$ is an independent copy
      of $(\varphi_n)_{n=1}^N$. 
\end{enumerate}
\end{ttheorem}

In Theorem~\ref{thm:simple_decoupling_of_vs_full_decoupling_intro} 
we use Definition \ref{definition:p-measure_extension},
where we extend the set $\calP$ to $\calP_{p\textnormal{-ext}}\supseteq \calP$ with
\begin{equation*}
\begin{aligned}
\calP_{p\textnormal{-ext}}
 :=
\Big\{ 
  \mu \in \calP_p(X)
  \colon
&
  \forall j\in \N\,
  \exists\,
    K_j\in \N
    \text{ and }
    \mu_{j,1},\ldots,\mu_{j,K_j} \in \calP
\\ & 
  \text{such that }
  \mu_{j,1} 
  *
  \cdots
  *
  \mu_{j,K_j} 
  \stackrel{w^*}{\rightarrow}
  \mu
  \text{ as } 
  j\rightarrow \infty
\\ & 
  \text{and } 
  \left(
    \mu_{j,1} 
    *
    \cdots
    *
    \mu_{j,K_j} 
  \right)_{j\in \N}
  \text{ is uniformly $L^p$-integrable}
\Big\}.
\end{aligned}
\end{equation*}
In Definition \ref{definition:class_of_adapted_processes}
$\adapt_p(\Omega,\F;X,\calP_{p\textnormal{-ext}})$ is defined to be 
the set of $\F$-adapted sequences $(d_n)_{n\in \N}$ in $\calL^p(\P;X)$ such that the regular versions 
$\kappa_{n-1}$ of $\P(d_n\in\cdot \,|\, \calF_{n-1})$ satisfy $\kappa_{n-1}[\omega,\cdot] \in \calP$
on a set of measure one. Theorem~\ref{thm:simple_decoupling_of_vs_full_decoupling_intro} remains valid 
if one exchanges $(d_n)_{n=1}^N$ with $(e_n)_{n=1}^N$ in \eqref{eq:intro_upperdecouphard}
and $(\varphi _n)_{n=1}^N$ with $(\varphi'_n)_{n=1}^N$ in \eqref{eq:intro_upperdecoupsimple}, respectively.
\medskip

In Theorem \ref{thm:simple_decoupling_of_vs_full_decoupling_intro} (i) $\Rightarrow$ (ii) is obvious,
our contribution in (ii) $\Rightarrow$ (i) is as follows: We extend the class of random variables, 
the admissible tangent distributions, and allow every progressive enlargement of the natural filtration of 
$\varphi = (\varphi_n)_{n=1}^N$. 
We avoid the straight blocking argument as in the case of UMD-transform, but still get a counterpart to
the Haar martingales in  \eqref{eq:intro_upperdecoupsimple}. 
Regarding problem {\bf (P2)} we show that in  \eqref{eq:intro_upperdecouphard}  and 
\eqref{eq:intro_upperdecoupsimple} the same class of conditional distributions can be taken.
 Regarding the enlargement of filtration we obtain the positive result as desired.
 Common settings in Theorem \ref{thm:simple_decoupling_of_vs_full_decoupling_intro} are the
following:
\bigskip
\begin{center}
\begin{tabular}{|c|c|c|r|}\hline
$\Delta$ & $\underline{\Psi_\lambda}(\xi)$ & $\overline{\Psi}(\xi)$ 
& $\sup_{\lambda \in \Delta} \E \underline{\Psi}_{\lambda} \left ( \left \| SF \right \|_Y \right )
   \le \E \overline{\Psi} \left ( \left \| TG \right \|_Z \right )$ \\  \hline 
${\rm card}(\Delta)=1$ & $\xi^p$ & $C^p \xi^p$ & $\| S F\|_{\calL^p(\P;Y)}  \le C \| T G\|_{\calL^p(\P;Z)}$ \\ 
&&&\\
${\rm card}(\Delta)=1$ & $1_{\{\xi > \mu\}}, \mu \ge 0$ & $C^p \xi^p$ & $\P (\| SF\|_Y > \mu) \le C^p \E \| TG\|_Z^p$ \\ 
&&&\\
$(0,\infty)$ &  $\lambda^p 1_{\{\xi > \lambda\}}$ & $C^p \xi^p$ & 
$\| SF\|_{\calL^{p,\infty}(\P;Y)} \le C \E \| T G\|_{\calL^p(\P;Z)}$ \\ 
\hline
\end{tabular}
\end{center}
\bigskip

\underline{Section \ref{sec:char_Pext_Rad}}:
Theorem ~\ref{thm:char_p-ext_Rademacher} provides a characterization 
of the measures $\calP_{p\textnormal{-ext}}^{\textnormal{Rad}}$ where 
$\calP^{\textnormal{Rad}} := \{ \frac12 \delta_{-\xi} + \frac12 \delta_{\xi}:\xi\in\R \}$.
\medskip

\underline{Section \ref{sec:stochInt}}: 
We study upper decoupling inequalities for dyadic martingales and equivalent properties.
More specifically, given $p\in (0,\infty)$ and a bounded and linear 
operator $T:X\to Y$ between Banach spaces $X$ and $Y$, where $X$ is separable,  we let
$D_p(T) := \inf c$, where the infimum is taken over all $c \in [0,\infty]$ such that 
      \[ \left \| \sum_{n=1}^N r_n T v_{n-1} \right \|_{\calL^{p}(\P;Y)} 
         \le c  
         \left \| \sum_{n=1}^N r_n' v_{n-1}  \right \|_{\calL^{p}(\P;X)} \] 
    for all $N\ge 2$,   for all $v_0\in X$ and $v_n:=h_n(r_1,\ldots,r_n)$ with 
      $h_n\colon \{-1,1\}^{n} \rightarrow X$, $n\in \{1,\ldots,N-1\}$,
where $(r_n)_{n=1}^N$ is a Rademacher sequence (the $r_n$ are independent and take the values $-1$ and $1$ with
probability $1/2$) and $(r_n')_{n=1}^N$ is an independent copy of $(r_n)_{n=1}^N$.
Using Theorem~\ref{thm:simple_decoupling_of_vs_full_decoupling_intro}
we prove the following (see Theorem~\ref{thm:stochastic_integration} below):

\begin{ttheorem}
\label{thm:stochastic_integration_intro} 
For a separable Banach space $X$ and $p,q\in (0,\infty)$ the following assertions are equivalent:
\begin{enumerate}
\item \label{label:1:thm:stochastic_integration_intro} 
      $D_p(I_X) < \infty$.
\item \label{label:2:thm:stochastic_integration_intro}
      There exists a $c\in (0,\infty)$ such that for 
      every stochastic basis $(\Omega,\calF,\P,\F=(\calF_t)_{t\in [0,\infty)})$, 
      every $\F$-Brownian motion $W=(W(t))_{t\in [0,\infty)}$ 
      and every simple $\F$-predictable $X$-valued process
      $(H(t))_{t\in [0,\infty)}$ it holds that
      \[ \left \| \int_0^\infty H(t) dW(t) \right \|_{\calL^{q}(\P;X)} 
         \le c \left \| S(H) \right \|_{\calL^{q}(\P)} \]
      with the square function 
      $S(H)(\omega):= \| f\mapsto \int_0^\infty f(t) H(t,\omega) dt \|_{\gamma(\calL^2((0,\infty);X)}$.

\end{enumerate}
\end{ttheorem}

The relation between decoupling and stochastic integration has already been studied in 
e.g.\ \cite{CoxVeraar:2011, Garling:1986, vanNeervenVeraarWeis:2007}.
Our contribution is that we take dyadic decoupling as a starting point and allow a progressive enlargement of the filtration
in $\eqref{label:1:thm:stochastic_integration_intro} \Rightarrow \eqref{label:2:thm:stochastic_integration_intro}$.
Therefore, our result is an extension of both~\cite{CoxVeraar:2011} and~\cite{Garling:1986}:
in \cite{CoxVeraar:2011} assertion
\eqref{label:2:thm:stochastic_integration_intro} was proved under the (so far more restrictive) assumption that 
one has a general upper decoupling. On the other hand, it seems that
the argument provided by Garling~\cite{Garling:1986} requires
the integrands to be adapted with respect to the filtration \emph{generated} by the Brownian motion 
they are integrated against. 
\smallskip

\underline{Section \ref{sec:decoupling_in_c0}}:
Here we discuss the behavior of the upper decoupling constant for dyadic martingales in the space $\ell_K^\infty$
and prove the following result:
\begin{ttheorem}
\label{thm:upper_decoupling_c0_intro}
For the diagonal operator
\[ D_{\sqrt{\log}}: c_0 \to c_0 
   \sptext{1}{with}{1}
   D_{\sqrt{\log}} \left ( (\xi_k)_{k=1}^\infty\right ) := \left ( \frac{\xi_k}{\sqrt{1+\log(k)}} \right ) \]
we have $D_2 \left ( D_{\sqrt{\log}} \right ) < \infty$.
\end{ttheorem}

\underline{Section \ref{sec:chaos}}:
For the decoupling with respect to dyadic martingales as 
considered in Sections \ref{sec:stochInt} and \ref{sec:decoupling_in_c0} we show in 
Theorem \ref{theorem:equivalence_chaos} that it is also natural to allow an upper decoupling with 
respect to a different distribution and that a Banach space $X$ allows 
for decoupling with respect to a distribution in some fixed chaos obtained from  $\calP_{2-{\rm ext}}^{\rm Rad}$ 
if and only if $D_2(I_X)<\infty$.
\medskip

In \underline{Section \ref{sec:relations_problems}} we compare the decoupling constants in our results to related decoupling
constants and state some open problems.
\underline{Appendix \ref{sec:main_proof}} contains the proof of Theorem~\ref{thm:simple_decoupling_of_vs_full_decoupling}.
For the reader's convenience,  \underline{Appendix \ref{sec:extrapolation}} contains a classical extrapolation result, 
Proposition~\ref{prop:extrapolation_decoupling_dyadic_martingales}, which is provided in the exact form that it is needed 
in Section~\ref{sec:stochInt}.


\section{Preliminaries}


\subsection{Some general notation}
\label{ss:notation}

We let $\N = \{1, 2,\ldots\}$ and $\N_0 = \{0,1,2,\ldots\}$.
For a vector space $V$ and $B\subseteq V$ we set $ -B := \{ x\in V \colon -x\in B \}.$
Given a non-empty set $\Omega$, we let $2^{\Omega}$ denote the system of all subsets of $\Omega$
and use $A \triangle B := (A\setminus B)\cup (B\setminus A)$ for $A,B\in 2^{\Omega}$.
A system of pair-wise disjoint subsets $(A_i)_{i\in I}$ of $\Omega$ is a \emph{partition} of $\Omega$,
where $I$ is an arbitrary index-set and $A_i = \emptyset$ is allowed, if $\bigcup_{i\in I} A_i = \Omega$.
If $(M,d)$ is a metric space we define
$d\colon M \times 2^M \rightarrow [0,\infty]$ by setting
$d(x,A) := \inf\{ d(x,y) \colon y\in A \}$ for all $(x,A)\in M\times 2^{M}$.
If $V$ is a Banach space and  $(M,d)$ a metric space, then $C(M;V)$ is the space of continuous maps 
from $M$ to $V$, and $C_b(M;V)$ the subspace of bounded continuous maps from $M$ to $V$. 
\medskip

{\sc Banach space valued random variables:} 
      Throughout this paper we let $X$ be a separable Banach space (not identically to $\{0\}$ to avoid 
      pathologies) and equip $X$ with the Borel 
      $\sigma$-algebra $\calB(X)$ generated by the norm-open sets. For $x\in X$ and $\eps>0$, the 
      corresponding open balls are given by $B_{x,\eps} := \{ y \in X \colon \| x - y \|_X < \eps \}$.
      For $B\in \calB(X)$ we let 
      $\bar{B}$ denote the closure of $B$ (with respect to $\left\| \cdot \right\|_{X}$),
      we let $B^{o}$ denote the largest open set contained in $B$, and we let $\partial B$
      denote the set $\bar{B}\setminus B^{o}$.
      Given a probability space $(\Omega,\calF,\P)$ and a measurable space $(S,\Sigma)$, 
      an $\calF/\Sigma$-measurable mapping $\xi\colon \Omega \rightarrow S$ 
      is called an \emph{$S$-valued random variable}. Therefore an
      $X$-valued random variable is an $\calF/\calB(X)$-measurable mapping. For a random variable $\xi:\Omega\to S$ the law
      of $\xi$ is given by $\calL(\xi)(A) := \P(\xi\in A)$ for $A\in \Sigma$.
      If $\xi_i$ are $S$-valued random variables
      on probability spaces $(\Omega_i,\calF_i,\calP_i)$, $i=1,2$,
      then $\xi_1$ and $\xi_2$ are 
      \emph{identical in law} if $\calL(\xi_1)= \calL(\xi_2)$.
      \medskip 
      
{\sc Lebesgue spaces:} For $X$ a separable Banach space and $(S,\Sigma)$ a measurable space, we define
      $\calL^0((S,\Sigma);X)$ to be the 
      space of $\Sigma/\calB(X)$-measurable mappings from $S$ to $X$. If $(S,\Sigma)$ is equipped with a $\sigma$-finite
      measure $\mu$ and $p\in (0,\infty)$, then we define 
       \[ \calL^p((S,\Sigma,\mu);X) 
	:= 
	   \left \{ \xi \in \calL^0((S,\Sigma,\mu);X)
	   \colon \| \xi \|_{\calL^p((S,\Sigma,\mu);X)}^p := \int_{S} \| \xi \|_X^p \,d\mu < \infty \right \}. 
       \]
      If there is no risk of confusion we write for example  
          $\calL^p(\mu;X)$ or $\calL^p(\Sigma;X)$ as shorthand notation for $\calL^p((S,\Sigma,\mu);X)$,
       and we set $\calL^p((S,\Sigma,\mu)):= \calL^p((S,\Sigma,\mu);\R)$.
\medskip
\pagebreak

{\sc Probability measures on Banach spaces:}
\renewcommand{\theenumi}{\arabic{enumi}}
\begin{enumerate}
\item By $\mathcal{P}(X)$ we denote the set of all probability 
      measures on $(X,\calB(X))$ and 
      for $p\in (0,\infty)$ we let
      \[ 
          \calP_p(X) := \left \{ 
                       \mu \in \calP(X) \colon \int_{X} \| x \|_X^p \,d\mu(x) < \infty
                        \right \}.
      \]
\item Given an index set $I\not = \emptyset$,  a family 
      $(\mu_i)_{i\in I} \subseteq \calP_p(X)$ is \emph{uniformly $\calL^p$-inte\-grable} if
      \[ \lim_{K\to\infty} \sup_{i\in I} \int_{\{ \|x\|_X \ge K\}}\ \|x\|_X^p d\mu_i (x) = 0. \]
      Accordingly, a family of $X$-valued random variables $(\xi_i)_{i\in I}$
      is \emph{uniformly $\calL^p$-integrable} if $(\calL(\xi_i))_{i\in I}$  is uniformly 
      $\calL^p$-integrable.
\item For $\mu \in \calP(X)$ and $\mu_n \in \calP(X)$, $n\in \N$,
      we write $\mu_n \stackrel{w^*}{\rightarrow} \mu$ as $n\rightarrow \infty$ if $\mu_n$
      converges weakly to $\mu$, i.e. if
      $\lim_{n\rightarrow \infty} \int_X f(x) \,d\mu_n (x) = \int_X f(x) \,d\mu(x)$ 
      for all $f \in C_b(X;\R)$. 
      Moreover, for a sequence of $X$-valued random variables $(\xi_n)_{n\in \N}$ 
      and an $X$-valued random variable $\xi$
      (possibly defined on different probability spaces)
      we write $\xi_n \stackrel{w^*}{\rightarrow} \xi$ as $n\rightarrow \infty$ 
      provided that
      $\calL(\xi_n) \stackrel{w^*}{\rightarrow} \calL(\xi)$ as $n\rightarrow \infty$.
\end{enumerate}
One main aspect of this article will be to work with a non-empty subset $\calP\subseteq \calP(X)$
instead of with the full set of measures $\calP(X)$ (see Theorem \ref{thm:simple_decoupling_of_vs_full_decoupling}
below). Here we will mostly assume that $\delta_0\in \calP$ with $\delta_0(B) = \one_{\{0\in B\}}$ being the Dirac measure
at $0\in X$. This is done to consider finitely supported sequences of $X$-valued random variables $(d_n)_{n\in \N}$ in 
a convenient way, but sometimes also for convenience within the proofs.\par 
\smallskip

We shall frequently use the following well-known result, which relates $\calL^p$-uniform integrability and convergence of moments:
\begin{lemma}\label{lem:conv_p_moments_vs_LpUI}
Let $p\in (0,\infty)$,
let $X$ be a separable Banach space, and let $\mu$, $(\mu_n)_{n\in\N}$ be a sequence in $\calP_p(X)$
such that $\mu_n\stackrel{w^*}{\rightarrow} \mu$. Then the following are equivalent:
\begin{enumerate}
 \item $\int_{X} \| x \|^{p}\,d\mu_n \rightarrow \int_{X} \| x \|^{p}\,d\mu_n$.
 \item $(\mu_n)_{n\in\N}$ is uniformly $\calL^p$-integrable.
\end{enumerate}
\end{lemma}

\begin{proof}
Apply e.g.~\cite[Lemma 4.11 (in (5) $\limsup$ can be replaced by $\sup$)]{Kallenberg:2002} to the random variables $\xi,\xi_1,\xi_2,\ldots$ 
where $\xi=\| \zeta \|_X^p$ and $\xi_n = \| \zeta_n \|_X^p$, and where $\calL(\zeta) = \mu$ and $\calL(\zeta_n)=\mu_n$.
\end{proof}

\renewcommand{\theenumi}{\roman{enumi}}
\medskip

{\sc Stochastic basis:} 
We use the notion of a stochastic basis $(\Omega,\calF,\P,\F)$, which is  
a probability space $(\Omega,\calF,\P)$ equipped with a filtration  $\F=(\calF_n)_{n\in \N_0}$, 
$\calF_0 \subseteq \calF_1 \subseteq \cdots \subseteq \calF$, and where we set 
$\calF_\infty := \sigma\left(\bigcup_{n\in \N_0} \calF_n\right)$. For measurable spaces $(\Omega,\calF)$ and $(S,\calS)$, and 
$\xi=(\xi_n)_{n\in \N}$ a sequence of $S$-valued random variables 
on $(\Omega,\calF)$, we let $\F^{\xi}=(\calF^{\xi}_n)_{n\in \N_0}$
denote the natural filtration generated by $\xi$, i.e.  
$\calF_0^\xi := \{\emptyset,\Omega\}$ and
$\calF_n^\xi := \sigma(\xi_1,\ldots,\xi_n)$ for $n \in \N$,
and  $\calF_\infty^\xi := \sigma(\xi_n : n\in \N)$.
\medskip

{\sc Convention:} If there is no risk of confusion, given an index set $I\neq \emptyset$ and 
a family $(a_i)_{i\in I}$ of random variables or elements of some Banach space, we say that 
      this family is finitely supported if only finitely many of the $a_i$ are not zero.


\subsection{Stochastic kernels} We provide some details for regular versions of conditional probabilities we shall use later.

\begin{definition}
Let $X$ be a separable Banach space. 
Given a measurable space $(S,\Sigma)$, a map  $\kappa: S\to \calP(X)$ is called a \emph{kernel} 
if it is $\Sigma/\calB(\calP(X))$-measurable,
where
\[ 
  \calB(\calP(X)) 
  :=
  \sigma\left(
    \left\{ 
      \{ \mu \in \calP(X) \colon \mu(A) \in B \}
      \colon
      A\in \calB(X), B \in \calB(\R)
    \right\}
  \right).
\]
\end{definition}

The following lemma is used later:

\begin{lemma}\label{lemma:BorelProb_countablyGen}
Let $X$ be a separable Banach space, then $\calB(\calP(X))$
is countably generated.
\end{lemma}

\begin{proof}
As $X$ is separable, there exists a countable $\pi$-system
$\Pi\subseteq \calB(X)$ that generates $\calB(X)$. Indeed, 
for a dense and countable subset $D\subseteq X$ 
we may choose as $\pi$-system
\[ \Pi 
  := \left \{
    \bigcap_{k=1}^{n} B_{x_k,\eps_k}
     \colon 
     n \in \N, x_k \in D, \eps_k \in (0,\infty)\cap \Q, k=1,\ldots, n
  \right \}.
\]
Define the countable system
$\calS \subseteq \calB (\calP(X))$ by
\[ \calS:= \{ \{ \mu \in \calP(X) : \mu(B) \in (a,b) \}\colon  a,b\in\Q, -\infty < a < b< \infty, B\in \Pi \}. \]
Define 
$
  \calA := \{ A  \in \calB(X) \colon \calP(X) \ni \mu \mapsto \mu(A) \in \R 
  \text{ is } \sigma(\calS)/\calB(\R)\text{-measurable} \}.
$
One may check that $\calA$ is a Dynkin system containing $\Pi$, whence 
the $\pi$-$\lambda$-Theorem implies $\sigma(\Pi) \subseteq \calA \subseteq \calB(X)$ and thus
$\calA=\sigma(\Pi)=\calB(X)$.
\end{proof}

\begin{remark}\label{rem:kernel} Let $X$ be a separable Banach space and $(S,\Sigma)$ a measurable space.
\begin{enumerate}
\item Alternatively one can say that a mapping $\kappa \colon S \times \calB(X)\to [0,1]$ 
      is a kernel if and only if the following two conditions hold:
      \begin{enumerate}
      \item For all $\omega\in S$ it holds that $\kappa[\omega,\cdot]\in \calP(X)$.
      \item For all $B\in \calB(X)$ the map $\omega \to \kappa[\omega,B]$ is $\Sigma/\calB(\R)$-measurable.
      \end{enumerate}
\item \label{item:kernel_equiv}	
      Let the space $(S,\Sigma)$ be equipped with a probability measure $\P$
      and let $\Pi\subseteq \calB(X)$ be a countable 
      $\pi$-system that generates $\calB(X)$ 
      (see the proof of Lemma~\ref{lemma:BorelProb_countablyGen}).
      For two kernels 
      $\kappa, \kappa' \colon S \rightarrow \calP(X)$ the following assertions are equivalent:
\begin{enumerate}
\item $\kappa[\omega,B] = \kappa'[\omega,B]$ for $\P$-almost all $\omega\in S$, 
for all $B\in\Pi$.
\item $\kappa[\omega,\cdot] = \kappa'[\omega,\cdot]$ for $\P$-almost all $\omega \in S$.
\end{enumerate}
\end{enumerate}
\end{remark}

We need the existence of kernels describing conditional probabilities:

\begin{theorem}[{\cite[Theorem 6.3]{Kallenberg:2002}}]
\label{thm:regular_cond_prob}
Let $X$ be a separable Banach space, $(\Omega,\calF,\P)$
a probability space, $\calG\subseteq \calF$
a sub-$\sigma$-algebra, and let $\xi:\Omega\to X$ be a random 
variable. Then there is a $\calG/\calB(\calP(X))$-measurable kernel
$\kappa \colon \Omega \to \calP(X)$ satisfying
\[  \kappa[\cdot,B] = \P(\xi \in B\,|\,\calG) \mbox{ a.s.} \]
for all $B\in \calB(X)$.
If $\kappa':\Omega\to\calP(X)$ is another kernel with this property, then
$\kappa'=\kappa$ a.s.
\end{theorem}

We refer to $\kappa$ as a \emph{regular version} of
$\P(\xi\in\cdot \,|\, \calG)$.


\subsection{Decoupling} 
We briefly recall the concept of decoupled tangent 
sequences as introduced by Kwapie{\'n} and Woyczy{\'n}ski in 
\cite{KwapienWoyczynski:1991}. For more details we refer to 
\cite{DelaPenaGine:1999,KwapienWoyczynski:1992} and the references therein.
 
\begin{definition}\label{def:decoupled_sequence}
Let $X$ be a separable Banach space, let $(\Omega,\calF,\P,(\calF_n)_{n\in \N_0})$ 
be a stochastic basis, and let $(d_n)_{n\in \N}$
be an $(\calF_n)_{n\in \N}$-adapted
sequence of $X$-valued random variables on
$(\Omega,\calF,\P)$.
A sequence of $X$-valued and $(\calF_n)_{n\in \N}$-adapted
random variables $(e_n)_{n\in \N}$ on $ (\Omega,\calF,\P)$
is called an {\em $(\calF_n)_{n\in \N_0}$-decoupled tangent sequence} of
$(d_n)_{n\in \N}$ provided there exists a $\sigma$-algebra 
$\calH \subseteq \calF$ satisfying $\sigma((d_n)_{n\in \N}) \subseteq \calH$ 
such that the following two conditions are satisfied:
\begin{enumerate}
\item {\sc Tangency:} For all $n\in \N$ and all $B\in \mathcal{B}(X)$ one has
      \[ \P(d_n \in B\,|\,\calF_{n-1})
         =
         \P(e_n \in B\,|\,\calF_{n-1})
         =
         \P(e_n\in B\,|\,\calH)
         \mbox{ a.s.} \]
\item {\sc Conditional independence:}
      For all $N\in\N$ and $B_1,\ldots,B_N\in\mathcal{B}(X)$ one has
      \[ \P(e_1\in B_1,\ldots, e_N\in B_N\,|\,\calH)
         =
         \P(e_1\in B_1\,|\,\calH)
         \cdot \ldots \cdot 
         \P(e_N\in B_N\,|\,\calH)
         \mbox{ a.s.} \]
\end{enumerate}
\end{definition}

A construction
of a decoupled tangent sequence is presented in \cite[Section 4.3]{KwapienWoyczynski:1992}.

\begin{example}\label{example:decoupled_sequence}
Let $(\Omega,\calF,\P,(\calF_n)_{n\in\N_0})$ be 
a stochastic basis, $(\varphi_n)_{n\in \N}$
and $(\varphi_n')_{n\in \N}$ 
two independent and identically distributed
sequences of independent, $\R$-valued random variables
such that $\varphi_n$ and $\varphi_n'$ are $\calF_n$-measurable 
and independent of $\calF_{n-1}$ for 
all $n\in \N$, and
let $(v_n)_{n\in \N_0}$ be an $(\calF_n)_{n\in\N_0}$-adapted sequence of $X$-valued 
random variables independent of $(\varphi_n')_{n\in \N}$.
Then $( \varphi'_n v_{n-1} )_{n\in \N}$ is an 
$(\calF_n)_{n\in\N_0}$-decoupled 
tangent sequence of 
$( \varphi_n v_{n-1} )_{n\in \N}$, where one may take
\[ \calH := \sigma( (\varphi_n)_{n\in \N}, (v_n)_{n\in \N_0}). \]
Similarly, $(\varphi_n)_{n\in \N}$ and $(\varphi_n')_{n\in \N}$
could be $X$-valued random 
variables and $(v_n)_{n\in \N_0}$ $\R$-valued. 
\end{example}


\section{A factorization for regular conditional probabilities}
\label{sec:factorization}

By Theorem \ref{thm:factor} below we contribute to the results obtained in 
Kallenberg \cite[Lemma 3.22]{Kallenberg:2002}) and Montgomery-Smith \cite{MontgomerySmith:1998}. Our contribution is that we provide
a factorization in the {\em strong sense}, not a representation in the distributional context.
Theorem \ref{thm:factor} is used to prove  Proposition~\ref{prop:simple_to_Lp} below, but might be of independent interest.
Also the usage in the proof of Proposition~\ref{prop:simple_to_Lp} yields to  a refined argument for the existence of a decoupled tangent sequence, 
so that it contributes to \cite{KwapienWoyczynski:1991} (cf. \cite[Proposition 6.1.5]{DelaPenaGine:1999}) as well.

\begin{theorem}\label{thm:factor}
Let $(\Omega,\calF,\P)$ be a probability space, $\calG\subseteq \calF$ be a $\sigma$-algebra, 
let $d\in \calL^0(\calF;\R)$ satisfy $d(\Omega)\subseteq [0,1)$, 
and let $\kappa \colon \Omega\times \calB([0,1)) \rightarrow [0,1]$
be a regular conditional probability kernel for $\calL(d|\calG)$. Let 
$(\bar{\Omega},\bar{\calF},\bar{\P}):= (\Omega \times (0,1], \calF\otimes \calB((0,1]),\P\otimes\lambda)$,
where $\lambda$ is the Lebesgue measure on $\calB((0,1])$.
Set $[0,0):=\emptyset$ and 
define $H\colon \bar{\Omega}\rightarrow [0,1]$, $d^{0}\colon \Omega\times [0,1] \rightarrow [0,1]$
by
\begin{align}
 H(\omega,s)     & := \kappa[ \omega, [0,d(\omega))] + s \kappa[ \omega, \{ d(\omega)\}],\\
 d^{0}(\omega,h) & := \inf\{ x\in [0,1] \colon \kappa[ \omega , [0,x]] \geq h \}.
\end{align}
Then
\begin{enumerate}
 \item\label{thm:factor:meas_H} $H$ is $\bar{\calF}/\calB([0,1])$-measurable, independent of $\calG\otimes \{\emptyset,(0,1]\}$, and 
 uniformly-$[0,1]$ distributed,  
 \item\label{thm:factor:meas_d} $d^0$ is $\calG\otimes \calB([0,1])/\calB([0,1])$-measurable, and 
 \item\label{thm:factor:rep} there is an $\calN\in \calF$ with $\P(\calN)=0$ such that 
      $d^{0}(\omega,H(\omega,s))= d(\omega)$ for all $(\omega,s)\in (\Omega \setminus\calN)\times (0,1]$.
\end{enumerate}
\end{theorem}
\smallskip

\begin{proof}
\eqref{thm:factor:meas_H}:
For all $n\in \N$, $\ell\in \{1,\ldots,2^n\}$ let $A_{n,\ell} := [(\ell-1)2^{-n},\ell 2^{-n})$.
Define $H_n \colon \bar{\Omega}\rightarrow [0,1]$ by
\begin{align*}
  H_n(\omega,s) 
& := 
  \sum_{\ell=1}^{2^n} 
    1_{\{d\in A_{n,\ell}\}} (\omega) 
    \big( 
      \kappa[ \omega, [0,(\ell-1) 2^{-n})] 
      + s \kappa[ \omega, A_{n,\ell}]
    \big).
\end{align*}
Note that $H_n$ is $\bar{\calF}/\calB([0,1])$-measurable. Moreover, for all $(\omega,s)\in \bar{\Omega}$
it holds that
\begin{align*}
& | H_n(\omega,s) - H(\omega,s) |
\\ &
\leq 
\sum_{\ell=1}^{2^n} 
  1_{\{d\in A_{n,\ell}\}} (\omega) 
  (1+s) \kappa[ \omega, A_{n,\ell} \setminus\{d(\omega)\}] 
\\ & \rightarrow 0 \text{ as } n\rightarrow \infty.
\end{align*}
In other words, $H$ is the point-wise limit of $\bar{\calF}/\calB([0,1])$-measurable functions, 
in particular, $H$ is $\bar{\calF}/\calB([0,1])$-measurable.
\smallskip

We now prove that the law of $H_n$ is given by the Lebesgue measure $\lambda_{[0,1]}$ on $\calB([0,1])$ and that 
$H_n$ is independent of $\calG\otimes \{\emptyset,(0,1]\}$ for all $n\in \N$. Let $n\in \N$, $G\in \calG$ and $B\in \calB([0,1])$. 
Note that for all $a,b\in [0,1]$ satisfying $a+b\leq 1$ it holds 
that $b \int_{0}^{1} 1_{\{a+sb\in B\}} \,ds = \lambda_{[0,1]}(B\cap [a,a+b])$. 
Using this and Fubini's theorem we obtain
\begin{align*}
&  \bar{P}( G \cap \{ H_n \in B \} )
\\  & =
  \sum_{\ell=1}^{2^n} 
    \int_{0}^{1} 
    \int_{G \cap \{ \omega\in \Omega \colon \kappa[ \omega, [0,(\ell-1) 2^{-n})]  + s \kappa[ \omega, A_{n,\ell}] \in B\} }
      1_{\{d\in A_{n,\ell}\}}(\omega)
    \,d\P
    \,ds
\\ & = 
  \sum_{\ell=1}^{2^n} 
    \int_{G} 
      \kappa[ \omega, A_{n,\ell} ]
      \int_{0}^{1}
	1_{\{ \omega\in \Omega \colon \kappa[ \omega, [0,(\ell-1) 2^{-n})]  + s \kappa[ \omega, A_{n,\ell}] \in B\}}(\omega)
      \,ds
    \,d\P
\\ & =
  \sum_{\ell=1}^{2^n} 
    \int_{G} 
      \lambda_{[0,1]}\big( B\cap \big[\kappa[\omega, [0, (\ell-1)2^{-n})], \kappa[ \omega, [0, \ell 2^{-n})] \big ]\big)
    \,d\P
\\ & =     \P(G)\cdot\lambda_{[0,1]}(B).
\end{align*}
This proves that $H_n$ is uniformly-$[0,1]$ distributed and independent of $\calG\otimes \{\emptyset,(0,1]\}$ 
for all $n\in \N$. This completes the proof 
of~\eqref{thm:factor:meas_H}, as $H$ is the point-wise 
limit of $(H_n)_{n\in\N}$. (Use e.g.\ that $\R$-valued random variables $\xi_1,\xi_2$ are 
independent if and only if for every two bounded continuous functions $f,g\in C(\R)$ 
it holds that $\E[f(\xi_1)g(\xi_2)] = \E[f(\xi_1)]\E[f(\xi_2)]$.)
\smallskip

\eqref{thm:factor:meas_d}: For all $x\in [0,1]$ note that 
\begin{align}\label{eq:dist_d} 
 \{ d^{0} \leq x \} 
 =
 \{ (\omega,h)\in \Omega \times [0,1] \colon \kappa[\omega, [0,x]] -h\geq 0 \}
    \in \calG\otimes \calB([0,1]). 
\end{align}
\par 
\eqref{thm:factor:rep}: 
It follows from~\eqref{eq:dist_d} and the definition of $H$
that we have, for all $x\in [0,1]$, that
\begin{align*}
& \{ (\omega,s)\in \bar{\Omega} \colon d^{0}(\omega,H(\omega,s)) \leq x \} 
\\ &=
\{ (\omega,s)\in \bar{\Omega}\colon \kappa[\omega, [0,x]] \geq \kappa[\omega,[0,d(\omega))] + s \kappa[\omega,\{d(\omega)\}] \}
\end{align*}
can be written as $B_x \times (0,1]$ \for some $B_x \in \calF$ and that we have that
\[ 
   B_x\times (0,1]
   \supseteq
   \{ (\omega,s)\in \bar{\Omega}\colon d(\omega) \leq x \} =:C_x \times (0,1].
\]
On the other hand from the fact that the image measure of the map $(\omega,s)\mapsto (\omega,H(\omega,s))$ as map a
map from $\overline{\Omega}$ into $\Omega \times [0,1]$ equals $\P\otimes \lambda$ we obtain, for all $x\in [0,1]$, that
\begin{align*}
\P(B_x) = \overline{\P}(B_x \times (0,1])
& = \E \int_{0}^{1} 1_{\{ d^{0}(\omega,h)\leq x \}} \,dh \,d\P(\omega)\\
& = \E \int_{0}^{1} 1_{\{  \kappa[\omega, [0,x]] \geq h \}} \,dh \,d\P(\omega) = \E \kappa[\cdot,[0,x]] = \P(C_x).
\end{align*}
It follows that $\P(B_x\setminus C_x)=0$ for all $x\in[0,1]$. Let $\calN := \cup_{q\in \Q\cap[0,1)} (B_q\setminus C_q)$ 
so that $\P(\calN)=0$. Then, observing that $B_x = \cap_{q\in \Q\cap[x,1)} B_q$ and $C_x = \cap_{q\in \Q\cap[x,1)} C_q$
for all $x\in [0,1)$, we have for all $(\omega,s)\in (\Omega\setminus \calN)\times (0,1]$ that 
$d^{0}(\omega,H(\omega,s)) = d(\omega)$.
\end{proof}

\begin{corollary}\label{cor:factor_through_subsigmaalgebra}
Let $(\Omega,\calF,\P)$ be a probability space, $\calG\subseteq \calF$ a $\sigma$-algebra, 
$X$ a separable Banach space, $d\in \calL^0(\calF;X)$.
Let 
$(\bar{\Omega},\bar{\calF},\bar{\P}):= (\Omega \times (0,1], \calF\otimes \calB((0,1]),\P\otimes\lambda)$,
where $\lambda$ is the Lebesgue measure on $\calB((0,1])$.
Then there exist random variables $H\colon \bar{\Omega}\rightarrow [0,1]$, $d^{0}\colon \Omega\times [0,1] \rightarrow X$
such that 
\begin{enumerate} 
 \item $H$ is a uniformly-$[0,1]$ distributed random variable independent of $\calG\otimes \{\emptyset, [0,1]\}$,
 \item $d^0$ is $\calG\otimes \calB([0,1])/\calB(X)$-measurable, and 
 \item there is an $\calN\in \calF$ with $\P(\calN)=0$ such that
       $d(\omega) = d^{0}(\omega, H(\omega,s))$ for all $(\omega,s)\in (\Omega\setminus\calN)\times (0,1]$.
\end{enumerate}
\end{corollary}

\begin{proof}
This is an immediate consequence of Theorem~\ref{thm:factor} and the fact that that $X$ is Borel-isomorphic to $[0,1)$, see 
e.g.~\cite[Theorem 13.1.1]{Dudley:03}.
\end{proof}


\section{Reduction of general decoupling}
\label{sec:reduction_decoupling}

To formulate our main result we introduce two basic concepts: In 
Definition \ref{definition:class_of_adapted_processes} we introduce a set 
of admissible adapted 
processes characterized by an assumption on the regular versions of the - in a sense - predictable projections, and in
Definition \ref{definition:p-measure_extension} we introduce an extension of 
a given set of probability measures that is natural in our context.

\begin{definition}
\label{definition:class_of_adapted_processes}
Let $X$ be a separable Banach space, $p\in (0,\infty)$,  
$\emptyset \not =  \calP \subseteq \calP_p(X)$, and
$(\Omega,\calF,\P,(\calF_n)_{n\in \N_0})$ be a stochastic basis. We denote by
$\adapt_p(\Omega,(\calF_n)_{n\in \N_0};X,\calP)$ the set of $(\calF_n)_{n\in \N}$-adapted sequences  
$(d_n)_{n\in \N}$ in $\calL^p(\P;X)$ with the property that for 
all $n\in \N$ there exists an $\Omega_{n-1} \in \calF$ 
satisfying $\P(\Omega_{n-1})=1$ and 
$\kappa_{n-1}[\omega,\cdot]\in\calP$ for all $\omega\in \Omega_{n-1}$,
where $\kappa_{n-1}$ is a regular version of $\P(d_n\in\cdot \,|\, \calF_{n-1})$.
\end{definition}

\begin{definition}
\label{definition:p-measure_extension}
For a separable Banach space $X$, $p\in (0,\infty)$ and $\emptyset \neq \calP\subseteq \calP_p(X)$
we let 
\begin{equation*}
\begin{aligned}
\calP_{p\textnormal{-ext}}
 :=
\Big\{ 
  \mu \in \calP_p(X)
  \colon
&
  \forall j\in \N\,
  \exists\,
    K_j\in \N
    \text{ and }
    \mu_{j,1},\ldots,\mu_{j,K_j} \in \calP
\\ & 
  \text{such that }
  \mu_{j,1} 
  *
  \cdots
  *
  \mu_{j,K_j} 
  \stackrel{w^*}{\rightarrow}
  \mu
  \text{ as } 
  j\rightarrow \infty
\\ & 
  \text{and } 
  \left(
    \mu_{j,1} 
    *
    \cdots
    *
    \mu_{j,K_j} 
  \right)_{j\in \N}
  \text{ is uniformly $L^p$-integrable}
\Big\}.
\end{aligned}
\end{equation*}
\end{definition}

The following Lemma~\ref{lem:char_Ppext_new} reveals some basic properties of $\calP_{p\textnormal{-ext}}$.
To this end, for $p\in (0,\infty)$ we introduce  on $\calP_p(X)\subseteq \calP(X)$ the metric
\begin{equation}
 d_p(\mu,\nu) := d_0(\mu,\nu) + \left| \int_X \| x \|^p \,d\mu(x) - \int_X \| x \|^p \,d\nu(x) \right|
\end{equation}
where $d_0$ is a fixed metric on $\calP(X)$ that metricizes the $w^*$-convergence, see for example
\cite[Theorem II.6.2]{Parthasarathy:1967}. 

\begin{lemma}
\label{lem:char_Ppext_new}
Let $p\in (0,\infty)$ and let $\calP\subseteq \calP_p(X)$ be non-empty. 
Then 
\begin{enumerate}
\item $(\calP_{p\textnormal{-ext}})_{p\textnormal{-ext}} = \calP_{p\textnormal{-ext}}$ and
\item $\calP_{p\textnormal{-ext}}$ is the smallest $d_p$-closed set $\calQ$ with  
      $\calQ\supseteq \calP$ and $\mu * \nu \in \calQ$ for all $\mu,\nu \in \calQ$.
\end{enumerate}
\end{lemma}

\begin{proof}
The assertion follows from Lemma \ref{lemma:*-hull} given the convolution is continuous with respect to $d_p$.
To verify this, we let $\mu,\nu,\mu_n,\nu_n\in \calP_p(X)$, $n\in\N$, such that 
$\lim_{n\rightarrow \infty} d_p(\mu,\mu_n) = \lim_{n\rightarrow \infty} d_p(\nu,\nu_n) = 0$. 
It is know that $\mu_n*\nu_n \stackrel{w^*}{\to} \mu * \nu$ as well (one can use \cite[Theorem 4.30]{Kallenberg:2002}).
Because for $K>0$  we have  
\begin{multline*}
        2^{-\max\{0,p-1\}} \int_{\|x+y\|_X\ge K\}}\|x+y\|_X^p d \mu_n(x) d\nu_n(y) 
    \le \int_{\{ \|x\|_X \ge K/2 \}} \|x\|_X^p d\mu_n(x) \\ + \int_{\{ \|y\|_X \ge K/2 \}} \|y\|_X^p d\nu_n(x) + 
        \frac{2^{p+1}}{K^p}\int_X \|y\|_X^p d\nu_n(y) \int_X\|x\|^p_Xd\mu_n(x)
\end{multline*}
we get that $\mu_n * \nu_n$ is uniformly $\calL^p$-integrable and use Lemma \ref{lem:conv_p_moments_vs_LpUI}
to obtain the convergence of the $p$-th moments.
\end{proof}
 
Now we formulate the main result of this section, i.e., Theorem~\ref{thm:simple_decoupling_of_vs_full_decoupling} below.
The proof of this theorem can be found in Appendix~\ref{sec:main_proof}.
\smallskip

\begin{theorem}
\label{thm:simple_decoupling_of_vs_full_decoupling}
Let $X$ be a separable Banach space,
let $\Phi \in C(X \times X;\R)$
be such that there exist 
constants $C,p\in (0,\infty)$ for which it holds that
$$ 
  | \Phi(x,y) |	
  \leq
  C( 1 + \| x \|_X^p + \| y \|_X^p)
$$
for all $ (x,y) \in X\times X$, and
let $\calP \subseteq \mathcal{P}_p(X)$ with $\delta_0\in \calP$.
Then the following assertions are equivalent:
\begin{enumerate}
\item\label{item:general_decoupling} 
      For every stochastic basis $(\Omega,\calF,\P,\F)$ with $\F=(\calF_n)_{n\in \N_0}$ and        
      every finitely supported\footnote{Recall that this means that there is an $N\in \N$ with $d_n\equiv 0$ for $n > N$.}
      $(d_n)_{n\in \N} \in \adapt_p(\Omega,\F;X,\calP_{p\textnormal{-ext}})$
      it holds that
      \begin{equation}\label{eq:decouphard}
	\E \Phi\left (\sum_{n=1}^\infty d_n,\sum_{n=1}^\infty e_n \right ) \le 0,
      \end{equation}
      provided that $(e_n)_{n\in \N}$ is an $\F$-decoupled tangent sequence 
      of $(d_n)_{n\in \N}$.
\item\label{item:simple_decoupling}
      For every probability space $(\Omega,\calF,\P)$,
      every finitely supported sequence
      of independent random variables 
      $\varphi = (\varphi_n)_{n\in \N}$ in $ \calL^p(\P;X)$ 
      satisfying $\calL(\varphi_n)\in \calP$ for all $n\in \N$,
      and every $ A_n \in \calF_{n}^{\varphi}$, $n\in \N_0$,
      it holds that
      \begin{equation}\label{eq:decoupsimple} 
	\E \Phi\left (
	  \sum_{n=1}^\infty \varphi_n  1_{A_{n-1}}, 
	  \sum_{n=1}^\infty \varphi'_n 1_{A_{n-1}} \right ) \le 0, 
      \end{equation} 
      where $(\varphi'_n)_{n\in \N}$ is an independent copy
      of $(\varphi_n)_{n\in \N}$.
\end{enumerate}
\end{theorem}
\smallskip

\begin{remark}\label{remark:trivial_implication}
For a sequence of random variables $(1_{A_{n-1}} \varphi_n )_{n\in \N}$ 
as in the setting of statement~\eqref{item:simple_decoupling}
in Theorem~\ref{thm:simple_decoupling_of_vs_full_decoupling} it holds 
that $(1_{A_{n-1}} \varphi_n )_{n\in \N} \in \adapt_p(\Omega,\F^{\varphi,\varphi'};X,\calP)$
with $\F^{\varphi,\varphi'}=(\calF^{\varphi,\varphi'}_n)_{n\in \N_0}$ 
given by
$\calF_0^{\varphi,\varphi'}:= \{\emptyset,\Omega\}$ and
$\calF_n^{\varphi,\varphi'}:=\sigma(\varphi_1,\varphi'_1,\ldots,\varphi_n,\varphi'_n)$ for $n\in \N$.
In particular, the 
implication~\eqref{item:general_decoupling}~$\Rightarrow$~\eqref{item:simple_decoupling}
is obvious by Example~\ref{example:decoupled_sequence}.
\end{remark}
\medskip

Theorem \ref{thm:simple_decoupling_of_vs_full_decoupling} can be extended to more general
$\Phi$. This is done by exploiting the following observation that is an immediate consequence
of Theorem \ref{thm:simple_decoupling_of_vs_full_decoupling}:

\begin{corollary}
\label{cor:simple_decoupling_of_vs_full_decoupling_supremum}
Let $X$ be a separable Banach space and let $\Phi_\lambda \in C(X \times X;\R)$, $\lambda \in \Delta$, 
for an arbitrary non-empty index-set $\Delta$. Suppose that there exist a $p\in (0,\infty)$
and constants $C_\lambda \in (0,\infty)$, $\lambda \in \Delta$, such that 
$$ 
  | \Phi_\lambda(x,y) |
  \leq
  C_\lambda( 1 + \| x \|_X^p + \| y \|_X^p)
$$
for all $ (x,y) \in X\times X$, and
let $\calP \subseteq \mathcal{P}_p(X)$ with $\delta_0\in \calP$.
Then assertions (i) and (ii) of Theorem ~\ref{thm:simple_decoupling_of_vs_full_decoupling}
remain equivalent when inequalities~\eqref{eq:decouphard} and~\eqref{eq:decoupsimple} are replaced by 
 \[ \sup_{\lambda\in \Delta}\E \Phi_\lambda\left (\sum_{n=1}^\infty d_n,\sum_{n=1}^\infty e_n \right ) \le 0 \] 
and 
 \[  \sup_{\lambda \in \Delta} 
        \E \Phi_\lambda \left ( \sum_{n=1}^\infty \varphi_n  1_{A_{n-1}}, 
                                \sum_{n=1}^\infty \varphi'_n 1_{A_{n-1}} 
                        \right ) \le 0,
      \] 
respectively.
\end{corollary}
\medskip

This corollary allows us to prove Theorem \ref{thm:simple_decoupling_of_vs_full_decoupling_intro}
from Section \ref{sec:intro}:

\begin{proof}[Proof of Theorem \ref{thm:simple_decoupling_of_vs_full_decoupling_intro}]
The statement for general $\Delta$ follows from the case $\Delta=\{\lambda_0\}$ so that we may assume this case and let
$\underline{\Psi} := \underline{\Psi}_{\lambda_0}$. By the lower- and upper semi-continuity we can find continuous
$\underline{\Psi}^\ell,\overline{\Psi}^\ell:[0,\infty)\to [0,\infty)$, $\ell \in \N$, such that 
$\underline{\Psi}^\ell(\xi) \uparrow \underline{\Psi}(\xi)$ and 
$C(1+| \xi |^p) \geq \overline{\Psi}^\ell(\xi) \downarrow \overline{\Psi}(\xi)$ for all $\xi \in [0,\infty)$. Next, we set
$\Phi_\ell(x,y):= \underline{\Psi}^\ell(\|Sx\|_Y) - \overline{\Psi}^\ell(\|Ty\|_Z)$, $\ell \in \N$. 
Then the monotone convergence theorem implies that for all $\xi,\eta\in \calL^p(X)$ the conditions
$\sup_{\ell\in \N}\E \Phi_\ell(\xi,\eta) \le 0$ and
$\E \left [ \underline{\Psi}(\|S\xi\|_Y) - \overline{\Psi}(\|T\eta\|_Z)\right ] \le 0$
are equivalent.
\end{proof}

Let us list some common choices of $\calP$ in the setting of decoupling 
inequalities. To do so, we exploit the following lemma:

\begin{lemma}
\label{lem:PhiConv}
Let $C,p\in (0,\infty)$, let
$X$ be a separable Banach space, let
$(\Omega,\calF,\P)$ be a probability space,
and let 
$
 \Phi \in C(X;\R)
$ 
be such that
\begin{equation}\label{eq:Phi_growth}
 | \Phi(x) |
 \leq 
 C(1+\| x \|_X^p )
\end{equation}
for all $x\in X$. Assume $\xi,\xi_n \in \calL^p(\P;X)$, 
$ n \in \N$, such that $\xi_n \stackrel{w^*}{\rightarrow} \xi$ 
as $n\rightarrow \infty$ and that $(\xi_n)_{n\in \N}$
is uniformly $L^p$-integrable. Then 
\begin{equation}\label{eq:PhiConv}
 \lim_{n\rightarrow \infty}
 \E \Phi ( \xi_n)
 =
 \E \Phi ( \xi ).
\end{equation}
\end{lemma}

\begin{proof}
It follows from the uniform $L^p$-integrability of $(\xi_n)_{n\in \N}$ and 
estimate~\eqref{eq:Phi_growth} that $(\Phi(\xi_n))_{n\in \N} $
is uniformly $L^1$-integrable. Moreover, note that $\Phi(\xi_n) \stackrel{w^*}{\rightarrow} \Phi(\xi)$ as $n\rightarrow \infty$,
so that we may apply Lemma \ref{lem:conv_p_moments_vs_LpUI} for $p=1$.
\end{proof}

We remark that $\xi_n \to \xi$ in $\calL^p(\P;X)$, $\xi_n,\xi \in \calL^p(\P;X)$, implies 
the assumptions of Lemma \ref{lem:PhiConv} (see \cite[Lemma 4.7]{Kallenberg:2002}).

\medskip

\begin{example}[\sc Adapted processes]
If $p\in (0,\infty)$ and $\calP=\calP_p(X)$, then 
$\calP_{p\textnormal{-ext}}=\calP$ by Lemma~\ref{lem:PhiConv}
and the space $\adapt_p(\Omega,\F;X,\calP)$
      consists of all $(\calF_n)_{n\in \N}$-adapted processes 
      $(d_n)_{n\in \N}$ in $\calL^p(\P;X)$.
\end{example}
\medskip

\begin{example}[\sc $L^p$-martingales]
If $p\in [1,\infty)$ and $\calP$ consists of all mean zero measures in $\calP_p(X)$,
      then $\calP_{p\textnormal{-ext}} = \calP$ by Lemma~\ref{lem:PhiConv} 
      (one can test with $\Phi(x):= \langle x,a \rangle$, where $a\in X'$ and $X'$ is the norm-dual) and  
      $\adapt_p(\Omega,\F;X,\calP)$ consists of all $L^p$-integrable $\F$-martingale difference sequences.
\end{example}
\medskip

\begin{example}[\sc Conditionally symmetric adapted processes]
Suppose $p\in (0,\infty)$ and  $\calP$ consists of all symmetric measures 
      in $\calP_p(X)$. As a measure $\mu\in \calP(X)$
      is symmetric if and only if for all $f\in C_b(X;\R)$ it holds that 
      $\int_X f(x) \,d\mu(x) = \int_X f(-x) \,d\mu(x)$, 
      it follows that $\calP_{p\textnormal{-ext}}=\calP$.
      Moreover, the set $\adapt_p(\Omega,\F;X,\calP)$ 
      consists of all sequences of $X$-valued $(\calF_n)_{n\in \N}$-adapted sequences 
      of random
      variables $(d_n)_{n\in \N}$ such that $d_n \in \calL^p(\P;X)$ and $d_n$
      is $\calF_{n-1}$-conditionally
      symmetric for all $n\in \N$, i.e., for all $n\in \N$ and all $B \in \calB(X)$
      it holds that
      $
	\P( d_n \in B \,|\, \calF_{n-1} ) = \P( d_n \in -B \,|\, \calF_{n-1} )
      $ a.s.
\end{example}
\medskip

\begin{example}[\sc One dimensional laws]\label{example:oneDlaw}
If $p\in (0,\infty)$, $\emptyset \not = \calP_0 \subseteq \calP_p(\R)$, 
      and
      \[
	\calP =\calP(\calP_0,X) :=
		\left\{ \mu \in \calP_p(X) 
	  \colon 
	    \exists \mu_0 \in \calP_0,\, x\in X \colon
	    \mu(\cdot) = \mu_0\big(\{ r \in \R \colon rx \in \cdot \} \big)
	\right\}, 
       \]
      then an $X$-valued random variable $\varphi$
      satisfies $\calL(\varphi)\in \calP$ if and only if there exists
      an $x\in X$ and a $\R$-valued random variable $\varphi_0$
      such that $\varphi = x \varphi_0$ and $\calL(\varphi_0) \in \calP_0$.
      Moreover, $\adapt_p(\Omega,\F;X,\calP)$ contains
      all sequences of the 
      form $(\varphi_n v_{n-1} )_{n\in\N}$ where $(\varphi_n)_{n\in\N}$ is an 
      $(\calF_n)_{n\in \N}$-adapted sequence 
      of $\R$-valued random variables 
      such that $\varphi_n$ is independent of $\calF_{n-1}$ 
      and $\calL(\varphi_n) \in \calP_{0}$, and
      $v_{n-1} \in \calL^p(\calF_{n-1};X)$ for all $n\in \N$. 
      Finally, it holds that 
      \[ 
         \calP((\calP_0)_{p\textnormal{-ext}},X) 
	 \subseteq  \calP_{p\textnormal{-ext}}.
      \]
      For a discussion of the case that 
      $\calP_0 = \{ \frac12 (\delta_{-1} + \delta_{1})\}$ and $X=\R$ we 
      refer to the next section. 
\end{example}


\section{Characterization of \texorpdfstring{$\calP_{p\textnormal{-ext}}$}{extended set of measures} for Rademacher sums}
\label{sec:char_Pext_Rad}

Given $p\in (0,\infty)$ and a non-empty $\calP\subset \calP_p(X)$, there does not seem to be a 
simple characterization of $\calP_{p\textnormal{-ext}}$ 
as defined in~Definition~\ref{definition:p-measure_extension}. However, Theorem ~\ref{thm:char_p-ext_Rademacher}
below deals with this question in the specific 
case that $X=\R$ and $\calP = \calP^{\textnormal{Rad}} := \{ \frac12 \delta_{-\xi} + \frac12 \delta_{\xi}:\xi\in\R \}$.
Recall that a \emph{Rademacher sequence} is a sequence of independent,
identically distributed random variables $(r_n)_{n\in\N}$ satisfying 
$\calL(r_1)= \frac12 \delta_{-1} + \frac12 \delta_{1} $.

\begin{theorem}\label{thm:char_p-ext_Rademacher}
Let $\calP^{\textnormal{Rad}} = \{ \frac12 \delta_{-\xi} + \frac12 \delta_{\xi}:\xi\in\R \}$.
Then $\calP_{p\textnormal{-ext}}^{\textnormal{Rad}} = \calP_{2\textnormal{-ext}}^{\textnormal{Rad}}$ 
for all $p\in (0,\infty)$ and for every $\mu \in \calP_{2\textnormal{-ext}}^{\textnormal{Rad}}$ 
there exist $\sigma\in [0,\infty)$ and $(a_n)_{n\in \N}\in \ell^2$
such that 
$\mu = \calL( \sigma \gamma + \sum_{n=1}^{\infty} a_n r_n )$,
where $\gamma$ is a standard Gaussian random variable 
and $(r_n)_{n\in \N}$ is a Rademacher sequence independent
of $\gamma$.
\end{theorem}

\begin{proof}
Let $p\in (0,\infty)$
and let $\mu \in \calP_{p\textnormal{-ext}}^{\textnormal{Rad}}$,
and let $(r_n)_{n\in\N}$ be a Rademacher sequence. Then 
for $k\in \N$ there exist finitely supported 
$a_k = (a_{k,n})_{n\in \N} \in \ell^2$
such that $a_{k,n}\ge 0$ and,
if we define $\psi_k := \sum_{n=1}^{\infty} a_{k,n} r_n $,
then $(\psi_k)_{k\in \N}$ is uniformly $L^p$-integrable and
$\calL(\psi_k) \stackrel{w^*}{\rightarrow} \mu$.
\smallskip

{\bf (a)} It follows from the Kahane-Khintchine inequalities that 
$(\psi_k)_{k\in \N}$ is uniformly $L^q$-integrable for 
all $q\in (0,\infty)$, from which we can conclude 
that $\calP_{p\textnormal{-ext}}^{\textnormal{Rad}} = \calP_{q\textnormal{-ext}}^{\textnormal{Rad}}$
for all $p,q\in (0,\infty)$ and that $\sup_{k\in \N} \E | \psi_k |^2 < \infty$.
As the law of $\psi_k$ is invariant with respect to permutations of 
the underlying Rademacher sequence we assume that
\[ 0 \leq a_{k,n+1} \leq a_{k,n} 
   \sptext{1}{for all}{1} 
   n,k \in \N, \]
so that  
\[     \sup_{k,n\in \N} a_{k,n}
     = \sup_{k\in \N} a_{k,1}
     \le \sup_{k\in \N} \| a_k \|_{\ell^2} =:c < \infty. \]
{\bf (b) Finding an appropriate sub-sequence 
$(\xi_j)_{j\in \N}$ of $(\psi_k)_{k\in \N}$}: 
Step (a) implies that
\begin{equation}\label{eq:bound_a}
  a_{k,n} \leq c n^{-1/2}
  \sptext{1}{for}{1} k,n \in \N.
\end{equation}
As $(a_{j,1})_{j \in \N}$ is bounded it contains a convergent sub-sequence 
$( a_{j^1_i,1})_{i \in \N}$ with limit $a_{0,1}$. 
Because $( a_{j^1_i, 2})_{i \in \N}$ is bounded as well, there is a convergent
sub-sequence $( a_{j^2_i, 2})_{i \in \N}$ with limit $a_{0,2}$.
Continuing, we extract from $( a_{j^{m-1}_i, m})_{i \in \N}$ a convergent
sub-sequence $( a_{j^{m}_i, m})_{i \in \N}$ with limit $a_{0,m}$.
Finally, for $m\in \N$, we pick $k_m \in \{ j^m_i : i \in N\}$ with
$1\le k_1 < k_2 < \cdots$ such that $| a_{k_m,n} - a_{0,n} | < c m^{-1/2}$
for all $n \in \{1,\dots, m \}$.
By \eqref{eq:bound_a} it follows that
$
  | a_{0,n} - a_{k_m,n} |
  \leq
  c \min \{ m^{-1/2} , n^{-1/2} \}
$
for all $ n,m\in \N$.
Moreover, by Fatou's lemma we have
$ 
  \| a_0 \|_{\ell^2} 
  \leq c
$
and therefore 
$ 
  \sup_{m\in \N}
    \| a_{k_m} - a_0 \|_{\ell^2}
  \leq 2c,
$
whence there exists a 
$ \sigma \in [0,2c]$
and a sub-sequence 
$ (b_j)_{j\in \N}
  :=
  ( a_{k_{m_j}} )_{j\in \N}
$
such that $m_j\geq j$ for all $j\in \N$
and 
$ \lim_{j\rightarrow \infty} 
 \| b_j - a_0 \|_{\ell^2} 
 =
 \sigma
$. Note that by construction
we have
\[
  | a_{0,n} - b_{j,n} |
  \leq
  c \min \{ j^{-1/2} , n^{-1/2} \}
\]
for all $j,n\in \N$.
The random variables
$ \xi_j := \sum_{n=1}^{\infty} b_{j,n} r_n $,
$ j\in \N$, form a sub-sequence of $(\psi_k)_{k\in \N}$,
so that $\xi_j \stackrel{w^*}{\rightarrow} \psi $
as well.
\medskip

{\bf (c) Decomposition of $(\xi_j)_{j\in \N}$:}
By construction there exists a non-decreasing sequence $(N_j)_{j\in N}$ in $\N$ such that
$\lim_{j\to\infty} N_j=\infty$ and such that
$
 \lim_{j\rightarrow \infty} 
  \sum_{n=1}^{N_j} | a_{0,n} - b_{j,n} |^2
  = 0
$, e.g. $N_j := \lfloor j^r\rfloor$ for some $r\in (0,1)$. 
Now we decompose
\[ \xi_j = \sum_{n=1}^{N_j} b_{j,n} r_n + \sum_{n=N_j+1}^\infty b_{j,n} r_n 
         =: \eta_j + \zeta_j. \]
The random variables  $\eta_j$ and $\zeta_j$ are independent. Moreover,
\[ \E \left | \sum_{n=1}^\infty a_{0,n} r_n - \eta_j \right |^2 
   =  \sum_{n=1}^{N_j} | a_{0,n} - b_{j,n} |^2 + \sum_{n=N_j+1}^\infty | a_{0,n}|^2 \to 0 
   \sptext{1}{as}{1} j\to \infty. \]
Regarding the $(\zeta_j)_{j\in \N}$ we observe that, by construction,
\[ \lim_{j\to \infty} \sup_{n>N_j} |b_{j,n}| = 0 
   \sptext{1}{and}{1} 
   \lim_{j\to \infty} \sum_{n=N_j+1}^\infty |b_{j,n}|^2 = \sigma^2,\]
where we use $\lim_{j\to\infty} \| b_j - a_0\|_{\ell^2}=\sigma$ and 
$\lim_{j\to\infty} \| a_0 - (b_{j,n} \one_{n\le N_j})_{n\in \N} \|_{\ell^2}=0$.
Adapting the Lindeberg condition \cite[Theorem 5.12]{Kallenberg:2002}) yields that $\zeta_j \stackrel{w^*}{\rightarrow} \sigma \gamma $
 where $ \gamma $ is a 
standard Gaussian distributed random variable. Now the assertion follows by the independence of $\eta_j$ and $\zeta_j$, and
because of $\eta_j \stackrel{w*}{\rightarrow} \sum_{n\in\N} a_{0,n} r_n$.
\end{proof}

\begin{remark}
\label{remark:inclusion_p_rad_ext}
Let $(r_n)_{n\in \N}$ be a Rademacher sequence and define 
\begin{equation*}
 \calP:=
 \left\{ 
    \calL\left( \sum_{n=1}^{\infty} a_n r_n \right)
    \colon
    (a_n)_{n=1}^{\infty} \in \ell^2 
 \right\}.    
\end{equation*}
Then we have the following:
\begin{enumerate}
\item \label{item:calP_small}
      $\calP \subsetneq \calP_{\textnormal{2-ext}}^{\textnormal{Rad}}$
      as $\calP$ does not contain the Gaussian law $N(0,1)$:
      let $(r_n)_{n\in \N}$ be a Rademacher sequence and
      $a = (a_n)_{n\in \N} \in \ell^2$, and
      suppose $ \xi:=\sum_{n=1}^{\infty} a_n r_n \sim N(0,1)$.
      We have $\E \xi^2 = \| a \|_{\ell^2}^2$ and 
      $\E \xi^4 = 3 \| a \|_{\ell^2}^4  -2 \| a \|_{\ell^4}^4$.
      Hence if $\E \xi^2 = 1$, then $\E \xi^4 < 3$, so that $\xi$ cannot have
      a standard Gaussian law.
\item There are symmetric measures in $\calP_2(\R)$ that do not belong to
      $\calP_{\textnormal{2-ext}}^{\textnormal{Rad}}$: we take a symmetric $\mu\in \calP_2(\R)$
      such that for a random variable $\xi$ with the law $\mu$ we have $\P(\xi=0)>0$, 
      $\E \xi^2 = 1$, and $\E \xi^4 = 3$. The condition $\P(\xi=0)>0$ implies that we can assume
      $\xi=\sum_{n=1}^{\infty} a_n r_n $ for  $a = (a_n)_{n\in \N} \in \ell^2$.
      But now we can conclude with the argument from \eqref{item:calP_small}.
\end{enumerate}
\end{remark}


\section{Dyadic decoupling and stochastic integration}
\label{sec:stochInt}

In this section we consider the case of decoupling of dyadic martingales and 
combine our main result, i.e., Theorem~\ref{thm:simple_decoupling_of_vs_full_decoupling}
with a standard extrapolation argument to obtain a decoupling result that is useful for 
the theory of stochastic integration of vector-valued stochastic processes, see 
Theorem~\ref{thm:stochastic_integration} below.
\bigskip

Before we start we explain by the  next lemma that {\it any reasonable upper decoupling} implies that the underlying Banach space $X$
cannot contain subspaces $E_N$ of dimension $N$, $N=1,2,\ldots$, such that the Banach-Mazur distances of $E_N$
to $\ell^\infty_N$ are uniformly bounded in $N$. The lemma is an adaptation of the examples found in
\cite{Garling:1990} and \cite{Veraar:2007}.
\medskip

\begin{lemma}
\label{lemma:upper_decoupling_finite_cotype}
Let $q\in [2,\infty)$. Then there exists a constant $c_q\in (0,\infty)$ 
such that for all $N\in \N$, 
every sequence of independent, identically distributed $\R$-valued random variables $\varphi=(\varphi_n)_{n=1}^{N}$,
and every sequence of independent, identically distributed mean-zero $\R$-valued random variables 
$\varphi'=(\varphi_n')_{n=1}^{N}$
independent of $\varphi$ and such that $\E(|\varphi_1'|^q)<\infty$
there exists an $(\calF_n^{\varphi})_{n=0}^{N}$-adapted
$\ell_{2^{N}}^{\infty}$-valued sequence $(v_n)_{n=0}^{N-1}$ such that:
\begin{enumerate}
\item $\E \Big( \left \| \sum_{n=1}^N \varphi'_n v_{n-1} \right\|^{q}_{\ell^\infty_{2^N}} \Big)
       \le c_q^q \, N^\frac{q}{2} \, \E(|\varphi_1'|^q)$.
\item If for some $c\in (0,\infty)$, $\eta \in (0,1]$ it holds that $\P(|\varphi_1|\ge c) = \eta$,
      then
      $\P \left ( \| \sum_{n=1}^N \varphi_n v_{n-1}  \|_{\ell^\infty_{2^N}} \ge \tfrac{c\eta N}{2} \right )  \ge \frac{1}{2}$
      for $N\ge 2/\eta$.  
\end{enumerate}
\end{lemma}

\begin{proof}
Set $\D_N := \{-1,1\}^N$ and for every
$\alpha\in \D_n$, $n\in \{1,\ldots,N-1\}$ define
$\D_N(\alpha):= \{ e \in \mathbb{D}_N \colon e_1=\alpha_1,\ldots,e_n=\alpha_n\}$.
Moreover, for all $n\in \{1,\ldots,N\}$ define $\alpha_n\colon \Omega\rightarrow \{-1,1\}$ and 
$v_{n-1}\colon \Omega \rightarrow \ell^{\infty}(\mathbb{D}_N)$
by  
\equa 
      \alpha_{n}(\omega) 
&:= & \begin{cases}
             1, & \varphi_n(\omega)  > 0 \\
            -1, & \varphi_n(\omega)  \leq 0 
             \end{cases}, \\
      \hspace*{1em} 
      v_{n-1}(\omega)(e) 
&:= & e_n 1_{\mathbb{D}_N((\alpha_1(\omega),\ldots,\alpha_{n-1}(\omega)))}(e),
\tion
where we agree that $v_{0}(\omega)(e) := e_1$.
Note that $(v_n)_{n=0}^{N-1}$ is $(\calF_n^{\varphi})_{n=0}^{N-1}$-adapted. Moreover, for all $\omega \in \Omega$
it holds that 
\begin{align}
 \left\| 
  \sum_{n=1}^{N} \varphi_n(\omega)  v_{n-1}(\omega)  
 \right\|_{\ell^{\infty}(\mathbb{D}_N)} 
 =
 \sum_{n=1}^{N} |\varphi_n(\omega)|,
\end{align}
which can be exploited in the following way to obtain the second assertion of the lemma:
We let $A_n:= \{ |\varphi_n| \ge c\}$. The we get for
$N\ge 2/\eta$ that 
\equa
      \P \left (  \sum_{n=1}^{N} |\varphi_n| \ge \tfrac{c\eta N}{2} \right )
&\ge& \P \left (\sum_{n=1}^N 1_{A_n} \ge \tfrac{\eta N}{2} \right ) 
 \ge  \P \left (\sum_{n=1}^N 1_{A_n} \ge N\eta -1 \right ) \\
& = &  \sum_{N\eta-1 \le \ell \le N} \eta^\ell (1-\eta)^{N-\ell} \binom{N}{\ell} 
 \ge \frac{1}{2} 
\tion
where we use that $N\eta-1$ is less than or equal the median of the binomial distribution
\cite[Corollary 1]{Kaas:Buhrman:1980}.
On the other hand, for all $\omega \in \Omega$ it holds that
\begin{align*}
 &
 \left\| 
  \sum_{n=1}^{N} \varphi_n'(\omega)  v_{n-1}(\omega)  
 \right\|_{\ell^{\infty}(\mathbb{D}_N)} 
 \\ &
 \leq 
 \sup_{n\in \{1,\ldots,N\}} 
  \left( 
    | \alpha_1(\omega) \varphi_1'(\omega) + \ldots + \alpha_{n-1} (\omega) \varphi_{n-1}'(\omega) | 
    + |\varphi'_{n}(\omega)|
  \right)
 \\ &
  \leq 
 3 \sup_{n\in \{1,\ldots,N\}} 
    | \alpha_1(\omega) \varphi_1'(\omega) + \ldots + \alpha_n(\omega) \varphi_{n}'(\omega) |,
\end{align*}
which by the Burkholder-Davis-Gundy inequality (using that $\varphi'$ has mean zero and is 
independent of $\varphi$) leads to the first assertion of the Lemma.
\end{proof}

\subsection{Stochastic integrals and \texorpdfstring{$\gamma$-}{}radonifying operators}\label{ss:stochint}
Let $X$ be a separable Banach space,
let $(\Omega,\calF,\P,(\calF_t)_{t\in [0,\infty)})$ be a stochastic 
basis\footnote{We do not need the {\em usual conditions} on the stochastic basis in this article.},
and let $W=(W_t)_{t\ge 0}$ be an $(\calF_t)_{t\in [0,\infty)}$-Brownian motion,
i.e., a centered $\R$-valued Gaussian process such that for all $0\leq s \leq t < \infty$ 
it holds that $W_t$ is $\calF_t$-measurable, $W_t - W_s$ is independent from $\calF_s$, and $\E W_s W_t = s$.
A process $H\colon [0,\infty) \times \Omega \to X$ is called 
\emph{simple and predictable} provided that there exists a partition
$0=t_0<\cdots  < t_N< \infty$ and random variables
$v_n\in \calL^{\infty}(\calF_{t_n};X)$, $n\in \{0,\ldots,N-1\}$, such that
for all $t\in [0,\infty)$ it holds that
\[ H(t,\omega) = \sum_{n=1}^N \one_{(t_{n-1},t_n]}(t) v_{n-1}(\omega). \]
For $H\colon [0,\infty) \times \Omega \rightarrow X$ an $X$-valued 
simple predictable process we define 
the stochastic integral $\int_0^\infty H(s) dW(s)$ in the usual way and we define
\[ u_H : \calL^2((0,\infty))\times \Omega \to X 
   \sptext{1}{by}{1}
   u_H (f)(\omega) := \int_0^\infty f(t) H(t,\omega) dt. \]
Note that for all $\omega\in \Omega$ we obtain a finite rank operator $u_H(\omega):\calL^2((0,\infty))\to X$. 
Given a finite rank operator
$T:\calL^2((0,\infty))\to X$ 
one can
define the $\gamma$-radonifying norm 
$\left\| \cdot \right\|_{\gamma(\calL^2((0,\infty));X)}$ by 
\[ \| T \|_{\gamma(\calL^2((0,\infty));X)} := \left \| \sum_{n=1}^\infty \gamma_n T e_n \right \|_{\calL^2(\P',X)}, \]
where $(e_n)_{n\in \N}$ is an orthonormal basis of $\calL^2((0,\infty))$ and 
$(\gamma_n)_{n\in \N}$ is a sequence of independent standard Gaussian random variables
on some probability space $(\Omega',\calF',\P')$. 
The $\gamma$-radonifying norm is independent of the chosen 
orthonormal basis.
For more information about the $\gamma$-radonifying
norm see, for example \cite[Chapter 3]{Pisier:1989} or the survey article~\cite{Neerven:2010}. 
For the relevance of $\gamma$-radonifying norms to the definition of vector-valued 
stochastic integrals, see the definition of and results on $W_p(X)$ in Definition~\ref{def:constants} 
and Theorem~\ref{thm:stochastic_integration}
below, or see~\cite{NeerVerWeis:2015} for more details.

\subsection{Decoupling constants}
In order to state our result, Theorem \ref {thm:stochastic_integration}, we first recall 
that a random variable $f\in \calL^0((\Omega,\calF,\P);X)$ is conditionally symmetric given a sub-$\sigma$-algebra
$\calG$ provided that $\P(\{f\in B\} \cap G) = \P(\{f\in - B\} \cap G)$ for all $B\in \calB(X)$ and 
$G\in\calG$.
\smallskip

For future arguments it is convenient to provide an explicit representation of a probability space with a dyadic filtration, 
as well as an extended space on which we can explicitly define decoupled dyadic martingales:

\begin{setting}\label{setting:PW}
For $\D:=\{-1,1\}^{\N}$ let $r=(r_n)_{n\in\N}$, $r_n\colon \D \rightarrow \{-1,1\}$, satisfy $r_n(e)=e_n$. Let $\calF_\D:=\sigma(r_n:n\in \N)$ and
$(\calF_{\D,n})_{n\in\N} := (\calF_n^r)_{n\in \N}$, and
assume that $\P_{\D}$ is the probability measure on $\calF_\D$ such that
$\P_{\D}(r_1=e_1,\ldots,r_n=e_n) = 2^{-n}$ for $n\in \N$ and $e\in \D$.
Moreover, let $(\D',\calF_{\D}',\P_{\D}',(r_n')_{n\in \N},(\calF_{\D,n}')_{n\in\N})$ be a copy of 
this construction.
\end{setting}
We also introduce some constants:
\smallskip

\begin{definition}\label{def:constants}
Assume the Setting~\ref{setting:PW}, let $X$ be a separable Banach space, and let $p\in (0,\infty)$.
\medskip

\underline{$D_p(X)$:}\label{item:decoupling} 
      Let $D_p(X)\in [0,\infty]$ be the infimum over all $c\in [0,\infty]$ such that for all $N\ge 2$
      it holds that
      \[ \left \| \sum_{n=1}^N r_n v_{n-1} \right \|_{\calL^{p}(\P_\D;X)} 
         \le c  
         \left \| \sum_{n=1}^N r_n' v_{n-1}  \right \|_{\calL^{p}(\P_\D\otimes \P'_\D;X)} \] 
      for all $v_0\in X$ and $v_n:=f_n(r_1,\ldots,r_n)$ with 
      $f_n\colon \{-1,1\}^{n} \rightarrow X$, $n\in \{1,\ldots,N-1\}$.
\smallskip 

\underline{$W_p(X)$:}\label{item:stochInt} Recall the notation introduced 
      in Section~\ref{ss:stochint}. Let $W_p(X)\in [0,\infty]$ be the 
      infimum over all $c\in [0,\infty]$ such that for every
      stochastic basis $(\Omega,\calF,\P,(\calF_t)_{t\in [0,\infty)})$,
      every $(\calF_t)_{t\in [0,\infty)}$-Brownian motion $W$, 
      and every $(\calF_t)_{t\in [0,\infty)}$-simple predictable process $H\colon [0,\infty)\times \Omega\rightarrow X$
      one has that 
      \[ \left \| \int_0^\infty H(s) dW(s) \right \|_{\calL^{p}(\P;X)} 
         \le c \left \|  \| u_H \|_{\gamma(\calL^2((0,\infty));X)} \right \|_{\calL^{p}(\P)} . \]

\underline{$\umd_p^{+,s}(X)$:}\label{item:squarefnc}
      Let $\umd_p^{+,s}(X)\in [0,\infty]$ be the 
      infimum over all $c\in [0,\infty]$ such that for
      every stochastic basis $(\Omega,\calF,\P,(\calF_n)_{n\in \N})$ 
      and every finitely supported
      sequence of $X$-valued random variables $(d_n)_{n=1}^{\infty}$ 
      such that $d_n \in \calL^p(\calF_n;X)$ 
      and $d_n$ is $\calF_{n-1}$-conditionally 
      symmetric for all $n\in \N$ it holds that
      \begin{align}\label{eq:BDG}	
	  \left\|
	    \sum_{n=1}^{\infty}
	      d_n
	  \right\|_{\calL^p(\P;X)}
      \leq
      c
	\left\|
	  \sum_{n=1}^{\infty} 
	    r_n d_n
	\right\|_{\calL^p(\P\otimes \P_{\D};X)}.
      \end{align}
\medskip      
      
\underline{$H_p(X)$:}\label{item:HaarDalt} Let $(h_n)_{n\in \N}$ be the Haar system 
      for $\calL^2((0,1])$ with $\esssup(|h_n|)=1$,
      and let $H_p(X)\in [0,\infty]$ be the 
      infimum over all $c\in [0,\infty]$ such that for
      all finitely supported sequences $(x_n)_{n\in \N}$ in $X$
      one has that
      \[          \left \| \sum_{n=1}^\infty  h_n  x_n \right \|_{\calL^{p}((0,1]);X)} 
         \le c  \left \| \sum_{n=1}^\infty r_n h_n x_n \right \|_{\calL^{p}((0,1]\times \D;X)}. \]
 
\underline{$|H|_p(X)$:}\label{item:HaarD} Let $(h_n)_{n\in \N}$ be as in the definition of $H_p(X)$,
      and let $|H|_p(X)\in [0,\infty]$ be the 
      infimum over all $c\in [0,\infty]$ such that for 
      all finitely supported sequences $(x_n)_{n\in \N}$ in $X$
      one has that
      \[          \left \| \sum_{n=1}^\infty  h_n  x_n \right \|_{\calL^{p}((0,1];X)} 
         \le c  \left \| \sum_{n=1}^\infty r_n |h_n| x_n \right \|_{\calL^{p}((0,1]\times \D;X)}. \]
\end{definition}

\begin{remark}\label{remark:PW}
Note that the processes $(r_n v_{n-1})_{n=1}^{N}$ considered in the definition of $D_p(X)$
above are precisely the $X$-valued dyadic martingales.
\end{remark}

\begin{remark}\rm 
\label{remark:Dp(X)}
Theorem \ref{thm:simple_decoupling_of_vs_full_decoupling} 
(see also Theorem~\ref{thm:simple_decoupling_of_vs_full_decoupling_intro})
and Theorem \ref{thm:char_p-ext_Rademacher}
imply that for a stochastic basis 
$(\Omega,\calF,\P,(\calF_n)_{n=0}^\infty)$, independent standard Gaussian random variables $(g_n)_{n=1}^N$ such that
$g_n$ is $\calF_n$-measurable, but independent from $\calF_{n-1}$,  and 
$v_n \in \calL^p((\Omega,\calF_n,\P);X)$, $n\in \N$, one has
for all $N\in \N$ that
\begin{equation}\label{eqn:decoupling_gauss}
    \left \| \sum_{n=1}^N g_n v_{n-1} \right \|_{\calL^p(\P;X)}
\le D_p(X) \left \| \sum_{n=1}^N g'_n v_{n-1} \right \|_{\calL^p(\P\otimes\P';X)},
\end{equation}
where  $(g'_n)_{n=1}^{N}$ are independent standard Gaussian random variables defined on an auxiliary 
probability space $(\Omega',\calF',\P')$.
\end{remark} 
\smallskip

The aim of this section is to verify the following two theorems.
The relation between the constants introduced in Definition~\ref{def:constants} is given by the first theorem:
\medskip

\begin{theorem}
\label{thm:stochastic_integration}
Let $X$ be a separable Banach space and $p\in (0,\infty)$.
\begin{enumerate}
\item \label{item:extrapolation} 
      If $D_p(X)<\infty$, 
      then $D_q(X)<\infty$ for all $q\in (0,\infty)$. 
\item \label{item:equiv_constants} 
      If $K_{p,2}$ is the constant in the $L^p$-to-$L^2$ Kahane-Khintchine inequality, then
       \[ 
               W_p(X) 
          \leq K_{p,2} D_p(X).
        \]
        Conversely, if $W_p(X) <\infty$, then $D_p(X)<\infty$.
\item \label{item:equal_constants} 
      $ D_p(X) = \umd_p^{+,s}(X) = H_p(X) = |H|_p(X)$.
\end{enumerate}
\end{theorem}
\medskip

The second theorem states that, in a sense, $D_{p}(X)$ is minimal for all $p\in [2,\infty)$.
To state this theorem we introduce the following notation (see also Example~\ref{example:oneDlaw}): 
for $\nu \in \calP(\R)$ we define 
\[ 	
  \calP(\nu,X) := 
  \big\{ 
    \mu \in \calP(X) 
    \colon 
      \exists x\in X \colon \mu (\cdot) 
      = 
      \nu\big(\{r\in \R : rx \in \cdot \} \big)
      \big\}. \]  
\smallskip
\begin{theorem}
\label{thm:D_2(X)_minimal}
Let $X$ be a separable Banach space, $p\in [2,\infty)$, let $\mu\in \calP_p(\R)$ satisfy $\int_\R r d\mu(r)=0$, 
$\sigma^2:=\int_{\R}|r|^2\,d\mu(r)\in (0,\infty)$,
and let $\gamma \in \calP(\R)$ be the standard 
Gaussian law. 
Then the following holds:
\begin{enumerate}
\item \label{it:gamma_in_all} $\calP(\gamma,X) \subseteq  (\calP(\mu,X))_{p\textnormal{-ext}}$. 
\item \label{it:D_2(X)_minimal}
      If there exist $q\in (0,p]$, $c\in (0,\infty)$ such that for every probability space $(\Omega,\calF,\P)$,
      every finitely supported sequence of independent random variables $\varphi = (\varphi_n)_{n\in \N}$ satisfying 
      $\calL(\varphi_n)\in \calP(\mu,X)$, and every 
      $ A_n \in \calF_{n}^{\varphi}$, $n\in \N_0$, it holds that
      \begin{equation}\label{eqn:thm:D_2(X)_minimal}
         \left \| \sum_{n=1}^\infty \varphi_n  1_{A_{n-1}} \right \|_{\calL^{q}(\P;X)}
         \le c  \left \| \sum_{n=1}^\infty \varphi'_n  1_{A_{n-1}} \right \|_{\calL^{q}(\P;X)}
      \end{equation} 
      where $(\varphi'_n)_{n\in \N}$ is an independent copy
      of $(\varphi_n)_{n\in \N}$, then $D_{q}(X)<\infty$.
\end{enumerate}
\end{theorem}
\smallskip

\begin{proof}
\eqref{it:gamma_in_all}  
Let $(\xi_n)_{n\in\N}$ be a sequence of independent, $\mu$-distributed random 
variables, and let $\mu_n := \calL((\sigma \sqrt{n})^{-1}\sum_{k=1}^{n} \xi_k)$.
Observe that $\calL((\sigma \sqrt{n})^{-1}\xi_1) \in \calP(\mu,\R)$.
Moreover, it follows from e.g.~\cite[Theorem 5]{Brown:1970} that $\mu_n \stackrel{w^*}{\rightarrow} \gamma$
and that $\int_{\R} |r|^p \,d\mu_n(r) \rightarrow \int_{\R} |r|^p \,d\gamma(r)$. It thus follows 
from Lemma~\ref{lem:conv_p_moments_vs_LpUI}
that $\gamma \in (\calP(\mu,\R))_{p\textnormal{-ext}}$ and 
hence $\calP(\gamma, X) \subseteq (\calP(\mu,X))_{p\textnormal{-ext}}$.
\medskip

\eqref{it:D_2(X)_minimal}
Observe that $\gamma \in (\calP(\mu,\R))_{p\textnormal{-ext}}$ implies
$\gamma \in (\calP(\mu,\R))_{q\textnormal{-ext}}$ for all $q\in (0,p]$.
Applying Theorem~\ref{thm:simple_decoupling_of_vs_full_decoupling} (Theorem~\ref{thm:simple_decoupling_of_vs_full_decoupling_intro})
implies inequality \eqref{eqn:decoupling_gauss} with $p$ replaced by $q$ and $D_p(X)$ replaced by $c$ from inequality 
$\eqref{eqn:thm:D_2(X)_minimal}$. The part (ii) of the proof below of Theorem \ref{thm:stochastic_integration}
gives exactly $W_q(X)<\infty$, and applying Theorem \ref{thm:stochastic_integration} once more (this time directly)
gives $D_q(X)<\infty$.
\end{proof}
\medskip

For the proof of Theorem~\ref{thm:stochastic_integration}
we shall use the following lemma.  

\begin{lemma}\label{lem:symmapprox}
Let $(\Omega,\calF,\P)$ be a probability space, let
$X$ be a separable Banach space, let $p\in (0,\infty)$, 
let $\calG\subseteq \calF$ be a $\sigma$-algebra, and 
let $f\in L^p(\calF;X)$ be $\calG$-conditionally symmetric. 
Then there exists a sequence of $\calG$-conditionally 
symmetric $\calF$-simple functions $(f_n)_{n\in \N}$
such that 
$\lim_{n\rightarrow \infty} \| f - f_n \|_{\calL^p(\P;X)} = 0$.
\end{lemma}

\begin{proof}
Let $(g_n)_{n\in \N}$ be a sequence of 
$\sigma(f)$-simple functions 
such that 
$\lim_{n\rightarrow \infty} \| f - g_n \|_{L^p(\P;X)} = 0$.
For $n\in \N$ let $m_n\in \N$
and 
$B_{n,k} \in \sigma(f)$,
$x_{n,k} \in X$,
$k\in \{1,\ldots,m_n\}$,
be
such that 
$ 
g_n
=
\sum_{k=1}^{m_n}
x_{n,k}
1_{\{ f \in B_{n,k}\}}.
$ 
Define, for $n\in \N$,
$ 
  f_n
  =
  \inv{2}
  \sum_{k=1}^{m_n}
  x_{n,k}
  (
    1_{ \{f \in B_{n,k}\}}
    -
    1_{ \{- f \in B_{n,k}\}}
  )
$
and observe that $f_n$
is $\calG$-conditionally
symmetric because $f$
is $\calG$-conditionally
symmetric. Moreover, the conditional symmetry of $f$ implies that
$\calL(f) = \calL(-f)$, whence 
\begin{align*}
& \left\|
    f - f_n 
 \right\|_{\calL^p(\P;X)}
\\
  &
 =
 \Bigg\| 
  \tinv{2}
  \Bigg(
  f 
  - 
  \sum_{k=1}^{m_n}
  x_{n,k}1_{ \{f \in B_{n,k}\}}
  \Bigg)
  -
  \tinv{2}
  \Bigg(
  -f
  - 
  \sum_{k=1}^{m_n}
  x_{n,k}1_{\{ - f \in B_{n,k}\}}
  \Bigg)
  \Bigg\|_{\calL^p(\P;X)}
\\
  & \leq
   2^{(\frac{1}{p}-1)^+}
   \Bigg\|
  f 
  - 
  \sum_{k=1}^{m_n}
  x_{n,k}1_{ \{f \in B_{n,k}\}}
  \Bigg\|_{\calL^p(\P;X)}
  = 
   2^{(\frac{1}{p}-1)^+}
  \| f - g_n\|_{\calL^p(\P;X)}.
\end{align*}
\end{proof}

\begin{proof}[Proof of Theorem~\ref{thm:stochastic_integration}]
Part~\eqref{item:extrapolation} is the assertion of Proposition \ref{prop:extrapolation_decoupling_dyadic_martingales}
below.
\par \medskip 

Part \eqref{item:equiv_constants}:
First we check $W_p(X) \leq K_{p,2} D_p(X)$.
For $0=t_0 < \cdots < t_N < \infty$ inequality \eqref{eqn:decoupling_gauss} gives 
\[ \left \| \sum_{n=1}^N (W_{t_n}-W_{t_{n-1}}) v_{n-1} \right \|_{\calL^p(\P;X)}
   \le D_p(X)  \left \| \sum_{n=1}^N (W'_{t_n}-W'_{t_{n-1}}) v_{n-1} \right \|_{\calL^p(\P\otimes \P';X)} \]
for all $\calL_p$-integrable and $\calF_{t_{n-1}}$-measurable random variables $v_{n-1}: \Omega\to X$
where $(W'_t)_{t\ge 0}$ is a Brownian motion defined on an auxiliary basis 
$(\Omega',\calF',\P',(\calF_t')_{t\in [0,\infty)})$.
Exploiting the Kahane-Khintchine inequality gives that
\begin{multline*}
       \left \| \sum_{n=1}^N (W'_{t_n}-W'_{t_{n-1}}) v_{n-1} \right \|_{\calL^p(\P\otimes \P';X)} \\
   \le K_{p,2} 
       \left ( \int_{\Omega} \left ( \int_{\Omega'}
       \left \| \sum_{n=1}^N (W'_{t_n}(\omega')-W'_{t_{n-1}}(\omega')) v_{n-1}(\omega) \right \|_X^2 d\P'(\omega') 
       \right )^\frac{p}{2} d\P(\omega) \right )^\frac{1}{p}.
\end{multline*}
For $H := \sum_{n=1}^{N} 1_{(t_{n-1},t_n]}v_{n-1}$ the result follows by the known relation
\[ \left ( \int_{\Omega'}
       \left \| \sum_{n=1}^N (W'_{t_n}(\omega')-W'_{t_{n-1}}(\omega')) v_{n-1}(\omega) \right \|_X^2 d\P'(\omega') 
       \right )^\frac{1}{2} = \| u_H(\omega) \|_{\gamma(\calL^2((0,\infty));X)}. \]
Conversely, let us assume that $W_p(X)<\infty$. Then Lemma \ref{lemma:upper_decoupling_finite_cotype} implies that $X$ has finite cotype. 
Thus the proof of \cite[Theorem 2.2]{Veraar:2007} guarantees that $D_p(X)<\infty$.
Here we exploit the fact that \cite[Proposition 9.14]{Ledoux:Talagrand:1991}
works (in their notation) with the parameter $r\in (0,1)$ as well: one starts
on the left-hand side with $\calL^r$, estimates this by $\calL^1$, applies
\cite[Proposition 9.14]{Ledoux:Talagrand:1991}, and uses \cite[Proposition 4.7]{Ledoux:Talagrand:1991}  
(Khintchine's inequality for a vector valued Rademacher series) to change $\calL^1$ back to $\calL^r$ on the right-hand side.
\par \medskip

Part \eqref{item:equal_constants} is divided into several steps:
\medskip

{\sc Proof} of $H_p(X) = |H|_p(X)$:
This is immediate as, in the notation of Definition \ref{def:constants}, the sequences 
$(r_nh_n)_{n=1}^N$ and $(r_n|h_n|)_{n=1}^N$ have the same distribution for all $N\in \N$.
\par\medskip

{\sc Proof} of $H_p(X) \leq \umd_p^{+,s}(X)$:
This inequality follows by taking $d_n = h_n x_n$, 
and $\calF_n = \sigma(h_0,\ldots,h_n)$, $n\in \N_0$,
in the definition of $\umd_p^{+,s}(X)$.
\par \medskip

{\sc Proof} of  $D_p(X) = |H|_p(X)$: We use 
the following standard construction.
Let $N\in \{2,3,\ldots\}$, let $\Delta_N := \{(n,k): n=0,\ldots,N-1, k=0,\ldots,2^n-1 \}$,
$A_{0,0}:=\D$, and let
\begin{align*}
  A_{n,k}
  & :=
  \Bigg\{ 
    \sum_{j=1}^{n}
     (r_j+1) 2^{ j-2 }
    =
    k
  \Bigg\}
  \in \calF_{\D,n} \sptext{1}{for}{1} 
    (n,k) \in \Delta_N\setminus\{(0,0)\}.
\end{align*}
If we set $h_{n,k} := r_{n+1} 1_{A_{n,k}}$
for $(n,k) \in \Delta_N$, then $(h_{n,k})_{(n,k)\in \Delta_N}$ has the same distribution as
$(h_{\ell})_{\ell=1}^{2^N-1}$.
Let $v_0\in X$ and $f_n\colon \{-1,1\}^n\rightarrow X$, $n\in \{1,\ldots,N-1\},$ be given
and let $v_n\colon \D \rightarrow X$, $n\in \{1,\ldots,N-1\}$ satisfy $v_{n}=f_n(r_1,\ldots,r_n)$.
Let $x_{n,k}\in X$ be the value of 
$v_n$ on $A_{n,k}$ for $(n,k)\in \Delta_N$. Then one has 
\begin{equation}\label{eq:equiv_HD}
 \sum_{n=1}^N r_n v_{n-1}
    = \sum_{(n,k)\in \Delta_N} h_{n,k} x_{n,k} 
    \quad \textnormal{and} \quad 
 \sum_{n=1}^N r'_n v_{n-1}
    = \sum_{n=0}^{N-1} \sum_{k=0}^{2^n-1} r'_{n+1} |h_{n,k}| x_{n,k}. 
\end{equation}
Next, observe that $(r'_{n+1}|h_{n,k}|)_{(n,k)\in \Delta_N}$ and $(r'_{n,k}|h_{n,k}|)_{(n,k)\in \Delta_N}$ 
have the same distribution, where 
$(r'_{n,k})_{(n,k)\in \Delta_N}$ is  a Rademacher system on an auxiliary probability space. 
This implies that $D_p(X) \leq |H|_p(X)$. Similarly, given $(x_{n,k})_{(n,k)\in \Delta_N}$,
we can construct $(v_n)_{n=0}^{N-1}$ such that~\eqref{eq:equiv_HD} holds. This implies $D_p(X) \geq |H|_p(X)$.
\par \medskip

{\sc Proof} of $\umd_p^{+,s}(X) \leq D_p(X)$:
With Lemma~\ref{lem:symmapprox} we approximate each $d_n$ in
$\calL^p(\calF_n;X)$ so that we may assume that the $d_n$ take finitely many values only.
Let
\[ \varepsilon_0 := \inf \{ \| d_n(\omega) \|_X : n=1,\ldots,N, \, \omega \in \Omega, \,
                      d_n(\omega)\not = 0 \} > 0 \]
where $\inf \emptyset := 1$. Take an $x\in X$ with $0<\|x\|_X<\varepsilon_0$
and the Rademacher sequence $(r'_n)_{n=1}^N$ given by the Setting \ref{setting:PW}.
If we define
\[  \tilde d_n (\omega,e'):= d_n(\omega) + r'_n(e')  x, 
    \sptext{1}{then}{1}
     \tilde d_n(\omega,e') \not = 0 \]
for all $(\omega,e')\in \Omega\times \D'$,
then $\tilde d_n$ is conditionally symmetric
given the $\sigma$-algebra $\calF_{n-1}\otimes \calF_{\D,n-1}'$. 
Because we may let $\|x\|\downarrow 0$
it suffices to verify the statement for $(\tilde d_n)_{n=1}^N$ or, in other words, 
we may assume without loss of generality that for all $n\in \N$ the range of 
$d_n$ is a finite set that does not contain $0$.
\smallskip

Note that by removing all (i.e., at most finitely many) atoms of measure zero in the $\sigma$-algebra $\calF_N^d$
and `updating' the definition of $(d_n)_{n=1}^{N}$ accordingly, we may assume 
that the filtration $(\calF_n^d)_{n=1}^N$ has the property that 
 $\calF_n^d$ is generated by finitely many atoms of positive measure.
\smallskip

Bearing in mind that 
for all $n\in \{1,\ldots,N\}$ the random variable $d_n$ takes only finitely many non-zero values, each with positive probability, 
one may check 
that for every atom $A\in \calF_{n-1}^d$, $n\in \{1,\ldots,N\}$,
there exist disjoint sets $A^{+}, A^{-} \in \calF_{n}^d$ such that $A = A^+ \cup A^{-}$,
$\P(A^+)=\P(A^-)$, 
and $\calL(d_{n} | A^+ ) = \calL(-d_{n} | A^{-})$.
Now we introduce a Rademacher sequence $(\rho_n)_{n=1}^N$, $\rho_n:\Omega \to \{-1,1 \}$,
defined as follows: for each atom $A$ of $\calF_{n-1}^d$ 
we set $\rho_n|_{A^+} \equiv 1$, and $\rho_n|_{A^-}\equiv -1$,
where $A^+$ and $A^{-}$ form a partition of $A$ as described above.
Moreover, we let $v_n := \rho_n d_n$ so that
$d_n = \rho_n v_n$. 
By construction, $\rho_n$ is independent from $\calF_{n-1}^d\vee \sigma(v_n)$.
It follows from the definition of $D_p(X)$ and Theorem~\ref{thm:simple_decoupling_of_vs_full_decoupling_intro} 
(see also Example~\ref{example:decoupled_sequence})
that 
\begin{align*}
&\Bigg\|
  \sum_{n=1}^{N} d_n
\Bigg\|_{\calL^p(\P;X)}
= 
\Bigg\|
  \sum_{n=1}^{N} \rho_n v_n
\Bigg\|_{\calL^p(\P;X)}
\leq 
D_p(X)
\Bigg\|
  \sum_{n=1}^{N} r'_n v_n
\Bigg\|_{\calL^p(\P\otimes\P'_{\D};X)}
\\ &
=
D_p(X)
\Bigg\|
  \sum_{n=1}^{N} r'_n r_n v_n
\Bigg\|_{{\calL}^p(\P\otimes\P'_{\D};X)}
=
D_p(X)
\Bigg\|
  \sum_{n=1}^{N} r'_n d_n
\Bigg\|_{\calL^p(\P\otimes\P'_{\D};X)}.
\end{align*} 
\end{proof}


\pagebreak
\section{Dyadic decoupling in \texorpdfstring{$c_0$ and $\ell_K^\infty$}{sequence spaces}}
\label{sec:decoupling_in_c0}

The aim of this section is to determine the asymptotics of  $D_2(\ell_\infty^K)$.
For this purpose we begin with the proof of Theorem \ref{thm:upper_decoupling_c0_intro}:

\begin{proof}[Proof of Theorem \ref{thm:upper_decoupling_c0_intro}]
First we fix $K,N\in \N$ and restrict $D_{\sqrt{\log}}$ to
$D_{\sqrt{\log}}:\ell_K^\infty \to \ell_K^\infty$. We take the Rademacher 
variables $(r_n)_{n=1}^N$ and $(r_n')_{n=1}^N$ from Setting ~\ref{setting:PW},
let $v_0\in \ell_K^\infty$, $f_n:\{-1,1\}^n\to \ell_K^\infty$
and $v_n:=f_n(r_1,\ldots,r_n)$ for $n\in \{1,\ldots,N-1\}$, and
\begin{align*}
d_n  =(d_n^{(1)},\ldots,d_n^{(K)}) & := r_n v_{n-1} : \D \to \ell_\infty^K
       \sptext{1}{for}{1}
       n\in \{1,\ldots,N\},\\
M_n^k &:=d_1^{(k)} + \cdots + d_n^{(k)}  \sptext{1.5}{for}{1}
       n\in \{1,\ldots,N\},\\
S_2(M^{(k)}) &:= \left (\sum_{n=1}^N |d_n^{(k)}|^2\right )^\frac{1}{2}. 
\end{align*}
We add $d_0=M_0=(M_0^{(1)},\ldots,M_0^{(K)}) \equiv 0 \in \ell_\infty^K$ to be in accordance with the literature.
By an extension of the Azuma-Hoeffding inequality \cite[Lemma 4.3]{Hitczenko:1990}
it is known  (see \cite[equation (4)]{Geiss:1997}) that there is a $c_1\in (0,\infty)$ with
\[ \| (M_n^{(k)})_{n=0}^N \|_{\bmo_{\psi_2}} \le c_1 \| S_2(M^{(k)}) \|_{\calL^\infty(\D_\P)} 
   \sptext{1}{for}{1}
   k\in \{1,\ldots,K\}, \]
where $\psi_2(t):=t^2$ and the $\bmo_{\psi}$-spaces are introduced in  \cite[p. 239]{Geiss:1997}.
From  \cite[Theorem 1.5]{Geiss:1997} we also know that, for $\alpha_k := (1+\log k)^{-1/2}$, 
\[     \left \| \left (\sup_{k\in \{1,\ldots,K\}} \alpha_k  |M_n^{(k)}| \right )_{n=0}^N \right \|_{\bmo_{\psi_2}} 
   \le c_2  \sup_{k\in \{1,\ldots,K\}} \left \| (M_n^{(k)})_{n=0}^N \right \|_{\bmo_{\psi_2}}
\]
for some $c_2\in (0,\infty)$, where one has to observe (in the notation of \cite{Geiss:1997}) that $\psi_2=\overline{\psi_2}$
by \cite[Example 4.3]{Geiss:1997}.
If we introduce the operators $A,B:E_N\to \calL^0(\D)$, where $E_N$ consists of all 
sequences $d=(d_n)_{n=0}^N$ as above (note that $K,N\in \N$ are fixed),
\[ A(d) := \sup_{k\in \{1,\ldots,K\}} \alpha_k \left | M_N^{(k)} \right |
   \sptext{1}{and}{1}
   B(d) := \sup_{k\in \{1,\ldots,K\}} S_2(M^{(k)}), \]
then the two previous estimates yield to 
$\left \| A(d) \right \|_{\bmo_{\psi_2}} \le c_1 c_2 \| B(d) \|_{\calL^\infty(\P_\D)}$. 
By \cite[Theorem 1.7]{Geiss:1997} we deduce
$\left \| A(d) \right \|_{\calL^2(\P_\D)} \le c  \| B(d) \|_{\calL^2(\P_\D)}$ and 
\equa
 \left  \| \sup_{k\in \{1,\ldots,K\}} \frac{\left | \sum_{n=1}^N d_n^{(k)} \right |}
                                       {\sqrt{1+\log(k)}} 
    \right \|_{\calL^2(\P_\D)} 
&\le& c \left \| \sup_{k\in \{1,\ldots,K\}} \left ( \sum_{n=1}^N |d_n^{(k)}|^2\right )^\frac{1}{2} \right \|_{\calL^2(\P_\D)} \\
& = & c \left \| \sup_{k\in \{1,\ldots,K\}} \left \| \sum_{n=1}^N r'_nd_n^{(k)} \right \|_{\calL^2(\P_\D')} \right \|_{\calL^2(\P_\D))} \\
&\le& c \left \| \sup_{k\in \{1,\ldots,K\}} \left | \sum_{n=1}^N r'_nd_n^{(k)} \right | \right \|_{\calL^2(\P_\D\otimes\P_\D')}\\
& = & c \left \| \sum_{n=1}^N r'_nd_n \right \|_{\calL^2(\P_\D\otimes\P_\D';\ell^\infty_K)}.
\tion

The statement follows with $K\to \infty$ by monotone convergence
and by observing that $\| (\xi_k)_{k\in \N}\|_{c_0} = \sup_{K\in \N} \| (\xi_k)_{k=1}^K\|_{\ell_\infty^K}$
and that the constants $c_1,c_2,c$ do not depend on $K$ and $N$.
\end{proof}

Corollary \ref{cor:decouple_ell_infty^K} below
complements \cite[Proposition 8.6.8, Example 8.6.13]{Pietsch:Wenzel:1998} where it is shown that
for every Banach space $X$, 
$N\in \{2,3,\ldots\}$, $v_0\in X$, and $v_n:=f_n(r_1,\ldots,r_n)$ with 
$f_n\colon \{-1,1\}^{n} \rightarrow X$, $n\in \{1,\ldots,N-1\},$ it holds that
\[ \left \| \sum_{n=1}^N r_n v_{n-1} \right \|_{\calL^2(\P_{\D};X)}
   \le \sqrt{N} \left \| \sum_{n=1}^N r'_n v_{n-1} \right \|_{\calL^2(\P_{\D}\otimes\P_{\D}'; X)}. \]
Here the decoupling constant $\sqrt{N}$ {\em depends on the length of the dyadic martingale} which we are
going to  decouple. In Corollary \ref{cor:decouple_ell_infty^K} we strengthen this result 
for $X=\ell_{2^N}^\infty$ by obtaining the same asymptotic upper bound, but now independent from the length 
of the martingale.
\smallskip

\begin{corollary}
\label{cor:decouple_ell_infty^K}
There is a constant $c\in [1,\infty)$ such that for all $K\in \N$ one has
\[ \frac{1}{c} \sqrt{1+\log(K)} \le D_2(\ell^\infty_K) \le c  \sqrt{1+\log(K)}. \]
\end{corollary}

\begin{proof}
The lower bound is known (see e.g. \cite{Garling:1990}, \cite{Veraar:2007}, or Lemma \ref{lemma:upper_decoupling_finite_cotype}).
The upper bound follows directly from Theorem~\ref{thm:upper_decoupling_c0_intro}, because (with the notation 
of the definition of $D_2(X)$ and with $a_1,\ldots,a_K$ being the unit vectors in $\ell_K^1$)
\begin{align*}
       \frac{\left  \| \sum_{n=1}^N r_nv_{n-1}  \right \|_{\calL^2(\P_\D;\ell^\infty_K)}}{\sqrt{1+\log(K)}} 
&\le \left  \| \sup_{k\in \{1,\ldots,K\}} \frac{\left | \sum_{n=1}^N \langle r_nv_{n-1},a_k \rangle \right |}
        {\sqrt{1+\log(k)}} \right \|_{\calL^2(\P_\D)} \\
&\le D_2 \left (D_{\sqrt{\log}} \right ) \left \| \sum_{n=1}^N r'_n v_{n-1} \right \|_{\calL^2(\P_\D\otimes \P_\D';\ell^\infty_K)}.
\qedhere
\end{align*}
\end{proof}

Given $\delta \in (0,1]$, we say that a sequence $(a_k)_{k\in I}\subset X'$, where $X'$ is the norm-dual of $X$
and where $I=\{1,\ldots,K\}$ or $I=\N$, is $\delta$-norming, provided that $\|a_k\|_{X'}=1$ and 
$\sup_{k\in I} |\langle x,a_k \rangle| \ge \delta \|x\|_X$ for all $x\in X$.
\medskip

\begin{corollary}
Let $X$ be a separable Banach space, let $(a_\ell)_{\ell\in I}\subset X'$ be $\delta$-norming
for some $\delta \in (0,1]$, and let 
$\calP=\calP^{\textnormal{Rad}}(X) := \{ \frac{1}{2} \delta_x + \frac{1}{2} \delta_{-x}: x\in X \}$.
For every stochastic basis $(\Omega,\calF,\P,\F)$ with $\F=(\calF_n)_{n\in \N_0}$ and        
every finitely supported $(d_n)_{n\in \N} \in \adapt_2(\Omega,\F;X,\calP_{2\textnormal{-ext}})$
it holds that
\[   
   \left \| \sum_{n=1}^{\infty} d_n \right \|_{\calL^2(\P;X)}
   \le \frac{D_2(D_{\sqrt{\log}})}{\delta} \left \| \sup_{k\in I} \left [ \sqrt{1+\log(k)} 
       \left | \left \langle \sum_{n=1}^{\infty}  e_n , a_k \right \rangle  \right | 
                              \right ] \right \|_{\calL^2(\P;X)}, 
\]
provided that $(e_n)_{n\in \N}$ is an $\F$-decoupled tangent sequence of $(d_n)_{n\in \N}$.
\end{corollary}

\begin{proof}
Theorems~\ref{thm:simple_decoupling_of_vs_full_decoupling_intro} and~\ref{thm:upper_decoupling_c0_intro}
imply that for all $K\in \N$ and $I_K :=  I \cap \{1,\ldots,K\}$ we have that
\equa
      \left \| \sum_{n=1}^{\infty} d_n \right \|_{\calL^2(\P;X)} \!\!\!
&\le& \frac{1}{\delta} 
      \sup_{K\in \N} \left \| \sup_{k \in I_K} \frac{1}{\sqrt{1+\log(k)}}
      \left | \left \langle \sum_{n=1}^{\infty} d_n , \sqrt{1+\log(k)} a_k \right \rangle \right |  \right \|_{\calL^2(\P;X)} \\
&\le& \frac{D_2(D_{\sqrt{\log}})}{\delta} 
      \sup_{K\in \N} \left \| \sup_{k \in I_K} 
      \left | \left \langle \sum_{n=1}^{\infty} e_n , \sqrt{1+\log(k)} a_k \right \rangle \right |  \right \|_{\calL^2(\P;X)} \\
& = & \frac{D_2(D_{\sqrt{\log}})}{\delta} 
      \left \| \sup_{k \in I} \left [ \sqrt{1+\log(k)}
      \left | \left \langle \sum_{n=1}^{\infty} e_n ,  a_k \right \rangle \right | \right ] \right \|_{\calL^2(\P;X)}.
\tion
\end{proof}


\section{Dyadic decoupling and chaos}
\label{sec:chaos}

For the dyadic decoupling considered in Sections \ref{sec:stochInt} 
and \ref{sec:decoupling_in_c0} we show in 
Theorem \ref{theorem:equivalence_chaos} below
that it is also natural to allow decoupling with 
respect to a different distribution (cf.~Definition~\ref{def:var_decoup}), i.e.\ one can vary the distribution 
used for the upper decoupling on the right side without changing any requirements 
on the underlying Banach space $X$. Moreover, Theorem \ref{theorem:equivalence_chaos} shows that a Banach space $X$ allows 
for decoupling with respect to a distribution in some fixed chaos obtained from  $\calP_{2-{\rm ext}}^{\rm Rad}$ 
if and only if it allows for decoupling with respect the first chaos, i.e., if and only if $D_2(X)<\infty$.

\begin{definition}\label{def:var_decoup}
Let $p\in (0,\infty)$, $\mu,\nu \in \calP_p(\R)$, and $X$ be a separable Banach space. Then
$D_p(X;\mu,\nu)$ is the infimum over all $c\in [0,\infty]$ such that 
\[ \left \| \sum_{n=1}^N \varphi_n v_{n-1} \right \|_{\calL^p(\P;X)}
   \le c \left \| \sum_{n=1}^N \psi_n' v_{n-1} \right \|_{\calL^p(\P\otimes \P';X)} \]
for all $N\ge 2$, where 
$\varphi_1,\ldots,\varphi_N$ are independent random variables on $(\Omega,\calF,\P)$ with law $\mu$,
$\psi_1',\ldots,\psi_N':\Omega'\to \R$ are independent random variables on $(\Omega',\calF',\P')$ with law $\nu$,
$v_0\in X$, and $v_n=f_n(\varphi_1,\ldots,\varphi_n)\in \calL^p(\P;X)$ for $n\in \{1,\ldots,N-1\}$
and Borel functions $f_n:\R^n\to X$.
\end{definition}

\begin{definition}
For $\alpha\in \calP(\R)$ and $L\in \N$ a measure $\mu\in \calP(\R)$ belongs to the
$L$-th chaos ${\calC}_L(\alpha)$ if there are an integer $K\ge L$ and 
$(a_{\ell_1,\ldots,\ell_L})_{1\le \ell_1 < \cdots < \ell_L\le K} \subseteq \R$
with
\begin{equation}\label{eq:mu_chaos}
 \mu = \calL\left ( \sum_{1\le \ell_1 < \cdots < \ell_L\le K}
                      a_{\ell_1,\ldots,\ell_L} \varphi_{\ell_1} \cdots  \varphi_{\ell_L}  \right ), 
\end{equation}
where $\varphi_1,\ldots,\varphi_K:\Omega\to \R$ are independent with law $\alpha$.
\end{definition}

\begin{theorem}
\label{theorem:equivalence_chaos}
For a separable Banach space $X$, $p\in [1,\infty)$, $L\in \N$, $\alpha \in \calP_{p-{\rm ext}}^{\rm Rad}$, 
and $\delta_0 \not = \mu\in \calC_L(\alpha)$ the following assertions are equivalent:
\begin{enumerate}
\item \label{item:1:theorem:equivalence_chaos}
      $D_{p}(X)<\infty$.
\item \label{item:2:theorem:equivalence_chaos}
      There exists a symmetric $\nu\in \bigcap_{r\in (0,\infty)}\calP_r(\R)$ such that $D_{p}(X;\mu,\nu)<\infty$.
\item \label{item:3:theorem:equivalence_chaos}
      $D_{p}(X;\mu, \frac{1}{2}(\delta_{-1}+\delta_1))<\infty$.
\item \label{item:4:theorem:equivalence_chaos}
      $D_{p}(X;\mu,\mu)<\infty$.
\end{enumerate}
\end{theorem}

\begin{proof}
$\eqref{item:1:theorem:equivalence_chaos} \Rightarrow \eqref{item:2:theorem:equivalence_chaos}$
Let $\mu$ be represented by~\eqref{eq:mu_chaos} (i.e., we are given $K,L\in \N$ satisfying 
$K>L$ and $(a_{\ell_1,\ldots,\ell_L})_{1\leq \ell_1<\ell_2<\ldots<\ell_L\leq K} \in \R$). 
If we define $\nu$ by setting
\[ \nu := \calL\left ( \sum_{1\le \ell_1 < \cdots < \ell_L\le K}
                      a_{\ell_1,\ldots,\ell_L} \varphi_{\ell_1}^1 \cdots  \varphi_{\ell_L}^L \right ), \]
where $\varphi_k^\ell:\Omega\to \R$, $(\ell,k)\in \{1,\ldots,L\}\times \{1,\ldots,K\}$ 
are independent and symmetric random variables with law $\alpha$, then it follows immediately by iteratively 
applying Theorem~\ref{thm:simple_decoupling_of_vs_full_decoupling} 
(or Theorem~\ref{thm:simple_decoupling_of_vs_full_decoupling_intro})
that $D_{p}(X;\mu,\nu) \le (D_{p}(X))^L <\infty$. Note that
that $\nu$ is symmetric and that, in general, $\mu\neq \nu$.
\smallskip

$\eqref{item:2:theorem:equivalence_chaos} \Rightarrow \eqref{item:3:theorem:equivalence_chaos}$
By Lemma \ref{lemma:upper_decoupling_finite_cotype} the assumption 
$D_{p}(X;\mu,\nu)<\infty$ implies that $X$ has finite cotype. This enables us to apply 
\cite[Proposition 9.14]{Ledoux:Talagrand:1991} to deduce that
\[ \left \| \sum_{n=1}^N \psi_n' v_{n-1} \right \|_{\calL^{p}(\P\otimes \P';X)}
   \le c  \left \| \sum_{n=1}^N r_n' v_{n-1} \right \|_{\calL^{p}(\P\otimes \P_\D';X)}, \]
where the constant $c>0$ depends only on $p$, the cotype of $X$, and $\nu$.
\smallskip

$\eqref{item:3:theorem:equivalence_chaos} \Rightarrow \eqref{item:4:theorem:equivalence_chaos}$
Let $\varphi_1,\ldots,\varphi_N$ be independent random variables on $(\Omega,\calF,\P)$ with law $\mu$,
$v_0\in X$, and $v_n=f_n(\varphi_1,\ldots,\varphi_n)\in \calL^p(\P;X)$ for $n\in \{1,\ldots,N-1\}$
and Borel functions $f_n:\R^n\to X$. We use \cite[Lemma 4.5]{Ledoux:Talagrand:1991} to conclude
\begin{align*}
&\left \| \sum_{n=1}^N \varphi_n v_{n-1} \right \|_{\calL^{p}(\P;X)} 
\le D_{p}
  \left (X;\mu, \tfrac{1}{2}(\delta_{-1}+\delta_1)\right ) 
      \left \| \sum_{n=1}^N r_n' v_{n-1} \right \|_{\calL^{p}(\P\otimes \P_\D';X)} \\
&\le \big|\E|\varphi_1'-\varphi_1''|\big|^{-1} D_{p}\left (X;\mu, \tfrac{1}{2}(\delta_{-1}+\delta_1)\right )
      \left \| \sum_{n=1}^N (\varphi_n'-\varphi_n'') v_{n-1} \right \|_{\calL^{p}(\P\otimes \P'\otimes \P'';X)}  \\
&\le 2\big|\E|\varphi_1'-\varphi_1''|\big|^{-1} D_{p}\left (X;\mu, \tfrac{1}{2}(\delta_{-1}+\delta_1)\right ) 
      \left \| \sum_{n=1}^N \varphi_n' v_{n-1} \right \|_{\calL^{p}(\P\otimes \P';X)} \\
\end{align*}
where $\varphi_1',\ldots,\varphi_N':\Omega'\to X$ and  $\varphi_1'',\ldots,\varphi_N'':\Omega''\to X$ are independent copies
of $\varphi_1,\ldots,\varphi_N$. The implication
$\eqref{item:4:theorem:equivalence_chaos} \Rightarrow \eqref{item:1:theorem:equivalence_chaos}$
follows from Theorem \ref{thm:D_2(X)_minimal}.
\end{proof}


\section{Relations to other constants and open problems}
\label{sec:relations_problems}

\subsection{Randomized UMD constants}
\label{subsec:random_umd}
Garling~\cite{Garling:1990} introduced upper- and lower randomized UMD inequalities. In our notation 
the corresponding constants would be $\umd_p^+(X)$ and $\umd_p^-(X)$, where $\umd_p^+(X)$ 
is defined as $\umd_p^{+,s}(X)$ but {\em without} the condition {\em conditionally symmetric},
and $\umd_p^-(X)$ is the constant for the reverse inequality. In 
general, this leads to a different behavior of the optimal constants. For example,  
it follows from  Hitczenko ~\cite[Theorem 1.1]{Hitczenko:1994} that $\sup_{p\in [2,\infty)} D_p(\R) < \infty$.
But, as outlined in \cite[p. 348]{CoxVeraar:2011}, one has 
$\umd_p^+(\R) \succeq \sqrt{p}$ as $p\to \infty$ by combining the result of Burkholder ~\cite[Theorem 3.1]{Burkholder:1988}  about the optimal behavior 
of the constant in the square function inequality and the behavior of the constant in the Khintchine inequality for Rademacher variables.

\subsection{General decoupling constants}
\begin{enumerate}[(a)]
\item McConnell proved in \cite[Theorem 2.2]{McConnell:1989} (see also \cite{Hitczenko:notes}) that a Banach space $X$ is a UMD space 
      if and only if tangent sequences have equivalent moments.
\item Cox-Veraar~\cite[Example 4.7]{CoxVeraar:2011} proved that the upper decoupling of martingales is valid in $L_1$, note that  $L_1$ is a pro-type
      of a non-UMD space.
\end{enumerate}

\subsection{Stochastic integration}
\label{subsec:Stochastic_integration}
In the development of stochastic integration theory in Banach spaces (as presented 
in e.g.~\cite{vanNeervenVeraarWeis:2007} and~\cite{CoxVeraar:2011}) 
the issue regarding the undesirable assumption on the 
filtration in the work by Garling~\cite{Garling:1986} was known to 
the authors. In those articles the problem was circumvented 
in two ways: 
\medskip

\begin{enumerate}[(a)]
\item In~\cite[Lemma 3.4]{vanNeervenVeraarWeis:2007}, a decoupling argument due to 
      Montgomery-Smith~\cite{MontgomerySmith:1998} is used
      to prove that for $p\in (1,\infty)$ it holds that
      $W_p(X) \leq \beta_p(X)$, where $\beta_p(X)$ is the $L^p$-UMD constant of $X$. 
      This approach does not cover $p\in (0,1]$, and moreover the UMD property does not seem to be 
      natural in this setting as (for example) it excludes $L_1$, but $W_p(L_1)<\infty$ 
      for all $p\in (0,\infty)$ (see also~\cite{CoxVeraar:2011}).
\item \label{subsec:Stochastic_integration:b}
      In ~\cite[Theorem 5.4]{CoxVeraar:2011} it is observed that $W_p(X)<\infty$ if $D_p^{\rm gen}(X)<\infty$, where 
      $D_p^{\rm gen}(X)$ is the infimum over all $c\in [0,\infty]$ such that 
      \[  \left\| \sum_{n=1}^{\infty} d_n \right\|_{\calL^p(\P;X)} 
          \leq  c \left\| \sum_{n=1}^{\infty} e_n \right\|_{\calL^p(\P;X)} \]
      whenever $(e_n)_{n\in\N}$ is an $\F$-decoupled tangent sequence of a
      finitely supported $X$-valued $\F$-adap\-ted sequence of random variables $(d_n)_{n\in \N}$. 
\end{enumerate}

\subsection{Open problems}
The approach in Section \ref{subsec:Stochastic_integration}\eqref{subsec:Stochastic_integration:b}
is unsatisfactory in the following respect: The Brownian motion is naturally connected to Rademacher sequences,
but  the definition of $D_p^{\rm gen}(X)$ had to use more general distributions because of the progressive enlargement of the
filtration, which was explicitly allowed in \cite{CoxVeraar:2011}. The handling of $D_p^{\rm gen}(X)$ is 
more involved than the handling of $D_p(X)$ in certain situations. One example concerns extrapolation results 
in the sense of Proposition~\ref{prop:extrapolation_decoupling_dyadic_martingales}. Here 
it has been  one of the contributions of~\cite{CoxVeraar:2011} to prove extrapolation properties 
for the general decoupling, which is more involved than 
our extrapolation in  Proposition~\ref{prop:extrapolation_decoupling_dyadic_martingales}.
So, the systematic approach of this article revealed the following open questions:

\begin{enumerate}[(a)]
\item Is there a constant $c_p>0$, depending at most on $p$, such that,
      for all Banach spaces $X$,
      \[ D^{\rm gen}_p(X) \le c_p D_p(X)? \]
\item Or one might ask without requiring a uniform estimate: Is that true 
      that $D_p(X)<\infty$ implies $D^{\rm gen}_p(X)<\infty$?

\item Another problem (which was not in the focus of this article) would be: 
      What is the relation between  $D_p(X)$ and $\umd_p^+(X)$ from 
      Subsection \ref{subsec:random_umd}?
\end{enumerate}


\appendix

\section{Proof of Theorem \ref{thm:simple_decoupling_of_vs_full_decoupling}}

\label{sec:main_proof}
In order to prove Theorem~\ref{thm:simple_decoupling_of_vs_full_decoupling}
it suffices, by Remark~\ref{remark:trivial_implication}, to prove that the second
statement in the theorem implies the first. 
In Section \ref{subsec:Proof:simple_decoupling_of_vs_full_decoupling} we will show that
this implication is an immediate corollary of Propositions~\ref{prop:extend}
and~\ref{prop:simple_to_Lp} below. 
Propositions~\ref{prop:extend} demonstrates that a 
progressive extension of the underlying filtration and an extension of the
set $\calP$ can be carried out. Proposition~\ref{prop:simple_to_Lp} uses 
Corollary \ref{cor:factor_through_subsigmaalgebra} to show that one can pass from relatively
simple sequences of random variables 
to the sequences considered in~\eqref{item:general_decoupling}
in Theorem~\ref{thm:simple_decoupling_of_vs_full_decoupling}. 


\subsection{Progressive enlargement of the filtration and extension of \texorpdfstring{$\calP$}{the set of admissible measures}}
The main result of this subsection is

\begin{proposition}\label{prop:extend}
Let $X$ be a Banach space,
let $\Phi \in C(X \times X;\R)$
be such that there exist 
constants $C,p\in (0,\infty)$ for which it holds that
\begin{equation}\label{eq:property_Phi}
  | \Phi(x,y) |
  \leq
  C( 1 + \| x \|_X^p + \| y \|_X^p)
\end{equation}
for all $ (x,y) \in X\times X$, and
let $\calP \subseteq \mathcal{P}_p(X)$ 
be a set of probability measures satisfying $\delta_0 \in \calP$.
Let $\calP_{p\textnormal{-ext}}$ be defined as in
Definition~\ref{definition:p-measure_extension}.
Then the following assertions are equivalent:
\begin{enumerate}
\item \label{item:v_adapted}
For every $N\in \N$, every stochastic basis $(\Omega,\calF,\P,(\calF_n)_{n=0}^N)$,
every sequence of independent random variables
$(\varphi_n)_{n=1}^N$ with $\calL(\varphi_n) \in \calP_{p\textnormal{-ext}}$
such that $\varphi_n$ is $\calF_{n}$ measurable
and independent of $\calF_{n-1}$,
and for all $A_n \in \calF_{n}$, $n \in \{0,\ldots,N-1\}$,
it holds that
\begin{equation*}
  \E
  \Phi\bigg(
      \sum_{n=1}^{N}
	\varphi_n 1_{A_{n-1}}
    ,
      \sum_{n=1}^{N}
	\varphi'_n 1_{A_{n-1}}
  \bigg)	
  \leq 0,
\end{equation*}
whenever $(\varphi_n')_{n=1}^N$ is a copy of $(\varphi_n)_{n=1}^N$
independent of $\calF_N$.

\item \label{item:v_func_of_psi}
For every $N\in \N$, every sequence of independent random variables
$(\varphi_n)_{n=0}^N$ on some probability space $(\Omega,\calF,\P)$
with $\calL(\varphi_n)\in \calP$,
and for all $A_n \in \calF_{n}^{\varphi}$, $n\in \{0,\ldots,N-1\}$,
it holds that
\begin{equation*}
  \E
  \Phi\bigg(
      \sum_{n=1}^{N}
	\varphi_n 1_{A_{n-1}}
    ,
      \sum_{n=1}^{N}
	\varphi'_n 1_{A_{n-1}}
  \bigg)
  \leq 0,
\end{equation*}
whenever $(\varphi_n')_{n=0}^N$ is an independent 
copy of $(\varphi_n)_{n=0}^N$.
\end{enumerate}
\end{proposition}

Concerning the proof of 
Proposition~\ref{prop:extend},
it is obvious that we need only prove
\eqref{item:v_func_of_psi}
$\Rightarrow$
\eqref{item:v_adapted}.
For this we need a series of lemmas. Here
Lemmas~\ref{lem:dyadic_aux} and~\ref{lem:dyadic_stefan}
are obtained by an adaptation
of~\cite[Lemma 12.8]{Neerven:2008}, 
in which the dyadic setting is considered and 
which simplifies 
the procedure originally sketched in~\cite{Maurey:1975}.
\smallskip

Recall that a probability space 
$(\Omega,\calF,\P)$ is called 
\emph{divisible} if for every
$A \in \calF$ and every $\theta \in (0,1)$
there exists an $A_{\theta} \in \calF$
such that $A_{\theta} \subset A$
and $\P(A_{\theta})=\theta \P(A)$.

\begin{lemma}\label{lem:dyadic_aux}
Let $(\Omega,\calF,\P)$ be a divisible probability space,
$X$ be a separable Banach space, $F\in \calF$,
and let $\mu\in \calP(X)$ be of the 
form $\mu = \sum_{k=1}^{n} \alpha_k \delta_{x_k}$
for some $n\in \N$, $\alpha_1,...,\alpha_n\in (0,1)$, and some distinct 
$x_1,...,x_n\in X$.
Let $\calG\subseteq \calF$ be a $\sigma$-algebra generated by a finite 
partition $(A_i)_{i=1}^{k}$ of $\Omega$
with $\P(A_i)>0$. 
Then there exists an 
$\calF$-measurable, $\mu$-distributed random variable
$\varphi$ that is independent of $\calG$, for which 
there exist
$H_1,H_2\in 
\sigma(\calG,\varphi)$
satisfying $H_1 \subseteq F \subseteq H_2 $ and 
\[       \P(H_2\setminus H_1) 
    \leq  \left [ \max_{j\in \{1,\ldots,n\}} \alpha_j \right ] 
         \min\big\{ \P(G_2\setminus G_1)\colon G_1,G_2\in \calG, 
         G_1\subseteq F \subseteq G_2\big\}. \]
\end{lemma}

\begin{proof}
As $(\Omega,\calF,\P)$ is divisible, we can
construct a partition 
$( A_{i,j} )_{\substack{ \scriptscriptstyle{ i\in \{1,\ldots,k\} } \\ \scriptscriptstyle{ j \in \{1,\ldots,n\} } } }$ 
of $\Omega$
with $A_{i,j} \in \calF$ for all $i,j$, such that
$A_i = \bigcup_{j=1}^{n} A_{i,j}$
for all $i\in \{1,\ldots,k\}$, and such that
$\P(A_{i,j}) = \alpha_j\P(A_i)$
for all $j\in \{1,\ldots,n\}$ and all $i\in \{1,\ldots,k\}$. 
\smallskip

The partition $(A_{i,j})_{\substack{ \scriptscriptstyle{ i\in \{1,\ldots,k\} } \\ \scriptscriptstyle{ j \in \{1,\ldots,n\} } } }$
is assumed to satisfy some conditions with respect to the set $F$
which we shall explain below. Before doing so,
we observe that given such a partition, the random variable 
$\varphi$ defined by $\varphi:= \sum_{i=1}^k \sum_{j=1}^n x_j \one_{A_{i,j}}$
has the law $\mu$ and is independent of $\calG = \sigma( ( A_i )_{i=1}^{k} )$,
and 
\[  \sigma(\calG,\varphi)
  = \sigma\big(\big\{ 
      A_{i,j}
      \colon 
      i\in \{1,\ldots,k\},
      j\in \{1,\ldots,n
  \big\}\big). \]
Let $I_{0}\subseteq \{1,\ldots,k\}$ be such that $i\in I_0$ if and only if $ 
A_i 
\cap F = \emptyset $, 
and $I_{1}\subseteq \{1,\ldots,k\}$ be such that  $i\in I_1$ if and only if $ 
A_i \subseteq F$.
Set $I_{\operatorname{mix}} = \{1,\ldots,k\}\setminus (I_0\cup I_1\}$
(one or two of the sets $I_0,I_1, I_{\operatorname{mix}}$ may be empty).
Observe that
\begin{align}\label{eq:same_sets}
 \sum_{i\in I_{\operatorname{mix}}} \P( A_i )
 = \min\big\{ 
  \P(G_2\setminus G_1)
    \colon G_1,G_2\in \calG, 
    G_1\subseteq F \subseteq G_2
 \big\}.
\end{align}
For $i\in I_0\cup I_1$ we simply partition the set $A_i$ into sets 
$(A_{i,j})_{j=1}^{n}$ that satisfy $A_{i,j} \in \calF $ and 
\[ \P(A_{i,j}) = \alpha_j\P(A_i) \]
for all $j\in \{1,\ldots,n\}$.
For $i\in I_{\operatorname{mix}}$ we choose 
the partition $(A_{i,j})_{j=1}^{n}$
not only such that it satisfies $A_{i,j} \in \calF$ 
and 
\[ \P(A_{i,j}) = \alpha_j\P(A_i) \]
for all $j\in \{1,\ldots,n\}$,
but also such that there is  {\em at most one}
$j\in \{1,\ldots,n\}$ such that 
$\emptyset \neq F \cap A_{i,j} \subsetneq  A_{i,j}$.
It follows from this construction and from~\eqref{eq:same_sets} that
\begin{equation*}
 \begin{aligned}
 & \min\big\{ 
    \P(H_2\setminus H_1)
    \colon 
    H_1,H_2 \in \sigma(\calG,\varphi), 
    H_1\subseteq F \subseteq H_2
  \big\} 
\\
  & \le 
    \sum_{i\in I_{\operatorname{mix}}}  
    \left [
    \max_{j\in \{1,\ldots,n\}} \alpha_j \right ]
    \P(A_i) 
\\
  & = 
    \left [ \max_{j\in \{1,\ldots,n\} }\alpha_j \right ]
    \min\big\{ 
      \P(G_2\setminus G_1)
      \colon 
      G_1,G_2\in \calG, 
      G_1\subseteq F \subseteq G_2
    \big\}.
 \end{aligned}
\end{equation*}
\end{proof}

\begin{lemma}\label{lem:dyadic_stefan}
Let $(\Omega,\calF,\P)$ be a divisible probability space, $X$ be a separable Banach space 
and 
let $\mu\in \calP(X)$ be of the 
form $\mu = \sum_{k=1}^{n} \alpha_k \delta_{x_k}$
for some $n\in \N$, $\alpha_1,...,\alpha_n\in (0,1)$, and some distinct 
$x_1,...,x_n\in X$.
Let $\calG\subseteq \calF$ be a 
$\sigma$-algebra generated by a finite  partition of atoms with positive measure.
Then for every $A\in \calF$ and 
every $\eps>0$ there exists an $m\in \N$ and 
$\calF$-measurable independent $\mu$-distributed
random variables $(\varphi_1,\ldots,\varphi_m)$
that are independent of $\calG$ such that 
there exists an  
$A_{\eps} \in \sigma(\calG, \varphi_1,\ldots,\varphi_m)$
satisfying
$\E | 1_{A} - 1_{A_{\eps}} | <\eps$.
\end{lemma}

\begin{proof} 
Let $A\in \calF$ and $\varepsilon > 0$ be given.
Define $\delta := \max_{j\in \{1,\ldots,n\}} \alpha_j\in (0,1)$ and let $m\in \N$ be 
such that $\delta^m<\varepsilon $.\medskip

\textbf{Step 1.} Apply Lemma~\ref{lem:dyadic_aux} with 
$\calG$ and $\calF$ as given to find an
$\calF$-measurable, $\mu$-distributed random variable
$\varphi_1$ that is independent
of $\calG$, and sets $H_{1,1},H_{1,2}\in \sigma(\calG,\varphi_1)$
such that
$ H_{1,1} \subseteq A \subseteq H_{1,2} $ and
\[
  \P(H_{1,2}\setminus H_{1,1}) \leq \delta.
\]
Define $\calG_1:=\sigma(\calG,\varphi_1)$ that is, by construction, a $\sigma$-algebra generated 
by a finite partition of sets of positive measure.
\smallskip

\textbf{Step $i$, $i=2,\ldots,m$.}
Apply Lemma~\ref{lem:dyadic_aux} 
with 
$\calG := \calG_{i-1}$, and with $\calF$ as given, to find an
$\calF$-measurable, $\mu$-distributed
random variable $\varphi_i$ that is independent
of $\calG_{i-1}$, and sets 
$ H_{i,1}, H_{i,2} \in \sigma(\calG_{i-1},\varphi_i) $
such that
$ H_{i,1} \subseteq A \subseteq H_{i,2} $ and
\[
  \P(H_{i,2}\setminus H_{i,1}) 
  \leq \delta \P(H_{i-1,2}\setminus H_{i-1,1})
  \leq \delta^i.
\]
Set $\calG_{i}:=\sigma(\calG_{i-1},\varphi_i)$.
\medskip

We have now obtained a sequence of independent,
$\calF$-measurable, $\mu$-distributed random variables 
$(\varphi_1,\ldots,\varphi_m)$ that are independent of $\calG$, and
 sets $H_{m,1}, H_{m,2}\in
\calG_m=\sigma(\calG,\varphi_1,\ldots,\varphi_m)$
such that $H_{m,1} \subseteq A \subseteq H_{m,2}$
and 
\[
  \P(H_{m,2}\setminus H_{m,1}) 
  \leq \delta^m < \varepsilon.
\]
Setting $A_{\eps}= H_{m,1}$ we obtain 
that $\E |  1_A - 1_{A_{\eps}} | <\varepsilon$.
\end{proof} 

\begin{lemma}
\label{lemma:approximation} 
Let $X$ be a separable Banach space, $p\in (0,\infty)$, let \mbox{$\Phi\in C( X\times X , \R)$}
satisfy~\eqref{eq:property_Phi} for all $(x,y)\in X\times X$, 
let $N\in \N$, and let 
$\mu_1,\ldots,\mu_N\in \calP_p(X)$. 
Then for all $\eps>0$ there exists a measurable
mapping $P_{\eps}\colon X \rightarrow X$ with finite range such that
for every sequence of independent random variables  
$(\varphi_1,\ldots,\varphi_N,\varphi_1',\ldots,\varphi_N')$  
on a probability space $(\Omega,\calF,\P)$ such that 
$\calL(\varphi_n)=\calL(\varphi_n')=\mu_n$, $n\in \{1,\ldots,N\}$,
and for all $F_1,\ldots,F_N\in \calF$ it holds that 
\begin{equation}\label{eq:PhiCont}
  \E
  \left|  
    \Phi
    \left( 
      \sum_{n=1}^{N}  \varphi_n  1_{F_n}, 
      \sum_{n=1}^{N}  \varphi_n' 1_{F_n}
    \right)
  -
    \Phi
    \left( 
      \sum_{n=1}^{N} P_{\eps}(\varphi_n ) 1_{F_n}, 
      \sum_{n=1}^{N} P_{\eps}(\varphi_n') 1_{F_n}  
    \right)  
 \right|
 <
 \eps.
\end{equation}\par 
Moreover, if for some $n\in \{1,\ldots,N\}$ it holds that 
$\mu_n$ is not a Dirac measure, then $P_{\eps}$ 
may be chosen such that $\mu_n\circ P_{\eps}^{-1}$ is not a Dirac measure.
\end{lemma}
 
\begin{proof}
Fix $\eps>0$, set $M_{p} = \sum_{n=1}^{N} \int_{X} \| x \|_X^p \,d\mu_n(x)$ 
and let $K\subseteq X$ be a compact set such that 
\begin{equation}\label{eq:Kchoice}
 \sup_{n \in \{ 1,\ldots,N\} }
  \int_{K^c} (1 + N^p M_{p} + N^p \| x \|_{X}^p) \,d\mu_n(x)
 < 2^{-(p-1)^+}(8CN)^{-1} \eps,
\end{equation}
where $C$ is as in~\eqref{eq:property_Phi}. (Note that such a set $K$ exists
as $X$ is separable and hence $\mu_1,\ldots,\mu_N$ are Radon measures, and moreover $\mu_1,\ldots,\mu_N \in P_p(X)$.)
It follows that for $K^N = K\times \ldots \times K$ ($N$ times) one has 
\begin{equation}\label{eq:constructK}
\begin{aligned}
&   \int_{(K^N)^c} 
    \left(
      1+ N^{p}\sum_{n=1}^{N} \| x_n \|_X^p
    \right)
  \,d\mu_1(x_1)\ldots d\mu_N(x_N) 
\\  &
  \leq 
  \sum_{j=1}^{N}
  \int_{\{x_j\in K^c\}} 
    \left(
    1+N^p \sum_{n=1}^{N} \| x_n \|_X^p
    \right)
  \,d\mu_1(x_1)\ldots d\mu_N(x_N) 
\\ &
  \leq 
  \sum_{j=1}^{N}
  \int_{\{x_j\in K^c\}} 
  \left(
    1+N^p \| x_j \|_X^p
  \right)
  \,d\mu_1(x_1)\ldots d\mu_N(x_N)
\\ & \quad 
  +
  \sum_{j=1}^{N}
  \sum_{n=1, n\neq j}^{N}
  \int_{\{x_j\in K^c\}} 
      N^p \| x_n \|_X^p
  \,d\mu_1(x_1)\ldots d\mu_N(x_N)
\\ &
  \leq
  \sum_{j=1}^{N}
  \int_{\{x_j\in K^c\}} 
  \left(
    1 + N^p M_{p} + N^p \| x_j \|_X^p
  \right)
  \,d\mu_1(x_1)\ldots d\mu_N(x_N)
\\ & < 2^{-(p-1)^+}
  (8C)^{-1}\eps.
\end{aligned}
\end{equation}

As $\Phi$ is continuous, it is uniformly continuous on $K\times K$ and hence 
there exists a $\delta\in (0,\infty)$ such that if $x_1,y_1,x_2,y_2\in K$ and 
$\| x_1 - x_2 \|_X<\delta$, $\| y_1 - y_2 \|_X < \delta$, then 
it holds that $|\Phi(x_1,y_1) - \Phi(x_2,y_2) | < \frac{\eps}{2}$.
Note that without loss of generality we may assume that $\delta \leq N^{-1/p}$ and that $K\neq \emptyset$.
Now let $M\in \N$ and $\{U_{1},\ldots,U_{M}\}\subseteq \calB(X)$ be a partition of $K$
such that for all $m\in \{1,\ldots,M\}$ it holds that $U_m\neq \emptyset$
and 
\begin{equation*}
  \sup_{m\in \{1,\ldots,M\} }
  \sup_{x,y\in U_m}
  \| x-y \|_X <  N^{-1} \delta.
\end{equation*}
Let $x_1,\ldots,x_M\in X$ be such that $x_m\in U_m$, $m\in \{1,\ldots,M\}$.
Let $x_0\in \{x\in X \colon \| x \|_{X} = N^{-1} \delta \}\setminus \{x_1,\ldots,x_M\}$ 
(this will be important for the last part of the proof of the lemma).
Define $P_{\eps} \colon X \rightarrow X$ by 
\begin{equation}
 P_{\eps}(x)
 =
 \begin{cases}
 x_m; & x\in U_m, \\
 x_0; & x\not\in K.
 \end{cases}
\end{equation}
Observe that by construction for all $x\in X$
it holds that 
\begin{equation}\label{eq:Pepsbound}
 \| P_{\eps}(x) \|_{X}
 \le 
 \| x \|_{X} + N^{-1}\delta
 \leq 
  \| x \|_{X} + N^{-(1+\frac{1}{p})}
\end{equation}
and for all $x\in K$ it holds that 
\begin{equation}\label{eq:Pepsapprox}
 \| x - P_{\eps}(x) \|_{X}
 < 
 N^{-1}\delta.
\end{equation}

We verify that $P_{\eps}$ satisfies the desired properties.
Indeed clearly $P_{\eps}$ has finite range. Moreover, let $(\Omega,\calF,\P)$ be a probability space 
and let $(\varphi_1,\ldots,\varphi_N,\varphi_1',\ldots,\varphi_N')$ be 
random variables on this space such that 
$\calL(\varphi_n)=\calL(\varphi_n')=\mu_n$, $n\in \{1,\ldots,N\}$,
and let $F_1,\ldots,F_N\in \calF$. For simplicity of notation 
define 
$\xi  = \sum_{n=1}^{N} \varphi_n  1_{F_n}$, 
$\xi' = \sum_{n=1}^{N} \varphi_n' 1_{F_n}$,
$\xi_{\eps}  = \sum_{n=1}^{N}  P_{\eps}(\varphi_n)  1_{F_n}$, and 
$\xi_{\eps}' = \sum_{n=1}^{N}  P_{\eps}(\varphi_n') 1_{F_n}$.
Define 
\begin{equation*}
  K_{\varphi} = 
  \{ 
    \omega \in \Omega 
    \colon 
    (\varphi_1,\ldots,\varphi_N,\varphi_1',\ldots,\varphi_N')
    \in K^{2N} 
   \} .
\end{equation*}
Observe that by~\eqref{eq:Pepsapprox} for $\omega \in K_{\varphi}$ it holds that 
\begin{equation*}
  \| \xi(\omega) - \xi_{\eps}(\omega) \|_X 
  \leq \sum_{n=1}^{N} \| \varphi_n - P_{\eps}(\varphi_n) \|_{X} 
  < \delta,
\end{equation*}
and similarly $\| \xi'(\omega) - \xi'_{\eps}(\omega) \|_X < \delta$, whence 
for all $\omega\in K_{\varphi}$ it holds that 
\begin{equation}\label{eq:Phi_est_1}
 | 
  \Phi( \xi(\omega) , \xi'(\omega) ) 
  -
  \Phi( \xi_{\eps}(\omega), \xi'_{\eps}(\omega) )
 |
 <  
 \frac{\eps}{2}.
\end{equation}
By the estimate above, Assumption~\eqref{eq:property_Phi}, and inequalities 
\eqref{eq:Pepsbound} 
and~\eqref{eq:constructK} it now follows
that
\equa
&   & \hspace*{-5em} \E | \Phi( \xi, \xi') - \Phi( \xi_{\eps}, \xi_{\eps}' )| \\
& = &
  \int_{K_{\varphi}} 
    | 
    \Phi( \xi, \xi')
    -
    \Phi( \xi_{\eps}, \xi_{\eps}' )
    |
  \,d\P
  +
  \int_{K_{\varphi}^c}
    | 
    \Phi( \xi, \xi')
    -
    \Phi( \xi_{\eps}, \xi_{\eps}' )
    |
  \,d\P
\\
& < &
  \frac{\eps}{2} 
  + 
  C
  \int_{K_{\varphi}^c}
    ( 2 + \| \xi \|_X^p + \| \xi' \|_X^p 
     + \| \xi_{\eps} \|_X^p + \| \xi'_{\eps} \|_X^p )
  \,d\P
\\
& \leq &
  \frac{\eps}{2} 
  + 
  2C
  \int_{K_{\varphi}^c}
    \left(
      1 
      + N^{p} 
      \left(
	\sum_{n=1}^{N} \| \varphi_n \|_X^p  
	+ \sum_{n=1}^{N} \| P_{\eps}(\varphi_n) \|_X^p 
      \right)
     \right)
  \,d\P
\\ &
 \leq &
 \frac{\eps}{2}
 +
  2 ^{(p-1)^+} 4C
  \int_{K_{\varphi}^c}
    \left(
      1 
      +  N^{p} \sum_{n=1}^{N} \| \varphi_n \|_X^p  
    \right)
  \,d\P
 < \eps.
\tion
Recalling the definition of $\xi, \xi', \xi_{\eps}$ and $\xi_{\eps}'$ 
this completes the proof of estimate~\eqref{eq:PhiCont}.
\smallskip

In order to prove the final statement in the lemma, we make some 
minor adjustments to the proof above. Indeed, suppose that for some $n\in \{1,\ldots,N\}$ it holds that 
$\mu_n$ is not a Dirac measure. It follows that there exists a compact set $F\in \mathcal{B}(X)$
such that $\mu_n(F)\in (0,1)$. Now proceed as above, but with the 
additional assumption that the set $K$ satisfying~\eqref{eq:Kchoice}
also satisfies $F\subseteq K$, and that the partition $\{U_1,\ldots,U_M\}$ is chosen 
such that $U_m \cap F \in \{U_m,\emptyset\}$ for all $m\in \{1,\ldots,M\}$. As $x_0,x_1,\ldots,x_M$
are all distinct values by construction, this ensures that there exists a set $G\in \mathcal{B}(X)$
such that $P_{\eps}^{-1}(G)=F$.
\end{proof}

\begin{lemma}\label{lemma:Bcont_approx}
Let $X$ be a separable Banach space, $\mu\in \calP(X)$, 
and let 
$
  \calB_{\mu\textnormal{-cont}}(X) 
  = 
  \{ B \in \calB(X) \colon \mu(\partial B) = 0 \}.
$
Then for all $B\in \calB(X)$ and all $\eps>0$
there exists an $B_{\eps}\in \calB_{\mu\textnormal{-cont}}(X)$
such that $ \mu( B \triangle B_{\eps} ) < \eps$.
\end{lemma}

\begin{proof}
Define
\begin{equation*}
 \mathcal{A}
 =
 \left\{ 
  B\in \mathcal{B}(X)
  \colon
  \forall \eps>0\, \exists B_{\eps}\in \calB_{\mu\textnormal{-cont}}(X) 
  \text{ such that }
  \mu( B \triangle B_{\eps} ) < \eps
 \right\}.
\end{equation*}
One may check that $\mathcal{A}$ is a Dynkin-system and that $\mathcal{A}$
contains all closed sets, whence the result follows by the $\pi$-$\lambda$-Theorem.
\end{proof}

\begin{lemma}\label{lemma:weak_approx}
For $N\in \N$ and a separable Banach space $X$ let 
$\varphi_1,\ldots,\varphi_N:\Omega \to X$ and
$\varphi_{1,k},\ldots,\varphi_{N,k}:\Omega_k \to X$, $k\in \N$, be families 
of independent random variables with
\[ \mbox{w*-}\lim_{k\to \infty} \varphi_{n,k}= \varphi_n
   \sptext{1}{for}{1} n\in \{1,\ldots,N\}. \]
Let $(\varphi'_n)_{n=1}^N:\Omega\to X^N$ and $(\varphi'_{n,k})_{n=1}^{N}:\Omega_k\to X^N$ be 
independent copies of $(\varphi_n)_{n=1}^N$ and $(\varphi_{n,k})_{n=1}^N$, respectively,
$v_0\in \R$, and  $B_n \in \calB(X^n)$ such that
$\P( (\varphi_j)_{j=1}^n \in \partial B_n) =0$
for $n\in \{1,\ldots,N-1\}$. Then, for the $w^*$-convergence in $X\times X$ it holds that 
\begin{multline*}
  \hspace*{-.8em}
  \mbox{w*-}\lim_{k\to\infty} \left (
       \varphi_{1,k} v_0 + \sum_{n=2}^{N} \varphi_{n,k}  1_{ \{ ( \varphi_{j,k} )_{j=1}^{n-1} \in B_{n-1} \}} ,
       \varphi_{1,k}'v_0 + \sum_{n=2}^{N} \varphi_{n,k}' 1_{ \{ ( \varphi_{j,k} )_{j=1}^{n-1} \in B_{n-1} \}} 
    \right ) \\
=
 \left ( 
        \varphi_{1} v_0 + \sum_{n=2}^{N}  \varphi_{n}  1_{ \{ ( \varphi_{j} )_{j=1}^{n-1} \in B_{n-1} \}},
        \varphi_{1}'v_0 + \sum_{n=2}^{N}  \varphi_{n}' 1_{ \{ ( \varphi_{j} )_{j=1}^{n-1} \in B_{n-1} \}} 
    \right ). 
\end{multline*}
\end{lemma}

\begin{proof}
By \cite[Theorem 4.30]{Kallenberg:2002} we find probability spaces $(M^{(n)},\Sigma^{(n)},\Q^{(n)})$ and
$(M'^{(n)},\Sigma'^{(n)},\Q'^{(n)})$, and random variables 
\[ \psi_{n,k}, \psi_n  : M^{(n)}  \to X
   \sptext{1}{and}{1}
   \psi_{n,k}',\psi_n' : M'^{(n)} \to X, \]
such that
\begin{enumerate}
\item $\psi_{n,k} \stackrel{d}{=} \varphi_{n,k}$, \, $\psi_n \stackrel{d}{=} \varphi_n$, \, 
      and \, $\lim\limits_{k\to\infty} \psi_{n,k}  = \psi_n $ $\Q^{(n)}$-a.s.
\item $\psi_{n,k}' \stackrel{d}{=} \varphi_{n,k}'$, \, $\psi_n' \stackrel{d}{=} \varphi_n'$, \,
      and \,
      $\lim\limits_{k\to\infty} \psi_{n,k}' = \psi_n'$ $\Q'^{(n)}$-a.s.
\end{enumerate}
We let 
$(M,\Sigma,\Q):= 
\Big [ \otimes_{n=1}^N  (M^{(n) },\Sigma^ {(n)},\Q^ {(n)}) \Big ]\otimes 
\Big [ \otimes_{n=1}^N  (M'^{(n)},\Sigma'^{(n)},\Q'^{(n)}) \Big ]$
and extend $\psi_{n,k}, \psi_n,\psi_{n,k}',\psi_n'$ to $M$. 
Note that for all $n\in \{ 1,\ldots,N-1\}$ it holds that 
$\Q( (\psi_j)_{j=1}^n \in \partial B_{n} ) =0$ 
whence an application of the Portmanteau Theorem (see e.g.~\cite[Theorem 4.25]{Kallenberg:2002}) 
to $B_n^0$ and $\overline{B_n}$, the interior and closure of $B_n$, implies that
\begin{equation*}
 \lim_{k\rightarrow \infty} 1_{ \{ (\psi_{j,k} )_{j=1}^{n} \in B_n \}} 
 =
 1_{ \{ ( \psi_{j} )_{j=1}^{n} \in B_n \} }
 \quad \Q\text{-a.s.}
\end{equation*}
Therefore,
\[
 \lim_{k\to\infty} \left[ \psi_{1,k} v_0 + \sum_{n=2}^{N} \psi_{n,k} 1_{ \{ ( \psi_{j,k} )_{j=1}^{n-1} \in B_{n-1} \}} \right ] 
 =
 \psi_{1}  v_0 + \sum_{n=2}^{N} \psi_{n} 1_{ \{ ( \psi_{j} )_{j=1}^{n-1} \in B_{n-1} \}}
 \quad \Q\mbox{-a.s.}
\]
and
\[
 \lim_{k\to\infty} \left [ \psi'_{1,k} v_0 + \sum_{n=2}^{N} \psi'_{n,k} 1_{ \{ ( \psi_{j,k} )_{j=1}^{n-1} \in B_{n-1} \}} \right ]
 =
 \psi'_{1}  v_0 + \sum_{n=2}^{N} \psi'_n 1_{ \{ ( \psi_{j} )_{j=1}^{n-1} \in B_{n-1} \}}\quad \Q\mbox{-a.s.},
\]
which completes the proof.
\end{proof}

\begin{proof}[Proof of Proposition \ref{prop:extend}]

\textbf{(a)}
We verify that if \eqref{item:v_func_of_psi} holds, then \eqref{item:v_func_of_psi} remains
valid with $\calP$ replaced by $\calP_{p\textnormal{-ext}}$.
\medskip

\textbf{(a.0)}
First we consider $(\varphi_n)_{n=1}^N$
such that for all $n\in \{1,\ldots,N\}$ there exist a 
$ K_n\in \N$ and $(\mu_{n,k})_{k=1}^{K_n} \subseteq \calP$
with
$
\calL(\varphi_n)
=
  \mu_{n,1} 
  *
  \cdots
  *
  \mu_{n,K_n} 
$.
As $A_{n}\in \calF_{n}^{\varphi}$ for $n\in \{0,\ldots,N-1\}$ there are 
$B_n\in X^{n}$ such that 
$ A_n = \{(\varphi_j)_{j=1}^{n} \in B_n\}$ for $n\ge 1$,
whereas for $n=0$ we have $A_0\in \{\emptyset,\Omega\}$.
Now let $((\psi_{n,k})_{k=1}^{K_n})_{n=1}^N$
be independent random variables satisfying 
$\calL(\psi_{n,k})=\mu_{n,k}$.
Then it holds in distribution that
\[     \varphi_1 1_{A_0} +
   \sum_{n=2}^{N} \varphi_n 1_{ \{ (\varphi_j)_{j=1}^{n-1} \in B_{n-1} \} } 
   \stackrel{d}{=}              \sum_{k=1}^{K_1} \psi_{1,k} 1_{A_0}
   + \sum_{n=2}^{N}  \sum_{k=1}^{K_n} \psi_{n,k}  1_{ \{ (\sum_{\ell=1}^{K_{j}} \psi_{j,\ell})_{j=1}^{n-1} \in B_{n-1} \} }
 \]
 and
\[ \varphi_1' 1_{A_0} + \sum_{n=2}^{N} \varphi_n' 1_{ \{ (\varphi_j)_{j=1}^{n-1} \in B_{n-1} \} }
   \stackrel{d}{=} 
   \sum_{k=1}^{K_1} \psi_{1,k}' 1_{A_0}
   + 
   \sum_{n=2}^{N} \sum_{k=1}^{K_n} \psi_{n,k}' 1_{ \{ (\sum_{\ell=1}^{K_{j}} \psi_{j,\ell})_{j=1}^{n-1} \in B_{n-1} \} }
 \]
where $((\psi_{n,k}')_{k=1}^{K_n})_{n=1}^N$
is an independent copy of $((\psi_{n,k})_{k=1}^{K_n})_{n=1}^N$.
Hence if \eqref{item:v_func_of_psi} holds, then \eqref{item:v_func_of_psi} remains
valid with $\calP$ replaced by the set of finite convolutions of elements from $\calP$.
\medskip

\textbf{(a.1)} Assume that $\calL(\varphi_n)\in \calP_{p\textnormal{-ext}}$ for $n\in \{1,\ldots,N\}$. 
By Lemmas \ref{lemma:Bcont_approx} and \ref{lem:PhiConv}
we can restrict the $A_n$ to $ A_n = \{(\varphi_j)_{j=1}^{n} \in B_n\}$ with
$\P((\varphi_j)_{j=1}^{n} \in \partial B_n)=0$ and $B_n \in \calB(X^n)$ for $n\in {1,\dots,N-1}$,
whereas we keep $A_0\in \{\emptyset,\Omega\}$.
We find a  uniformly $\calL_p$-integrable family of independent random variables $(\varphi_{n,k})_{n=1,k=1}^{N,\infty}$ such 
that $\calL(\varphi_{n,k})$ is a finite convolution of measures from $\calP$
and $\calL(\varphi_{n,k})\stackrel{w^*}{\to} \calL(\varphi_n)$ as $k\to \infty$.
Lemma \ref{lemma:weak_approx} gives
\begin{multline*}
       \bigg ( \varphi_{1,k}  1_{A_0} + \sum_{n=2}^{N} \varphi_{n,k}  1_{ \{ ( \varphi_{j,k} )_{j=1}^{n-1} \in B_{n-1} \}} ,
               \varphi_{1,k}' 1_{A_0} + \sum_{n=2}^{N} \varphi_{n,k}' 1_{ \{ ( \varphi_{j,k} )_{j=1}^{n-1} \in B_{n-1} \}} \bigg )
     \\
  \stackrel{w^*}{\longrightarrow}
        \bigg ( \varphi_{1}  1_{A_0} + \sum_{n=2}^{N}  \varphi_{n}  1_{ \{ ( \varphi_{j} )_{j=1}^{n-1} \in B_{n-1} \}},
                \varphi_{1}' 1_{A_0}+ \sum_{n=2}^{N}  \varphi_{n}' 1_{ \{ ( \varphi_{j} )_{j=1}^{n-1} \in B_{n-1} \}} \bigg )
\end{multline*}
as $k\to \infty$. By Lemma \ref{lem:PhiConv} we conclude step (a), i.e. \eqref{item:v_func_of_psi} is valid 
for $\calP_{p\textnormal{-ext}}$.
\medskip

\textbf{(b)}
We now prove that if \eqref{item:v_func_of_psi} 
holds with $\calP$ replaced by $\calP_{p\textnormal{-ext}}$,
then \eqref{item:v_adapted} holds.
\smallskip

\textbf{(b.0)}
Let $(\Omega,\calF,\P,(\calF_n)_{n=0}^N)$ and $(\varphi_n)_{n=1}^N$  be as in \eqref{item:v_adapted}, and set
$\mu_n := \law(\varphi_n)$. If each $\mu_n$ is a Dirac measure in an $x_n\in X$,
then $\E\Phi(\sum_{n=1}^N\varphi_n 1_{A_{n-1}},\sum_{n=1}^N\varphi_n' 1_{A_{n-1}})$ is a weighted sum of 
terms  $\Phi(\sum_{n\in I} x_n,\sum_{n\in I} x_n)$ with $I\subseteq \{1,\ldots,n\}$ (the empty sum is treated as zero).
In this case \eqref{item:v_func_of_psi} implies that each of these terms in non-positive, so that
in what follows we assume there exists an $\ell \in \{1,\ldots,N\}$ such that $\mu_{\ell}$
is \emph{not} a Dirac measure.
We will prove that for all $\eps>0$ it holds 
that 
\begin{equation}\label{eq:extend_eps}
  \E
  \Phi\bigg(
      \sum_{n=1}^{N}
	\varphi_n 1_{A_{n-1}}
    ,
      \sum_{n=1}^{N}
	\varphi'_n 1_{A_{n-1}}
  \bigg)
  <
  3\eps,
\end{equation}
which completes the proof of \eqref{item:v_adapted}.
By passing to the larger probability space
$(\Omega,\calF,\P)\otimes ([0,1],\calB([0,1]),\lambda)$
(where $\lambda$ is the Lebesgue measure), endowed with
the filtration $(\calF_n \otimes \calB([0,1]))_{n=0}^N$,
we may assume that $(\Omega,\calF_n,\P)$ is divisible for all
$n\in \{0,\dots,N\}$.
Fix $\varepsilon>0$, and let $P_{\eps}$ be 
as in Lemma~\ref{lemma:approximation} with the property that 
$\mu_{\ell}\circ P_{\eps}^{-1} $ is not a Dirac measure.
Recall
from Lemma~\ref{lemma:approximation} that 
\begin{equation}\label{eq:Phiapprox1}
  \E\left| 
  \Phi\bigg(
      \sum_{n=1}^{N}
	\varphi_n 1_{A_{n-1}}
    ,
      \sum_{n=1}^{N}
	\varphi'_n 1_{A_{n-1}}
  \bigg)
  - 
  \Phi\bigg(
      \sum_{n=1}^{N}
	P_{\eps}(\varphi_n) 1_{A_{n-1}}
    ,
      \sum_{n=1}^{N}
	P_{\eps}(\varphi'_n) 1_{A_{n-1}}
  \bigg)  
  \right|  
  <
  \eps.
\end{equation}

\textbf{(b.1)}
\smallskip
Set $\mu_{\ell,\eps} := \mu_{\ell}\circ P_{\eps}^{-1}$ 
and recall that $\mu_{\ell,\eps}$
is not a single Dirac measure, but a finite sum of  Dirac measures.
For $m\in \N$ we proceed as follows:
\medskip

\textbf{Step 0.} 
We apply Lemma~\ref{lem:dyadic_stefan} 
with $\calG = \{\emptyset,\Omega\}$ 
and $\calF = \calF_0$ to find a
$k_{m,1} \in \{2,3,\ldots\}$ and a sequence
of independent, $\calF_0$-measurable, 
$\mu_{\ell,\eps}$-distributed random variables
$(\psi_{m,1},...,\psi_{m,k_{m,1}-1})$, 
and to find an
$
A_{m,0} \in
  \sigma(\psi_{m,1},...,\psi_{m,k_{m,1}-1})$
with
\[ 
  \| 1_{A_0} - 1_{A_{m,0}} \|_{L^p(\P)}
  \leq  
  2^{-m}. 
\]
Set $\psi_{m,k_{m,1}}:= P_\eps(\varphi_1)$.

\smallskip

\textbf{Step $n;\,n =1,\ldots,N-1$.}       
Apply Lemma~\ref{lem:dyadic_stefan}
with $\calG=\sigma(\psi_{m,1},\ldots, \psi_{m,k_{m,n}})$ 
(note that $\psi_{m,k_{m,1}}=P_\eps(\varphi_1)$,\ldots, $\psi_{m,k_{m,n}}=P_\eps(\varphi_n)$)
and  $\calF=\calF_{n+1}$, to find a $k_{m,n+1} \in \{k_{m,n}+2,k_{m,n}+3,\ldots\}$
and independent, $\calF_{n+1}$-measurable, 
$\mu_{\ell,\eps}$-distributed random variables
$(\psi_{m,k_{m,n}+1},...,\psi_{m,k_{m,n+1}-1})$, 
independent
of the $\sigma$-al\-gebra $\sigma(\psi_{m,1},\ldots, \psi_{m,k_{m,n}})$ as well,
and to find an
$ 
  A_{m,n}\in \sigma(\psi_{m,1},\ldots,\psi_{m,k_{m,n+1}-1})
$
with 
\[
   \| 1_{A_n} - 1_{A_{m,n}} \|_{\calL^p(\P)}
  \leq
    2^{-m}. 
\]
Set $\psi_{m,k_{m,n+1}}:= P_\eps(\varphi_{n+1})$.
\medskip

\textbf{(b.2)}
By construction, 
\begin{multline*}
  \lim_{m\rightarrow \infty}
  \left\|
  \sum_{n=1}^{N} (1_{A_{m,n-1}} - 1_{A_{n-1}}) P_{\eps}(\varphi_n) 
  \right\|_{\calL^p(\P;X)} \\
  = 
\lim_{m\rightarrow \infty}
  \left\|
  \sum_{n=1}^{N} (1_{A_{m,n-1}} - 1_{A_{n-1}}) P_{\eps}(\varphi_n') 
  \right\|_{\calL^p(\P;X)}
  = 0,
\end{multline*}
 hence by Lemma~\ref{lem:PhiConv} there exists an $M \in \N$
such that, for $B_{n-1}=A_{M,n-1}$,
\begin{multline}\label{eq:Phiapprox2b}
 \bigg |
  \E \Phi\left(
   \sum_{n=1}^{N} P_{\eps}(\varphi_n ) 1_{B_{n-1}} ,
   \sum_{n=1}^{N} P_{\eps}(\varphi_n') 1_{B_{n-1}}
  \right) \\
  -
  \E\Phi\left(
   \sum_{n=1}^{N} P_{\eps}(\varphi_n ) 1_{A_{n-1}},
   \sum_{n=1}^{N} P_{\eps}(\varphi_n') 1_{A_{n-1}}
  \right)
 \bigg |
 <
 \eps.
\end{multline}

\textbf{(b.3)} 
We would like to apply \eqref{item:v_func_of_psi} to 
$ \sum_{n=1}^{N} 1_{B_{n-1}} \varphi_n $, however,
our construction of $B_{n-1}$ only guarantees that 
$P_{\eps}(\varphi_n)$ is independent of $B_{n-1}$,
\emph{not} that $\varphi_n$ is independent of $B_{n-1}$.
Hence we proceed as follows.
Let $k_n:=k_{M,n}$ for $n\in \{1,\ldots,N\}$ and $(\tilde{\psi}_k)_{k=1}^{k_{N}}$
be a sequence of independent random 
variables such that 
$ \tilde{\psi}_{k_{n}}$
is $\mu_n$-distributed for all $n\in \{1,\ldots,N\}$,
and $ \tilde{\psi}_{k}$ is $\mu_{\ell}$-distributed for 
all $k\in \{1,\ldots,k_N\} \setminus \{k_1,\dots,k_N\}$.
By the Factorization Lemma for all $n\in \{1,\ldots,N\}$ there exists
a $ C_{n-1} \in X^{k_{n}-1}$ such 
that
\[  B_{n-1}
     = 
    \Big\{
      (
	\psi_{1},
	\ldots, 
	\psi_{k_{n}-1}
      )
      \in C_{n-1}
    \Big\}.
\]
Define, for $n\in \{1,\ldots,N\}$, 
$
  \tilde{B}_{n-1} 
  =
  \{(P_{\eps}(\tilde{\psi}_k))_{k=1}^{k_{n}-1}\in C_{n-1}\}.
$
By construction,  
\begin{equation*}
 \left( 
  \sum_{n=1}^{N} 1_{B_{n-1}} P_{\eps}(\varphi_n)
 , 
  \sum_{n=1}^{N} 1_{B_{n-1}} P_{\eps}(\varphi_n')
 \right)
\stackrel{d}{=}  
 \left(
  \sum_{n=1}^{N} 1_{\tilde{B}_{n-1}} P_{\eps}(\tilde{\psi}_{k_n})
 ,
  \sum_{n=1}^{N} 1_{\tilde{B}_{n-1}} P_{\eps}(\tilde{\psi}_{k_n}')
 \right),
\end{equation*}
where $(\tilde{\psi}_k')_{k=1}^{k_{N}}$ is an independent copy of 
$(\tilde{\psi}_k)_{k=1}^{k_{N}}$, so that
\begin{multline}\label{eq:Phiapprox3}
 \E  
 \Phi\left(
  \sum_{n=1}^{N} P_{\eps}(\varphi_n ) 1_{B_{n-1}},
  \sum_{n=1}^{N} P_{\eps}(\varphi_n') 1_{B_{n-1}} 
 \right) \\
 =
 \E
 \Phi\left(
  \sum_{n=1}^{N}  P_{\eps}(\tilde{\psi}_{k_n}) 1_{\tilde{B}_{n-1}}
  ,
  \sum_{n=1}^{N} P_{\eps}(\tilde{\psi}_{k_n}') 1_{\tilde{B}_{n-1}} 
 \right).
\end{multline}
By the assumption that Proposition~\ref{prop:extend}~\eqref{item:v_func_of_psi} holds for 
$\calL(\varphi_n)\in \calP_{p\textnormal{-ext}}$
it follows that
\begin{equation*}
 \E
 \Phi\left(
  \sum_{n=1}^{N} \tilde{\psi}_{k_n}  1_{\tilde{B}_{n-1}}
  ,
  \sum_{n=1}^{N} \tilde{\psi}_{k_n}' 1_{\tilde{B}_{n-1}}
 \right)
 \leq 
 0.
\end{equation*}
The inequality above in combination with Lemma~\ref{lemma:approximation} implies that  
\begin{equation*}
 \E  
 \Phi\left(
  \sum_{n=1}^{N}  P_{\eps}(\tilde{\psi}_{k_n})  1_{\tilde{B}_{n-1}},
  \sum_{n=1}^{N}  P_{\eps}(\tilde{\psi}_{k_n}') 1_{\tilde{B}_{n-1}}
 \right)
 <
 \eps.
\end{equation*}
Combining the inequality above with~\eqref{eq:Phiapprox3},~\eqref{eq:Phiapprox2b}, and~\eqref{eq:Phiapprox1}
we arrive at the desired estimate~\eqref{eq:extend_eps}.
\end{proof}


\subsection{From simple decoupling to general decoupling}
\label{subsec:from_simple_to_Lp}

In addition to the space $\adapt_p(\Omega,\F;X,\calP)$
introduced in Definition~\ref{definition:class_of_adapted_processes}
we shall need the following one:

\begin{definition}
\label{definition:class_of_adapted_processes_2}
Let $X$ be a separable Banach space, $p\in (0,\infty)$, $\emptyset \not = \calP \subseteq \calP_p(X)$, and
let $(\Omega,\calF,\P,(\calF_n)_{n=0}^N)$, $N\in \N$, be a stochastic basis. We shall denote by
$\adapt_{p\textnormal{-simple}}(\Omega,(\calF_n)_{n=0}^N;X,\calP)$ the set of $(\calF_n)_{n=1}^N$-adapted sequences  
$(d_n)_{n=1}^N$ such that for all $n\in \{1,\ldots,N\}$ there exist 
$K_n\in \N$, 
a partition $(A_{n-1,k})_{k=1}^{K_n}\subseteq \calF_{n-1}$ 
of $\Omega$ consisting of sets of positive measure,
and $\mu_{n,1},\ldots,\mu_{n,K_n}\in\calP$
such that $\sum_{k=1}^{K_n} 1_{A_{n-1,k}} \mu_{n,k}$ 
is a regular version of $\P(d_n\in\cdot \,|\, \calF_{n-1})$.
\end{definition}

\begin{proposition}\label{prop:simple_to_Lp}
Let $X$ be a separable Banach space,
let $\Phi \in C(X \times X;\R)$
be such that there exist 
constants $C,p\in (0,\infty)$ for which it holds that
\begin{equation}\label{eq:property_Phi2}
  | \Phi(x,y) |
  \leq
  C( 1 + \| x \|_X^p + \| y \|_X^p)
\end{equation}
for all $ (x,y) \in X\times X$, and
let $\calP \subseteq \mathcal{P}_p(X)$ with $\delta_0 \in \calP$.
Then the following assertions are equivalent:
\begin{enumerate}
\item \label{item:general_decoupling_2}
      For every stochastic basis $(\Omega,\calF,\P,\F)$ with $\F=(\calF_n)_{n\in \N_0}$
      and every finitely supported 
      $(d_n)_{n\in \N} \in \adapt_p(\Omega,\F;X,\calP)$
      it holds that
      \[     \E \Phi\left (\sum_{n=1}^\infty d_n,\sum_{n=1}^\infty e_n \right ) 
         \le 0, \] 
      whenever $(e_n)_{n\in \N}$ is an $\F$-decoupled tangent sequence 
      of $(d_n)_{n\in \N}$.
\item \label{item:general_decoupling_simple}
      For every $N\in \N$, every stochastic basis $(\Omega,\calF,\P,(\calF_n)_{n=0}^N)$ and 
      $(d_n)_{n=1}^N \in \adapt_{p\textnormal{-simple}}
      (\Omega,(\calF_n)_{n=0}^N;X,\calP)$ it holds that
      \[ \E \Phi\left (\sum_{n=1}^N d_n,\sum_{n=1}^N e_n \right ) \le 0, \] 
      whenever $(e_n)_{n=1}^N$ is an $(\calF_n)_{n=0}^N$-decoupled 
      tangent sequence of $(d_n)_{n=1}^N$.
\item \label{item:simple_decoupling_2}
      For every $N\in \N$, every stochastic basis $(\Omega,\calF,\P,(\calF_n)_{n=0}^N)$,
      every sequence of independent random variables $(\varphi_n)_{n=1}^N$ with 
      $\calL(\varphi_n) \in \calP$ such that $\varphi_n$ is $\calF_{n}$-measurable
      and independent of $\calF_{n-1}$, and for all $A_n \in \calF_{n}$, $n \in \{0,\ldots,N-1\}$,
      it holds that
      \begin{equation*}
      \E
      \Phi\bigg(
      \sum_{n=1}^{N}
	\varphi_n 1_{A_{n-1}}
    ,
      \sum_{n=1}^{N}
	\varphi'_n 1_{A_{n-1}}
      \bigg)	\leq 0,
      \end{equation*}
      whenever $(\varphi_n')_{n=1}^N$ is a copy of $(\varphi_n)_{n=1}^N$
      independent of $\calF_N$.
\end{enumerate}
\end{proposition}

We shall use the following lemmas to prove the proposition above.

\begin{lemma}
\label{lem:aux_decoupling}
Let $(\Omega,\calF,\P)$ be a probability space,
let $K\in \N$ and let $A_k \in \calF $, $k\in \{1,\ldots,K\}$,
be such that
$(A_k)_{k=1}^{K}$ is a partition of $\Omega$ and $\P(A_k)>0$ for 
all $k\in \{1,\ldots,K\}$. 
Let $X$ be a separable Banach space and let $d\colon \Omega \rightarrow X$
be a random variable. Let $\kappa \colon \Omega \rightarrow \calP(X)$
be a regular version of $\calP( d \in \cdot \,|\, \sigma((A_k)_{k=1}^{K}))$,
i.e., $\kappa = \sum_{k=1}^{K} \mu_k 1_{A_k}$ for some 
$\mu_1,\ldots,\mu_K\in \calP(X)$. Let $(\Omega',\calF',\P')$ be an
auxiliary probability space and let $(d'_k)_{k=1}^{K}$ be a sequence 
of independent $X$-valued random variables on $(\Omega',\calP',\calF')$
satisfying $ \calL(d_k') = \mu_k$, and let, for all $k\in \{1,\ldots,K\}$,
$d_k \colon (\Omega,\calF,\P) \times (\Omega',\calF',\calP') \rightarrow X$ 
be a random variable defined by
\begin{equation}\label{eq:def_dk}
  d_k(\omega,\omega') 
  = 
  d(\omega) 1_{A_k}(\omega) 
  + d_k'(\omega') 1_{\Omega \setminus A_k}(\omega)
\end{equation}
for all $(\omega,\omega')\in \Omega\times \Omega'$.
Let $\calG\subseteq \calF$ 
be a $\sigma$-algebra such that $\sigma( (A_k)_{k=1}^K ) \subseteq \calG$.
Define $\calG_0 = \calG \otimes \{\emptyset, \Omega'\}$
and for $k\in \{1,\ldots,K\}$ define
\begin{equation}
 \calG_k 
 =
 \sigma( \calG_0, d_1,\ldots,d_k ).
\end{equation}
Then the following holds:
\begin{enumerate} 
 \item For all $(\omega,\omega')\in \Omega\times \Omega'$
 it holds that 
 $ 
  d(\omega) 
  = 
  \sum_{k=1}^{K} 
    d_{k}(\omega,\omega') 1_{A_k}(\omega)
 $. \label{item:dsum}
 \item\label{item:dkfilt} 
  $(d_k)_{k=1}^K$ is $(\calG_k)_{k=1}^{K}$-adapted.
  \item\label{item:ddist}
 $\calL(d_k) = \mu_k$.
 \item\label{item:dkindep} $(\calG_0,\sigma(d_1),\ldots,\sigma(d_K))$ are independent
  if and only if 
 for all $k\in \{1,\ldots, K\}$, all $A\in \calG$ 
 satisfying $A\subseteq A_k$ and $\P(A)>0$, and all $B\in \calB(X)$ it holds that
\begin{equation}\label{eq:indep_cond}
 \P( d \in B \,|\, A ) :=
 \frac{ \P( \{ d\in B \} \cap A )}{ \P(A) }
 = \mu_k(B).
\end{equation}
\end{enumerate}
\end{lemma}

\begin{proof}
Claims~\eqref{item:dsum},~\eqref{item:dkfilt}
and~\eqref{item:ddist} are trivial. 
Regarding claim~\eqref{item:dkindep},
suppose that $d_k$ is independent of $\calG_{k-1}$
for all $k\in \{1,\ldots,K\}$, then in particular
$(d_k)_{k=1}^{K}$ is independent of $\calG_0$. Let $k\in \{1,\ldots,K\}$, 
$A\in \calG$ satisfying $A\subseteq A_k$ and $\P(A)>0$, and $B\in \calB(X)$ be given. Then
\begin{equation*}
\begin{aligned}
 \P( d\in B \,|\, A ) 
& =
 \frac{ \P( \{ d \in B \} \cap A ) }
 { \P(A)}
 =
 \frac{ \P\otimes \P' ( \{ d_k \in B \} \cap (A \times \Omega') )}
 { \P\otimes \P'(A \times \Omega')}
\\ & 
 = 
 \P\otimes \P'( d_k \in B )
 =
 \mu_k(B),
\end{aligned}
\end{equation*}
where we use \eqref{item:dsum},  independence, and \eqref{item:ddist}.
In order to prove the reverse implication, 
let $B_1,...,B_K \in \calB( X )$, let $k\in \{1,\ldots,K\}$, and let $A\in \calG$
be such that $A\subseteq A_k$ and $\P(A)>0$.
It holds that
\begin{equation}
\begin{aligned}
&    \P \otimes \P' ( (A\times \Omega') \cap \{ d_1\in B_1,\ldots,d_K \in B_K \} )
\\
&    =
    \P \otimes \P' 
    \Big[
    ( A \cap \{d \in B_k\} ) 
    \times 
    \Big(
      \bigcap_{\ell \in \{1,\ldots,K\}\setminus \{k\}} \{ d_\ell' \in B_\ell  \}
    \Big)
    \Big]
\\ & 
    = 
    \P (A) \P( d\in B_k \,|\, A )
    \prod_{\ell \in \{1\ldots,K\}\setminus \{k\}} \P'( d_{\ell}' \in B_{\ell} )
    = 
    \P (A)
    \prod_{\ell \in \{1\ldots,K\}} \mu_{\ell}(B_{\ell}).
\end{aligned}
\end{equation}
This suffices to prove the reverse implication.
\end{proof}
\medskip

Let $N\in \N$ and let 
$(d_n)_{n=1}^N\in \adapt_{p\textnormal{-simple}}(\Omega,\P,(\calF_n)_{n=0}^N;X,\calP)$ 
i.e., for all $n\in \{1,\ldots,N\}$ we have a $K_n\in \N$
such that 
\begin{enumerate}
\item $\P( d_n \in \cdot \,|\, \calF_{n-1}) = \sum_{k=1}^{K_n} \mu_{n,k} 1_{A_{n-1,k}}$ 
      a.s., where
\item $A_{n-1,1},\ldots,A_{n-1,K_n} \in \calF_{n-1}$ is a partition of 
      $\Omega$ with $\P( A_{n-1,k} ) >0$ and $\mu_{n,1},\ldots,\mu_{n,K_n} \in \calP$.
\end{enumerate}
We set $K_{0}:=1$ and
\equa 
J_0  &:= & \{ (n,k) : n\in \{0,\ldots,N\}, k\in \{1,\ldots,K_n\}\}, \\
J    &:= & \{ (n,k) : n\in \{1,\ldots,N\}, k\in \{1,\ldots,K_n\}\}. 
\tion
On $J_0$ we introduce the lexicographical order
$(m,j) \prec (n,k)$ if either 
$m<n$ and $j\in \{1,\ldots,K_m\}$ or $m=n$ and $j\in \{1,\ldots,k\}$.
For an auxiliary  probability space $(\Omega',\calF',\P')$ and  independent random variables  
$d'_{n,k} \colon \Omega' \rightarrow X$ with  $\calL(d'_{n,k})= \mu_{n,k}$,
$n\in \{1,\ldots,N\}$, and $k\in \{1,\ldots,K_n\}$ let
\begin{equation*}\label{eq:def_psi}
  d_{n,k}(\omega,\omega')
  := 
  d_n(\omega) 1_{A_{n-1,k}}(\omega)
  +
  d'_{n,k}(\omega')1_{\Omega\setminus A_{n-1,k}}(\omega).
\end{equation*}
We define 
\equa
\calG_{0,1}   &:= & \calF_0 \otimes \{ \emptyset,\Omega'\}, \\
\calG_{n,k}   &:= & \calG_{n-1,K_{n-1}} \vee \sigma (d_{n,1},\ldots,d_{n,k}) \quad
                    (n\in \{1,\ldots,N\}, k\in\{1,\ldots,K_n-1\}), \\
\calG_{n,K_n} &:= & \calG_{n-1,K_{n-1}} \vee \sigma (d_{n, 1},\ldots,d_{n,K_n}) 
                                        \vee \bigg ( \calF_n \otimes \{ \emptyset,\Omega'\} \bigg ).
\tion
Finally, we let
\begin{enumerate}
\item $K:= K_1+\cdots+K_N$,
\item $(\calH_{\ell})_{\ell=0}^K$ be the lexicographical ordering of $(\calG_{n,k})_{(n,k) \in J_0}$,
\item $(\varphi_{\ell})_{\ell=1}^K$ be the lexicographical ordering of $(d_{n,k})_{(n,k)\in J}$, 
\item $(A_{\ell})_{\ell=0}^{K-1}$ be the lexicographical 
      ordering of $(A_{n-1,k})_{(n,k)\in J}$,
\item $(\varphi'_{\ell})_{\ell=1}^K$ be the lexicographical ordering of $(d'_{n,k})_{(n,k)\in J}$. 
\end{enumerate}
\medskip
\pagebreak
\begin{lemma}
\label{lem:A_sim_to_disjoint} 
The following holds true:
\begin{enumerate}
\item $(\varphi_\ell)_{\ell=1}^K$ is  $(\calH_{\ell})_{\ell=1}^K$-adapted and $\varphi_\ell$ is 
      independent from $\calH_{\ell-1}$.
\item $A_\ell\in \calH_\ell$ for $\ell\in \{0,\ldots,K-1\}$.
\item $\sum_{n=1}^N d_n = \sum_{\ell=1}^K \varphi_\ell 1_{A_{\ell-1}}$.
\item If $e_n := \sum_{k=1}^{K_n} d'_{n,k} 1_{A_{n-1,k}}$,
      then
      \begin{enumerate}
      \item $(e_n)_{n=1}^N$ is a decoupled tangent sequence of $(d_n)_{n=1}^N$ and
      \item $\sum_{n=1}^N e_n =  \sum_{\ell=1}^K \varphi'_\ell 1_{A_{\ell-1}}$.
      \end{enumerate}
\end{enumerate}
\end{lemma}

\begin{proof}
(i) Fix $\ell\in \{ 1,\ldots,K\}$. By definition, $\varphi_\ell$ is  $\calH_\ell$-measurable. To show
that $\varphi_\ell$ is independent from $\calH_{\ell-1}$ it is enough to verify that 
\[ \calG_{n-1,K_{n-1}},d_{n,1},\ldots,d_{n,K_n} \]
are independent for $n\in \{1,\ldots,N \}$. Because 
\[
       \calG_{n-1,K_{n-1}}
    =  \calF_{n-1} \otimes \sigma \bigg (d'_{m,j} :  m\in \{1,\ldots,n-1\}, j\in \{1,\ldots,K_m\} \bigg ), 
\]
where for $n=1$ the second factor is replaced by $\{ \emptyset,\Omega'\}$, it remains to check that 
\[  \bigg (\calF_{n-1} \otimes \{ \emptyset,\Omega' \} \bigg ),d_{n,1},\ldots,d_{n,K_n} \]
are independent. But this follows from Lemma \ref {lem:aux_decoupling}.
\smallskip

(ii) and (iv) follow by construction.

(iii) follows from 
 $d_n(\omega) 
    = 
    \sum_{k=1}^{K_{n}} d_{n,k}(\omega,\omega') 1_{A_{n-1,k}}(\omega)$ for $n\in \{1,\ldots,N\}$.
\smallskip
\end{proof}

\begin{proof}[Proof of Proposition~\ref{prop:simple_to_Lp}]
\eqref{item:general_decoupling_2} $\Rightarrow$ \eqref{item:simple_decoupling_2} follows by 
Example~\ref{example:decoupled_sequence}.
\medskip
 
\eqref{item:simple_decoupling_2} $\Rightarrow$ \eqref{item:general_decoupling_simple}
follows from Lemma \ref{lem:A_sim_to_disjoint}.
\medskip

\eqref{item:general_decoupling_simple} $\Rightarrow$ \eqref{item:general_decoupling_2}
Let $(d_n)_{n\in\N}$ be as in \eqref{item:general_decoupling_2} with $d_n\equiv 0$ if $n>N$
for some $N\in \N$. By Corollary~\ref{cor:factor_through_subsigmaalgebra} we obtain for $n\in \{1,\ldots,N\}$
random variables $d_n^0\colon \Omega\times [0,1]\rightarrow X$ and $H_n\colon \Omega \times (0,1] \rightarrow [0,1]$
such that $d_n^0$ is $\calF_{n-1}\otimes \calB([0,1])/\calB(X)$-measurable, $H_n$ is independent of 
$\calF_{n-1}\otimes \{\emptyset,(0,1]\}$ with $\calL(H_n)=\lambda$,
and $d_n(\omega) = d^0_{n}(\omega,H_n(\omega,s))$ for all $(\omega,s)\in \Omega_n\times (0,1]$ for some $\Omega_n\in\calF_n$ 
of measure one. We define 
$\bar{\Omega} := \Omega \times (0,1]^N$,
$\bar{\calF}_0 := \calF_0\otimes \{ \emptyset, (0,1]^N  \}$, 
$\bar{\calF}_n := \calF_n \otimes \sigma (\pi_1,\ldots \pi_n)$, where 
$\pi_n : (0,1]^N \to (0,1]$ is the projection onto the $n$-th coordinate, and 
$\bar{\P}:= \P\otimes \lambda_N$ where $\lambda_N$ is
the Lebesgue measure on $(0,1]^N$. 
Then $\bar{H}_n, \bar{d}_n, \bar{e}_n \colon \bar{\Omega} \rightarrow [0,1]$
are given by $\bar{H}_n(\omega,s):=H_n(\omega,\pi_n(s))$, $\bar{d_n}(\omega,s) := d_n^{0}(\omega,\bar{H}_n(\omega,s))$, 
and $\bar{e}_n(\omega,s):=d^{0}_n(\omega,\pi_n(s))$, $n\in \{1,\ldots,N\}$. 
We get: 
\smallskip
\begin{enumerate}
 \item $\{ (\omega,s)\in \bar{\Omega} \colon \bar{d}_n(\omega,s) = d_n(\omega) \} \supseteq \Omega_n \times  (0,1]$
 \item The fact that $\bar{H}_n$ is uniformly-$[0,1]$ distributed and  independent of $\bar{\calF}_{n-1}$ and 
	$d_n^0$ is $\calF_{n-1}\otimes \calB([0,1])/\calB(X)$-measurable implies that $(\bar{e}_n)_{n\in\N}$ is an 
        $(\bar{\calF}_n)_{n=1}^{N}$-decoupled tangent sequence of $(\bar{d_n})_{n\in\N}$.
 \item The function $\kappa_{n-1}[\omega,B]:=\lambda(\{h\in [0,1] : d_n^0(\omega,h)\in B\})$, $B\in \calB(X)$, is a regular
       conditional probability for $\P(d_n\in \cdot\,|\,\calF_{n-1})$. 
\end{enumerate}
\smallskip

Next note that for all $\eps>0$ and $n\in \{1,\ldots,N\}$ 
there exists an $d^{0,\eps}_n\colon \Omega\times [0,1]\rightarrow X$ which satisfies  
\begin{equation*} 
 d^{0,\eps}_n(\omega,s) = \sum_{k=1}^{m_{\eps}} 1_{F_{n-1,k,\eps}}(\omega) f_{n,k,\eps}(s),
\end{equation*}
with $m_{\eps}\in \N$, $f_{n,k,\eps}\in \calL^p([0,1];X)$, pair-wise disjoint $F_{n-1,1,\eps},\ldots,F_{n-1,m_\eps,\eps} 
\in \calF_{n-1}$ of positive measure, and $\| d^{0,\eps}_n - d^{0}_n \|_{\calL^p(\Omega\times [0,1];X)} < \eps$. 
Moreover, we can pick $f_{n,k,\eps}$ such that $\calL(f_{n,k,\eps})\in \calP$.
Indeed, let $(f_\ell)_{\ell\in \N}\subset \calL^p([0,1];X)$ be dense. For $\eps>0$ we can choose an appropriate 
$\eta=\eta(p,\eps)>0$ and define 
\begin{align*}
 S_1&:= \{ \omega\in \Omega: \|d_n^0(\omega,\cdot) - f_1 \|_{\calL^p([0,1];X)} < \eta\},\\
 S_2&:= \{ \omega\in \Omega: \|d_n^0(\omega,\cdot) - f_2 \|_{\calL^p([0,1];X)} < \eta\}\setminus S_1, \ldots.
\end{align*}
From these sets $(S_n)_{n\in\N}$ we extract the collection $(F_{n-1,k,\eps})_{k=1}^{m_\eps}$, find
$\omega_k\in F_{n-1,k,\eps}\cap \kappa_{n-1}^{-1}(\calP)$, and let $f_{n,k,\eps}:=d_n^0(\omega_k,\cdot)$.
We continue and define $\bar{d}_n^{\eps}, \bar{e}_n^{\eps} \colon \bar{\Omega} \rightarrow [0,1]$
by $\bar{d}_n^{\eps}(\omega,s) := d_n^{0,\eps}(\omega,\bar{H}_n(\omega,s))$, 
and $\bar{e}_n^{\eps}(\omega,s):=d^{0,\eps}_n(\omega,\pi_n(s))$, $n\in \{1,\ldots,N\}$.
By construction we have
\begin{enumerate}
 \item $\| \bar{e}^{\eps}_n - \bar{e}_n \|_{\calL^p(\bar{\Omega};X)} = 
       \| \bar{d}^{\eps}_n - \bar{d}_n \|_{\calL^p(\bar{\Omega};X)} < \eps$ for all $n\in \{1,\ldots,N\}$,
 \item $(\bar{e}_n^{\eps})_{n=1}^{N}$ is an $(\bar{\calF}_n)_{n=1}^{N}$-decoupled
        tangent sequence of $(\bar{d_n}^{\eps})_{n\in\N}$,
 \item a conditional regular conditional probability kernel of $d_n^{\eps}$ given $\bar{\calF}_{n-1}$
is given by $\sum_{k=1}^{m_{\eps}} 1_{F_{n-1,k,\eps}}(\omega) \calL(f_{n,k,\eps})$.
\end{enumerate}
This concludes \eqref{item:general_decoupling_simple} $\Rightarrow$ \eqref{item:general_decoupling_2}.
\end{proof}

\subsection{Proof of Theorem \ref{thm:simple_decoupling_of_vs_full_decoupling}}
\label{subsec:Proof:simple_decoupling_of_vs_full_decoupling}
\eqref{item:simple_decoupling} $\Rightarrow$ \eqref{item:general_decoupling} 
First we apply Proposition \ref{prop:extend} to deduce from 
Theorem \ref{thm:simple_decoupling_of_vs_full_decoupling} \eqref{item:simple_decoupling} the statement 
in  Proposition \ref{prop:extend}  \eqref{item:v_adapted}.
Now we use Proposition \ref{prop:simple_to_Lp} \eqref{item:simple_decoupling_2} for $\calP_{p\textnormal{-ext}}$ instead of for 
$\calP$ and obtain Proposition \ref{prop:simple_to_Lp} \eqref{item:general_decoupling_2} for $\calP_{p\textnormal{-ext}}$, which is
Theorem \ref{thm:simple_decoupling_of_vs_full_decoupling} \eqref{item:general_decoupling}.


\section{Extrapolation}
\label{sec:extrapolation}

The following extrapolation result is standard. It was proved for the general decoupling (which is not exactly the same as what we state in 
Proposition \ref{prop:extrapolation_decoupling_dyadic_martingales} below) in \cite{CoxVeraar:2011} or can be proved by more general results 
(for example, from  \cite{Geiss:1997}). For the convenience of the reader we include a proof for our setting.

\begin{proposition}
\label{prop:extrapolation_decoupling_dyadic_martingales}
Let $X$ be a Banach space and let $D_q(X)\in [0,\infty]$, $q\in (0,\infty)$,
be as defined in Definition~\ref{def:constants}. If there exists a $p\in (0,\infty)$
such that $D_p(X)<\infty$, then $D_q(X)<\infty$ for all $q\in (0,\infty)$.
\end{proposition}

\begin{proof}
Fix $N\in \{2,3,\ldots\}$. Let $(r_n)_{n=1}^N$, $(r'_n)_{n=1}^N$ be independent Rademacher 
sequences on a probability space $(\Omega,\calF,\P)$ and let $v_0\in X$ and
$h_n\colon \{-1,1\}^{n}\rightarrow X$, $n\in \{1,\ldots,N-1\}$, be given.
Define $v_n:\Omega \to X$ by
$v_n  := h_n(r_1,\ldots,r_n)$ for $n\in \{1,\ldots,N-1\}$, and $v_N := 0 \in X$ for notational convenience. Let
$f_0=g_0:=0\in X$, and 
\[
    f_n := \sum_{j=1}^{n} r_j v_{j-1} 
    \sptext{1}{and}{1}
    g_n := \sum_{j=1}^{n} r_j' v_{j-1}
\]
for $n\in \{1,2,\ldots,N\}$.
Fix $\lambda\in (0,\infty)$ and let $\delta, \beta \in  (0,\infty)$ satisfy
$\beta>1+\delta$. Define the stopping times $\mu,\nu,\sigma:\Omega\to \{1,\ldots,N+1\}$ by
 \begin{align*}
   \mu & := \min\{n\in \{1,2,\ldots,N\} \colon \n f_n \n_{X} > \lambda\},\\
   \nu & := \min\{n\in \{1,2,\ldots,N\} \colon \n f_n \n_{X} > \beta\lambda\},\\
\sigma & := \min\{n\in \{0,1,\ldots,N\} \colon 
		\max\{ \n g_n \n_{X}, \| v_n \|_X \}
		> \delta\lambda 
	      \}, 
  \end{align*}
with $\min\emptyset := N+1$. Define
$f^*_N,g^*_N,v_{N-1}^*\colon \Omega \rightarrow \R$ by
$f^*_N := \sup_{n\in \{1,\ldots,N\}} \| f_n \|_X$, 
$g^*_N := \sup_{n\in \{1,\ldots,N\}} \| g_n \|_X$, 
$v^*_{N-1} := \sup_{n\in \{0,\ldots,N-1\}} \| v_n \|_X$, 
and
\[  ^{\mu}\!f^{\nu\minsym \sigma}_N := 
    \sum_{j=2}^N r_{j}v_{j-1} 1_{\{ \mu < j \leq \nu \minsym \sigma \}}
    \sptext{1}{and}{1}
    ^{\mu}\!g^{\nu\minsym \sigma}_N 
   := \sum_{j=2}^N r'_j v_{j-1} 1_{\{ \mu < j \leq \nu \minsym \sigma \}}.
\] 
By definition of $D_p(X)$ it holds that 
 \begin{align}\label{eq:stopdecoup}
   \| ^{\mu}\!f^{\nu\minsym \sigma}_N \n_{L^p(X)}
  \leq D_p(X) \n ^{\mu}\!g^{\nu\minsym \sigma}_N \n_{L^p(X)}. 
 \end{align}
On the set $\{ \sigma > \mu \}$ we have
\begin{align*}
  \n ^{\mu}\!g_N^{\nu\minsym \sigma} \n_{X} 
    & = \n g_{\nu\minsym \sigma \minsym N} - g_{\mu} \n_{X}
      \leq \n v_{(\nu\minsym \sigma \minsym N) -1} \n_X  + \| g_{(\nu \minsym \sigma \minsym N)-1} \|_X 
      + \| g_{\mu} \|_X
    \leq 3\delta\lambda.
\end{align*}
As $^{\mu}\!g_N^{\nu\minsym \sigma} = 0$ on $\{ \sigma \leq \mu\}$, 
it follows that
\begin{equation}\label{eq:supgest}
  \| ^{\mu}\!g^{\nu\minsym \sigma}_{N}\|_{L^p(X)} 
    \leq 3 \delta \lambda \left[\P( \sigma > \mu)\right]^{\frac{1}{p}} 
    \leq 3 \delta \lambda \left[\P(f^*_N>\lambda)\right]^{\frac{1}{p}}.
\end{equation}
On the other hand, on the set $\{ \nu \leq N, \sigma = N+1\}$ we have
$\nu > 1$ because $\delta<\beta$ and thus on that set we have
\begin{equation}\label{eq:fest}
  \n ^{\mu}\!f_N^{\nu\minsym \sigma} \n_{X}  
    \geq \n f_{\nu} - r_{\mu} v_{\mu-1} - f_{\mu-1} \n_X 
    \geq \beta\lambda - \delta\lambda -\lambda.
\end{equation}
By \eqref{eq:fest}, Chebyshev's inequality, and estimates~~\eqref{eq:stopdecoup} and~\eqref{eq:supgest}
it follows that
\begin{equation}\label{eq:lambest}
\begin{aligned}
   \P( f^*_N > \beta \lambda, 
    g^*_N \maxsym v_{N-1}^* \leq \delta\lambda) 
&  =  \P(\nu\leq N, \sigma = N+1)\\
& \leq \P(\n ^{\mu}\!f_N^{\nu\minsym \sigma} \n_{X} 
	    \ge (\beta-\delta-1)\lambda) \\
& \leq (\beta-\delta-1)^{-p}\lambda^{-p} 
	    \n ^{\mu}\!f_N^{\nu\minsym \sigma} \n_{L^p(X)}^p\\
& \leq [D_p(X)]^p (\beta-\delta-1)^{-p}\lambda^{-p}
	    \n {}^{\mu}\!g^{\nu\minsym \sigma}_N \n_{L^p(X)}^p \\
& \leq [ 3\delta D_p(X)]^p (\beta-\delta-1)^{-p}
	    \P(f^*_N>\lambda).
\end{aligned}
\end{equation}
As $\lambda>0$ was arbitrary, it follows from~\eqref{eq:lambest}
that for all $q\in (0,\infty)$ it holds that 
\begin{equation}\label{eq:fpest}
\begin{aligned}
\E |f^*_N|^q
  & \leq 
    \beta^{q}
      [3\delta D_p(X)]^p 
      (\beta-\delta-1)^{-p} 
      \E |f^*_N|^q
    + \beta^{q}\delta^{-q} 
      \big[
	\E |g^*_N|^q + \E |v_{N-1}^*|^q
      \big].
\end{aligned}
\end{equation}
Fix $q\in (0,\infty)$, and set $\beta := 2$ and 
$ \delta := 2^{-1-\inv{p}-\frac{q}{p}}3^{-1}[D_p(X)]^{-1}.$
Note that $D_p(X)\geq 1$ as $X\not =\{0\}$, so that
$\delta\leq \frac{1}{2}$, 
$\beta-\delta-1\geq \frac{1}{2}$, and 
\begin{align*}
\beta^{q}
  [3\delta D_p(X)]^p 
  (\beta-\delta-1)^{-p} 
\leq \tinv{2}.
\end{align*}
By the L\'evy-Octaviani lemma
applied conditionally
(see e.g.~\cite[Corollary 1.1.1]{KwapienWoyczynski:1992})
we have
$
\max\{ \E|v_{N-1}^*|^q , \E|g^*_N|^q \}\leq 2\E\n g_N \n^q_{ X}.
$
By substituting this into~\eqref{eq:fpest} we obtain
   \[ E |f^*_N|^q \leq \tfrac{1}{2} \E |f^*_N|^q + 4 \beta^{q}\delta^{-q} \E \|g_N\|_X^q
   \sptext{1}{and hence}{1}
   E |f^*_N|^q \leq 8\cdot 2^{q}\delta^{-q} \E \|g_N\|_X^q. \]
Recalling the choice of $\delta$, we obtain
\[     8\cdot 2^{q}\delta^{-q}
   \le 8\cdot 2^{q} [2^{-1-\inv{p}-\frac{q}{p}}3^{-1}[D_p(X)]^{-1}]^{-q}
    =  8\cdot 2^{2q +\frac{q}{p}+\frac{q^2}{p}}3^q [D_p(X)]^q. \]
We conclude that $D_q(X) \leq 3\cdot 2^{\frac{3}{q}+2+\frac{1}{p}+\frac{q}{p}} D_p(X)$,
which completes the proof.
\end{proof}


\section{}

\begin{lemma}
\label{lemma:*-hull}
Assume a  metric space $(M,d)$ and a continuous map $*:M\times M \to M$ with  
$(x*y)*z=x*(y*z)$ for $x,y,z\in M$. Let $\emptyset \not = \calP\subseteq M$ and 
\[ \overline{\calP}^* := {\rm cl}_d(\{ x_1 * \cdots * x_L : x_1,\ldots,x_L\in \calP, L\in \N\}) \]
where the closure on the right side is taken with respect to $d$. Then one has  
$\overline{(\overline{\calP}^*)}^*= \overline{\calP}^*$ and 
$\overline{\calP}^*$ is the smallest $d$-closed set $\calQ$ with  
$\calQ\supseteq \calP$ and $\mu * \nu \in \calQ$ for all $\mu,\nu \in \calQ$.
\end{lemma}

\begin{proof}
The equality $\overline{(\overline{\calP}^*)}^*= \overline{\calP}^*$  follows from the continuity of $*$ and a standard diagonalization procedure. 
This also implies that $\mu * \nu \in \overline{\calP}^*$ for all $\mu,\nu \in \overline{\calP}^*$.
Now let us assume a set $\calQ$ as in the assertion. Then $x_1 * \cdots * x_L\in \calQ$ for all
$x_1,\ldots,x_L\in \calP$. As $\calQ$ is closed we deduce $ \overline{\calP}^*\subseteq\calQ$.
\end{proof}


{\bf Acknowledgments.}
The first author is supported by the research programme {\em VENI Vernieuwingsimpuls} with project 
number $ 639.031.549 $, which is financed by the Netherlands Organization for Scientific Research (NWO).
The second author is supported by the project {\em Stochastic Analysis and Nonlinear Partial 
Differential Equations, Interactions and Applications} of the Academy of Finland with project number 298641.
The authors wish to thank Mark Veraar, Peter Spreij and 
Stanis{\l}aw Kwapie{\'n} for useful comments. In 
particular, we thank 
Stanis{\l}aw Kwapie{\'n} for
Theorem ~\ref{thm:char_p-ext_Rademacher}.
The first author would also like to thank Lotte Meijer.


\def\polhk#1{\setbox0=\hbox{#1}{\ooalign{\hidewidth
  \lower1.5ex\hbox{`}\hidewidth\crcr\unhbox0}}}

\end{document}